\documentclass{article}
\usepackage[utf8]{inputenc}
\usepackage[T1]{fontenc}
\usepackage{lmodern}
\usepackage[a4paper]{geometry}
\usepackage[english]{babel}

\usepackage{fullpage}

\usepackage{amssymb,amsmath,mathtools,amsthm}
\usepackage{dsfont,color,psfrag}

\usepackage[all]{xy}

\newtheorem{theorem}{Theorem}
\newtheorem{prop-f}[theoreme]{Proposition}
\newtheorem{prop}[theorem]{Proposition}

\newtheorem{lemma}[theorem]{Lemma}
\newtheorem{definition}[theorem]{Definition}
\newtheorem{question}[theorem]{Question}
\newtheorem{claim}[theorem]{Claim}

\newcommand{\E}{\mathbb{E}}

\newcommand{\N}{\mathbb{N}}

\renewcommand{\P}{\mathbb{P}}
\newcommand{\Q}{\mathbb{Q}}
\newcommand{\R}{\mathbb{R}}

\newcommand{\cA}{\mathcal{A}}
\newcommand{\cB}{\mathcal{B}}
\newcommand{\cC}{\mathcal{C}}
\newcommand{\cD}{\mathcal{D}}

\newcommand{\cF}{\mathcal{F}}

\newcommand{\cL}{\mathcal{L}}

\newcommand{\cO}{\mathcal{O}}

\newcommand{\cR}{\mathcal{R}}
\newcommand{\cS}{\mathcal{S}}

\renewcommand{\1}{\mathds{1}}
\renewcommand{\d}{\text{d}}

\renewcommand{\epsilon}{\varepsilon}
\renewcommand{\phi}{\varphi}

\newcounter{numeroexo}

\newcommand*\interior[1]{\mathring{#1}}
\newcommand{\card}{\mbox{card}}

\newcommand{\back}{\text{Back}}
\newcommand{\forward}{\text{For}}
\newcommand{\cluster}{\overleftrightarrow C}

\newcommand{\config}{\overline\cS}
\newcommand{\tribuconfig}{\cF^{\config}}

\newcommand{\configinfo}{\widehat\cS}
\newcommand{\tribuconfiginfo}{\cF^{\configinfo}}
\newcommand{\good}{\text{Good\,}}
\newcommand{\graphe}{\text{Graph}}
\newcommand{\grapheoriente}{\overrightarrow{\text{Graph}}}
\newcommand{\cross}{\text{Cross}}
\newcommand{\nice}{\text{Nice}}
\newcommand{\crosslocal}{\text{CrossLocal}}

\newcommand{\dyn}{\mbox{\scriptsize{dyn}}}

\newcommand{\blur}{\text{blur}}
\newcommand{\blurred}{\text{blurred-region}}

\newcommand{\jb}[1]{\textcolor{red}{[ #1 ]}}

\begin{document}

\selectlanguage{english}

\title{\LARGE{Absence of percolation for infinite Poissonian systems of stopped paths}}
\author{\large{David Coupier \thanks{Intitut Mines T\'el\'ecom Nord Europe, \texttt{david.coupier@imt-nord-europe.fr}}, David Dereudre\thanks{University of Lille, LPP UMR 8524, \texttt{david.dereudre@univ-lille.fr}}, Jean-Baptiste Gouéré \thanks{IDP, Universit\'e de Tours, \texttt{jean-baptiste.gouere@univ-tours.fr}}}}

\date{\small{\today}}
\maketitle
\begin{abstract}

The state space of our model is the Euclidean space in dimension $d=2$. Simultaneously, from all points of a homogeneous Poisson point process, we let grow independent and identically distributed random continuum paths. Each path stops growing at time $t>0$ if it hits the trace of the other curves realized up until time $t$. Such dynamic is well-defined as long as the distribution of paths has a finite second moment at each time $t>0$. Letting the time runs until infinity so that each path reaches its stopping curve, we study the connected property of the graph formed by all stopped curves. Our main result states the absence of percolation in this graph, meaning that each cluster consists of a finite number of curves. The assumptions on the distribution of paths are very mild, with the main one being the so-called 'loop assumption' which ensures that finite clusters (necessarily containing a loop) occur with positive probability. The main issue in this model comes from the long-range dependence arising from long sequences of causalities in the hitting/stopping procedure. Most methods based on block approaches fail to effectively address the question of percolation in this setting.
\end{abstract}

\tableofcontents

\section{Introduction}

Our model belongs to the family of stopped germ-grain models in $\R^2$. In this context, a point (germ) from a point process in $\R^2$ initiates the growth of a random set (grain), which stops growing when it encounters another grain. These models naturally arise in material science and biology, for example, to model impurities or cracks in materials or the propagation of infection in a cellular tissue. They are also a significant class of models in stochastic geometry, presenting interesting and challenging percolation issues. The first example of such a model, introduced by H{\"a}ggstr{\"o}m and Meester in 1996, is the lilypond model \cite{HM}, where the germs are given by an homogeneous Poisson point process and the grains are discs with radii that increase over time with constant velocity. They proved that the clusters of stopped discs do not percolate. The main conjecture, which emerged from Question 6.1 in \cite{HM}, can be rephrased as follows: for a large class of distributions of  grains, the associated stopped germ-grain model does not percolate. Since then, the absence of percolation has been investigated for many germ-grain models in which grains are mainly convex bodies \cite{EL,HL} or line-segments that can be unilateral or bilateral, with constant or i.i.d.\ (even unbounded) velocities, according to isotropic directions or not \cite{CDLT,CDlS-unbounded,DEL,Hirsch}.

We address this problem in the present article. The primary difficulties encountered in this class of models arise from the long-range dependence of the connection graph and the lack of monotony or usefull correlation inequalities (such as FKG  inequalities). 

\subsection*{Our model}

In our case, the point process $\xi$ is an homogeneous Poisson point process in $\mathbb{R}^2$ with intensity one (other intensities could be considered by a simple rescaling procedure). The space of grains $\mathcal{C}$ is the space of continuous functions $h : \mathbb{R}_+ \to \mathbb{R}^2$ such that $h(0) = 0$, and we denote by $\mu$ a probability measure on $\mathcal{C}$. Each point $x$ in $\xi$ is then equipped with a random path $h_x$ with distribution $\mu$. We assume that the paths $(h_x)_{x \in \xi}$ are independent of each other and independent of $\xi$. This description is rigorously defined by a marked Poisson point processes $\overline\xi$ on $\R^2\times \mathcal C$ with intensity $dx\otimes\mu$.

Now, from each point $x \in \xi$, we let the path $x+h_x$ starts to grow, and it stops growing at time $t > 0$ when it hits the curves of other grains produced until time $t$. If a grain does not meet any other grain, it continues to grow indefinitely. The existence of this infinite-dimensional dynamics is not obvious, but it is possible to prove its existence as soon as $\mathbb{E}(\sup_{0\le s\le t } \|h_x(s)\|^2) < +\infty$ for all positive time $t$ (see Proposition \ref{p:existence}). 

Denote by $\tau(x)$ the stopping time of the grain $x \in \xi$ (which is infinite if $x$ never stops growing).
The main question is whether the random set 
$$
\bigcup_{x \in \xi} \Big\{x+h_x(t), 0 \le t \le \tau(x) \Big\}
$$
percolates; i.e., whether some of the connected components of this random set consist of an infinite number of curves. The main conjecture since the paper by H{\"a}ggstr{\"o}m and Meester mentioned above is that, under reasonable assumptions, \textbf{percolation should not occur}. Let us be more specific about those reasonable assumptions. Consider the graph of connections $\grapheoriente(\overline\xi)$ with vertex set $\xi$ where we put an edge $x\to y$ from $x$ to $y$ if the path from $x$ hits the curve from $y$. If loops are not possible, then the connected components of this graph are infinite.
Therefore the possibility of existence of loops is a necessary condition for non percolation. Because of the continuous nature of our model, if one seeks for a simple general sufficient condition, it is natural to assume that loops arise at arbitrary small space and time scales. We thus formulate the following question.

\begin{question} \label{q} Do  the following assumptions imply absence of percolation?
\begin{itemize}
\item The model is well defined. (We give a precise meaning to this assumption in Section \ref{s:model_driven_by_mu} with the notion of tempered measure).
\item The "loop-assumption" holds at any small spatial and temporal scale $\alpha>0$. It means that for any $\alpha>0$, with positive probability, the points (and the paths) in the ball 
$B(0,\alpha)$ produce before time $\alpha>0$ a loop inside the ball that disconnects the center $0$ from the outside of the ball. See Figure \ref{f:loop_intro}
(We refer to Section \ref{s:absence-percolation} for a precise definition.)
\end{itemize}
\end{question}
\begin{figure}[!ht] \label{f:loop_intro}
\begin{center}
\includegraphics[width=5cm,height=5cm]{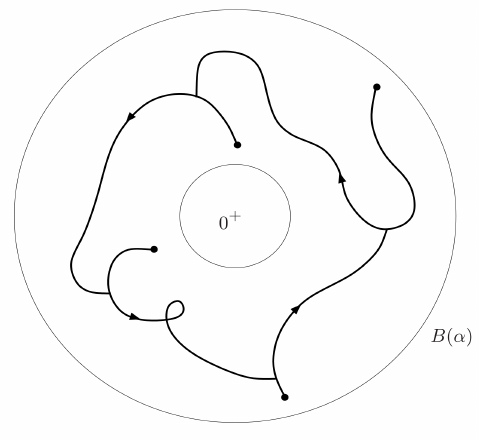}
\caption{In this picture, the loop is made up of $n=4$ points (the black points). The union of their curves separates $0$ from the outside of $B(\alpha)$. The black arrows indicate the direction in which the grains grow.}
\end{center}
\end{figure}

Let us mention that it is easy to see that some models lack the loop condition and therefore exhibit percolation. For instance, if the paths go in a straight line to the right or to the top with equal probability $1/2$, the loops are not possible. See also the navigation models studied in \cite{BB,BM} which are directed according to a given deterministic vector.

\subsection*{Results}

Our main result is a positive answer to Question \ref{q} under the following mild extra assumption on the distribution of paths $\mu$: there exists $t > 0$ such that $\E(\sup_{0\le s\le t } \|h_x(s)\|^2) < +\infty$. See Theorem \ref{t}.

%More marginally, note that we assume the second moment $\mathbb{E}(\sup_{0\le s\le t } \|h_x(s)\|^2) < +\infty$ only for a specific value of $t > 0$, not for all $t > 0$ as in the existence %Proposition \ref{p:existence} mentioned above. 

\subsection*{Comparison with models in the literature.}

As mentioned above, the first stopped germ-grain model in the literature is the lilypond model \cite{Daley-Last-05,DSS,HM}, where the grains are balls. In this case, the grains are thick, and achieving the locality of the interaction is easier, as the probability of a distant grain hitting another grain decreases exponentially with the distance. See also generalizations to convex grains \cite{EL,HL}. Studying thin grains, like curves, poses more challenges. 
The question has been raised in \cite{DEL} where line-segment models are introduced. 
The first non percolation result for this kind of model has been obtained in \cite{Hirsch} in a non isotropic setting.
Here, paths move with constant velocity in straight lines to the right, left, top, or down with equal probability $1/4$. Similar to the ball model, percolation does not occur, but proving this is significantly more challenging. The proof relies on a block method and a sprinkling procedure, and it heavily uses the fact that directions are limited to the four cardinal points and that the velocity is bounded. The absence of percolation for the general isotropic line-segment model with unbounded velocity $V$ have been proved later in \cite{CDLT,CDlS-unbounded} with a super exponential moment assumption $\E(e^{V^s})<+\infty$, $s>1$. 
Note that, as a corollary of our main theorem, we relax the super exponential moment in the line-segment model since we require only a finite second moment $\E(V^2)<+\infty$ (Proposition \ref{p:LineSegment}).

All strategies developed prior to this paper relied on paths moving in straight lines with constant or strongly integrable velocity. These settings enable the proof of good upper bounds on the probability of crossing a large box by introducing several obstacles to block any line. This leads to useful localization properties. Dealing with less stringent assumptions (whether on speed or on the paths of the grains) is more challenging and requires new approaches. Regarding paths, it appears that regularity beyond mere continuity is irrelevant. For example, rough curves (as Brownian curves, see Section \ref{sect:BrownianModel}) or space-filling curves (as Peano curves) are permitted. This allows for unusual hitting behavior, including simultaneous or coincident hits.

To sum up, in the present paper we greatly generalize the non percolation results for germ-grain models with thin grains in two main directions: on the growth of grains since a second moment condition is assumed instead of a super exponential one, and on the shape of grains since only the continuity of curves is assumed instead of line-segments.

This is achieved through two main ingredients.
\begin{itemize}
\item We develop a quantitative multi-scale analysis to control the dependencies in a related technical model.
\item As it not possible to localize in a useful way the properties of interest, we develop a notion of  scenarios in which what happens in some bounded window is only known up
to a finite number of locally defined scenarios.
\end{itemize}
We give a more detailed sketch in the next part.

\subsection*{Ingredients of the proof.}

Putting an edge between $x$ and $y$ when $x$ hits $y$ or $y$ hits $x$ defines an undirected graph with vertex set $\xi$, denoted by $\graphe(\overline\xi)$, for which we aim to state the absence of percolation (Theorem \ref{t}). Putting an arrow $x \to y$ from $x$ to $y$ when $x$ hits $y$ defines an out-degree one graph (up to adequat and irrelevant definitions about vertices hitting zero or more than two vertices) denoted by $\grapheoriente(\overline\xi)$. The forward set $\forward(x)$ of a vertex $x$ is defined as the set of all vertices $y$ such that there exists an oriented path from $x$ to $y$ in $\grapheoriente(\overline\xi)$. Soft arguments relying only on stationarity and the out-degree one property allows to reduce the proof of Theorem \ref{t} to that of the absence of infinite forward sets in $\grapheoriente(\overline\xi)$. Such arguments, sometimes quoted as \textit{mass transport principle}, were already present in \cite{DEL,Daley-Last-05}.

Our global strategy to prove the absence of infinite forward sets comes from \cite{CDLT,CDlS-unbounded} and can be explained in rough terms as follows. Consider by absurd an infinite forward set. When one travels along this forward set, one should encounter opportunities to add to the current configuration a suitable (small) loop that would break the considered forward set, making it finite and contradicting our assumption. The crux of the matter consists in proving that the number of these opportunities is infinite but the non-local character of the model makes this task tricky. Indeed, when one travels along a forward set, one gathers more and more information and it is difficult to rule out the fact that such opportunities could become increasingly rare. Moreover, this difficulty is compounded by the long range dependence in the model: modifying a grain arbitrarily far away from a given grain $x \in \xi$ can change the lifetime of $x$ by a domino effect. In particular, these strong dependencies forbid any useful stabilizing definition for $\tau(x)$ (or other relevant quantities) and therefore prevent the use of the associated tools from stochastic geometry.
%Pas sûr de la pertinence ici: see the works \cite{LRSY,PY} based on stabilization or exponential-stabilization.

Let us be more concrete about opportunities to add loops. We refer to Figure \ref{fig:good} for an illustration. Consider a point $x_0$. One looks for some $\alpha>0$ such that we can place a ball $B(v,2\alpha)$ \textit{just before} the impact of $x_0$ on its stopping grain $s(x_0)$ and overlapping no other grains. Henceforth, a loop suitably placed inside $B(v,2\alpha)$ will stop the grain $x_0$ without reducing its backward set $\back(x)$ (which is the set of vertices $y$ whose the forward set contains $x$). This idea was already present in \cite{CDLT,CDlS-unbounded}. Note that it is important not to alter the backward set: imagine that we are exploring a forward set and that, at some point, we arrive at $x_0$ where we want to create (say by a modification argument) a loop after $x_0$ in order to break the forward set. However if this creation modifies the backward of $x_0$ then the forward set we were exploring could by-pass $x_0$ and continues forever.

\begin{figure}[!ht]
\begin{center}
\includegraphics[width=7.5cm,height=4.2cm]{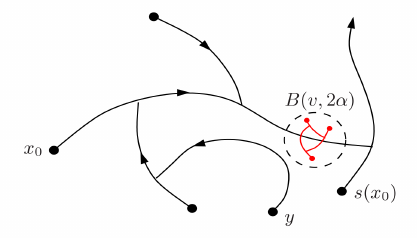}
\caption{\label{fig:good} Grains of $\overline\xi$, especially $x_0$ and its stopping grain $s(x_0)$, are represented in black. Black arrows indicate the sense in which these grains grow. A loop built with three new grains (in red) and located inside the small ball $B(v,2\alpha)$ allows to stop the grain $x_0$ without reducing its backward set. Note that if $B(v,2\alpha)$ overlapped the grain $y$, the red loop could stop it and then would reduce the backward set of $x_0$ (to which $y$ belongs).}
\end{center}
\end{figure}

Say that a vertex $x_0 \in \xi$ is $\alpha$-good if we can place a ball $B(v,2\alpha)$ as described above. Otherwise say that $x_0$ is $\alpha$-bad. Still by soft arguments, we can reduce our task to proving that, for small enough $\alpha>0$, there exists no infinite forward sets composed only with $\alpha$-bad points-- and we now have a parameter $\alpha>0$ to play with. Besides, determining a suitable location for the small ball $B(v,2\alpha)$ just before the impact of $x_0$ on its successor $s(x_0)$ requires to know the local picture around the grain created by $x_0$, especially who is its successor and the lifetime at which $x_0$ ceases growing, which are clearly non-local informations. This is a crucial obstacle in establishing the non existence of infinite forward sets of $\alpha$-bad points. This difficulty were overcome in \cite{CDlS-unbounded} by the introduction of an additional technical assumption called "shield" whose fulfillnes was only proved under strong integrability assumptions. As we aim at reducing such extra technical conditions, we need a new approach.

To overcome this difficulty, we develop in this article a way to witness $\alpha$-goodness without revealing too much of $\overline\xi$. There are two layers in our approach.

\begin{itemize}
\item First, \textbf{we introduce an augmented model, called the $\sharp$-model}, in which the augmented grain $G^\sharp(x) := x+h_x([0,\tau^\sharp(x)])$ contains the grain $x+h_x([0,\tau(x)])$ of the original model. One crucial point is that the dependencies are less strong in this model. Under our second moment assumption, the $\sharp$-model is amenable to some \textbf{quantitative multi-scale analysis} which enables us to prove that the $G^\sharp(x)$'s are not too large. See Proposition \ref{p:33} for a formal statement. The aforementioned multi-scale analysis provides a lower bound in the passage time in some kind of first-percolation problem associated with a subcritical Boolean model (see Proposition \ref{p:lestnsontsympas}) which could be of independent interest.
\item There is not enough information in the $\sharp$-neighbourhood of $x$ to determine exactly what happens around the grain $x$ (two vertices $x$ and $y$ are $\sharp$-neighbours when, roughly, $G^\sharp(x)$ and $G^\sharp(y)$ overlap). Therefore, in order to gain locality we have no choice but to lose information. \textbf{We achieve to determine what really happens around $x$ up to a finite number of scenarios} and we declare $x$ as $(\alpha,\sharp)$-good if it is $\alpha$-good in each of these scenarios. See Proposition \ref{p:DieseGoodImplyGood} for a formal statement. An advantage of the concept of $(\alpha,\sharp)$-goodness is locality: the fact that a vertex $x$ is $(\alpha,\sharp)$-good only depends on its $\sharp$-neighbours, and their $\sharp$-neighbours, and so on, up to four levels deep. Above all, $(\alpha,\sharp)$-goodness witnesses $\alpha$-goodness, i.e. $(\alpha,\sharp)$-goodness of a vertex $x$ implies $\alpha$-goodness of $x$.
\end{itemize}

Once the key Propositions \ref{p:33} and \ref{p:DieseGoodImplyGood} are proved, a multi-scale approach (initially introduced in \cite{G-AOP08}) can be used to prove that for small enough $\alpha>0$ there exists no infinite forward sets of $(\alpha,\sharp)$-bad points. Therefore, there exists no infinite forward sets of $\alpha$-bad points. As mentioned above, this is sufficient to ensure absence of percolation in the model.

\subsection*{Plan of the paper.}

The paper is organized as follows. In the next Section \ref{s2}, we introduce the model and present our main result, Theorem \ref{t}. Additionally, two main examples are provided: the line-segment model and the Brownian model. The general scheme of the proof of the theorem is given in Section \ref{s:proof}, which relies on two major Propositions, \ref{p:33} and \ref{p:DieseGoodImplyGood}, whose proofs are postponed to Sections \ref{s_prop17} and \ref{sect:ProofDieseGood}, respectively.

\section{Model and main results}\label{s2}

\subsection{Deterministic model}
\label{sect:Model}

\subsubsection{The space of configurations}

Let $\cC$ be the space of continuous functions $h : \R_+ \to \R^2$ such that $h(0)=0$. We equip $\cC$ with the product $\sigma$-algebra $\cF^{\cC}$. Let $\overline\pi :\R^2 \times \cC \to \R^2$ denotes the projection onto $\R^2$. The state space is
\[
\config = \{ \overline S \subset \R^2 \times \cC : \overline S \text{ is finite or countable and the restriction of } \overline\pi \text{ to } \overline S \text{ is one-to-one}\}.
\]
Let $\overline S \in \config$. We write $S = \overline\pi(\overline S)$. As $\overline\pi_{|\overline S}$ is one-to-one we can write
\[
\overline S = \{ (x,h_x), x \in S\}
\]
where $(h_x)_{x \in S}$ is a family of continous functions from $\R_+$ to $\R^2$ which vanish at $0$. Let $x \in S$. We define a continuous function $g_x:\R_+ \to \R^2$ by
\[
g_x(t)  = x + h_x(t).
\]
For all $t \ge 0$ we also set
\[
G_x(t) = g_x([0,t]).
\]
When considering some configuration $\overline S$ in the sequel, we freely use the notations $S, h_x, g_x$ and $G_x$ without introducing them.
For a  configuration $\overline S$ and a set $B  \in \cB(\R^2)$, we use the following short notation for the restriction of $\overline S$ to $B \times \cC$:
\[
\overline S_{|B} := \overline S \cap \big(B \times \cC).
\]

Write $\#X$ for the cardinality of a set $X$.
For any $A$ in $\cB(\R^2) \otimes \cF^{\cC}$, define a map $N_A$ from $\config$ to $\N \cup \{\infty\}$ by 
\[
N_A(\overline S) = \#(\overline S \cap A).
\]
We equip $\config$ with the $\sigma-$field $\tribuconfig$ generated by the family of maps $N_A, A \in \cB(\R^2) \otimes \cF^{\cC}$.

\subsubsection{The static point of view}\label{sdefstat}

Fix $\overline S \in \config$. In the introduction we gave an informal dynamical description of the model which will be formalized in Section \ref{sect:DynamicView}. It appears that a static point of view -- which focuses on the properties of the family of times at which each grain stops growing -- is more appropriate to formulate a robust definition. This approach was already used to define the Lilypond model in \cite{Daley-Last-05}. Here we follow the definition of \cite{CDlS-unbounded} with some adaptations. 

Let $\overline S  \in \config$ be a configuration. Let $\tau:S \to (0,+\infty]$. For any $x \in S$, we think about $\tau(x)$ as the time at which the grain $x$ stops growing (if $\tau(x)=\infty$, then the growth never stops). Before stating the properties we require about $\tau$, let us define a useful notion. For $x \neq y \in S$ and $t_x \ge 0$, we say that \textbf{$x$ $\tau$-hits $y$ at time $t_x$} if there exists $t_y \ge 0$ such that 
\[
t_x \le \tau(x) \text{ and } t_y \le \tau(y) \text{ and } g_x(t_x)=g_y(t_y) \text{ and } t_y \le t_x ~.
\]
The property in display can be rephrased informally as follows. 
\begin{itemize}
\item The grain $y$ is still growing at time $t_y$ (as $t_y \le \tau(y)$) and thus the point $g_y(t_y)$ belongs to the grain $y$ at time $t_y$.
\item The point $g_y(t_y)$ thus also belongs to the grain $y$ at time $t_x$ (as $t_x \ge t_y$).
\item But the grain $x$ is still growing at time $t_x$ (as $t_x \le \tau(x)$).
\item Therefore the grain $x$ hits the grain $y$ at time $t_x$ (as moreover $g_x(t_x)=g_y(t_y)$).
\end{itemize} 
Note that when \textbf{$x$ $\tau$-hits $y$ at time $t_x$}, it  does not mean that it is for the first time. We are now ready to give the property we require regarding $\tau$, meaning the first time of hitting. A \textbf{$\overline S$-lifetime function} is a function $\tau : S \to (0,+\infty]$ such that:
\begin{enumerate}
\item \textbf{Stopping property.} For all $x \neq y \in S$ and $t_x \ge 0$, if $x$ $\tau$-hits $y$ at time $t_x$, then $\tau(x)=t_x$.
\item \textbf{Hitting property.} For all $x \in S$ such that $\tau(x)<\infty$, there exists $y \in S \setminus \{x\}$ such that $x$ $\tau$-hits $y$ at time $\tau(x)$.
\end{enumerate}
Very roughly speaking, the Stopping and Hitting properties can be respectively understood as follows: ``a hit implies a stop'' and ``a stop requires a hit''. From now, a function $\tau:S \to (0,+\infty]$ will always denotes a \textbf{$\overline S$-lifetime function} with the stopping and hitting properties. 

As a first consequence of the Stopping property, the Hard-core property asserts that only the extremity $g_x(\tau(x))$ of a grain may hit another grain:

\begin{lemma}[Hard-core property]
\label{lem:hard-core}
For all $x \neq y \in S$, $g_x([0,\tau(x))) \cap g_y([0,\tau(y))) = \emptyset$.
\end{lemma}

\begin{proof}
Let $x \neq y \in S$ and $t_x, t_y \ge 0$ be such that $t_x \le \tau(x)$, $t_y \le \tau(y)$ and $g_x(t_x)=g_y(t_y)$. If $t_y \le t_x$ then $x$ $\tau$-hits $y$ at time $t_x$ and thus the Stopping property implies $t_x=\tau(x) \not\in [0,\tau(x))$. Otherwise $y$ $\tau$-hits $x$ at time $t_y$ and the Stopping property implies $t_y=\tau(y) \not \in [0,\tau(y))$. To sum up, we got $t_x \not\in [0,\tau(x))$ or $t_y \not\in [0,\tau(y))$ which is the searched result.
\end{proof}

Note that if $x$ $\tau$-hits $y \in S \setminus \{x\}$ it necessarily happens at time $\tau(x)$ by the Stopping property. Therefore, $x$ $\tau$-hits $y$ is equivalent to
\begin{equation}
\label{e:x-stopped-by-y}
\tau(x) < \infty \text{ and there exists } t_y \ge 0 \text{ such that } t_y \le \tau(y) \text{ and } g_x(\tau(x)) = g_y(t_y) \text{ and } t_y \le \tau(x) ~.
\end{equation}

In our model, the hits on itself are ignored so that a grain may overlap itself infinitely many times without being stopped (think of a Brownian trajectory). We say that \textbf{$x$ is stopped at time $\tau(x)$} if $\tau(x) < \infty$ is finite. In that case, the Hitting property means that $x$ $\tau$-hits at least one $y \in S \setminus\{x\}$, but nothing prevents $g_x(\tau(x))$ to belong to several grains.

We say that \textbf{$x$ $\tau$-hits $y$ in a regular way} if 
\begin{equation}
\label{e:x-stopped-by-y-in-a-regular-way}
\tau(x) < \infty \text{ and there exists } t_y \ge 0 \text{ such that } t_y < \tau(y) \text{ and } g_x(\tau(x)) = g_y(t_y) \text{ and } t_y \le \tau(x) ~.
\end{equation}
This is a stronger than \eqref{e:x-stopped-by-y} as we require $t_y < \tau(y)$. It is important to note that, by the Hard-core property (Lemma \ref{lem:hard-core}), if $x$ $\tau$-hits $y$ in a regular way then there exists a unique such $y$.\\ 

A configuration $\overline S$ is said \textbf{tempered} if there exists a unique $\overline S$-lifetime function $\tau$. In that case, the model is well defined (from the configuration $\overline S$). In the sequel, we will work with tempered configurations and shorten ``$x$ $\tau$-hits $y$'' to ``$x$ hits $y$''.

\subsubsection{The dynamic point of view}
\label{sect:DynamicView}

Let us present a procedure, called the \textbf{Dynamical algorithm}, corresponding to the dynamical description of the model given in introduction. The intuition is the following: for any $x \in S$ and any time $t \ge 0$, $G_x(t)$ is the state of the grain $x$ at time $t$ if its growth has not been stopped yet; its growth stops as soon as it hits another grain. From a \textit{finite} configuration $\overline S$, the Dynamical algorithm returns a function $\tau_{\dyn}:S \to (0,+\infty]$ such that $\tau_{\dyn}(x)$ is the time at which the grain $x$ is stopped and $\tau_{\dyn}(x) = \infty$ if it is not stopped. The Dynamical algorithm consists of a While loop which involves the set $S'\subset S$ of grains still alive (or growing) at each step of the loop.\\

\noindent
\hspace*{0.5cm} For any $x\in S$, let $\tau_{\dyn}(x) = +\infty$ and $g'_x([0,t]) = g_x([0,t])$ for any $t \geq 0$.\\
\hspace*{0.5cm} Let $S' = S$.\\
\hspace*{0.5cm} \textbf{While} $S'\not= \emptyset$:\\
\hspace*{1cm} For any $x \not= y$, $x \in S'$ and $y \in S$, let $t(x\to y) = \inf \{ t \geq 0 : \, g_x(t) \in g_y'([0,t]) \}$.\\
\hspace*{1cm} Let $t^{\ast} = \inf \{ t(x\to y) : x \not= y, x \in S' \mbox{ and } y \in S \}$.\\
\hspace*{1cm} \textbf{If} $t^{\ast} = +\infty$ \textbf{Then} $S' = \emptyset$. \textbf{Else}:\\
\hspace*{1.5cm} Let $S''=\{ x \in S' : \exists y \in S \setminus \{x\} , t(x\to y) = t^{\ast} \}$.\\
\hspace*{1.5cm} For any $x \in S''$, let $\tau_{\dyn}(x) = t^{\ast}$ and $g'_x([0,t]) = g'_x([0,t^{\ast}])$ for any $t \geq t^{\ast}$.\\
\hspace*{1.5cm} Let $S' = S' \setminus S''$.\\
\hspace*{1cm} \textbf{End If}\\
\hspace*{0.5cm} \textbf{End While}\\

Let us specify that the infimum $t(x\to y)$, for $x\not= y$, is reached since the $g_x$'s are continuous functions. The same holds for the infimum $t^{\ast}$ because the configuration $\overline S$ is finite. Also the set $S''$ is allowed, at a given step of the While loop, to contain several elements: this corresponds to simultaneous hits between different grains. Finally, each auxiliary function $g'_x$ is possibly \textit{frozen} during the algorithm. This ensures that the piece of trajectory of $(g_x(t))_{t\geq 0}$ after its stopping time will play no role in the sequel.\\

Let us prove that $\tau_{\dyn}$ is a $\overline S$-lifetime function and this is the only one. This means that any finite configuration $\overline S$ is tempered and both definitions of the model-- the static one with lifetime functions and the dynamical one --coincide on each finite $\overline S$.

\begin{lemma}[Reconciliation Lemma]\label{l_rec}
Let $\overline S$ be a finite configuration. The function $\tau_{\dyn}$ generated by the dynamical algorithm from $\overline S$ is the only $\overline S$-lifetime function.
\end{lemma}

\begin{proof}
Let $\overline S$ be a finite configuration. Let us prove that $\tau_{\dyn}$ is a $\overline S$-lifetime function and first focus on the Stopping property. For all grains $x\not= y$, one has
\[
\tau_{\dyn}(x) \leq t(x\to y) \in [0,+\infty]
\]
where $t(x\to y)=\inf\{ t\geq 0 : g_x(t)\in g_y'([0,t])\}$ and $g_y'([0,t]) = g_y([0,t \wedge \tau_{\dyn}(y)])$. Assume that $x$ $\tau_{\dyn}$-hits $y$ at time $t_x$ for $x\not= y$ in $S$ and $t_x\geq 0$, i.e. there exists $t_y$ such that $g_x(t_x)=g_y(t_y)$ with $t_x\leq \tau_{\dyn}(x)$, $t_y\leq \tau_{\dyn}(y)$ and $t_y\leq t_x$. Then
\[
g_x(t_x) = g_y(t_y) \subset g_y([0,t_x \wedge \tau_{\dyn}(y)]) = g'_y([0,t_x])
\]
which means $t_x \geq t(x\to y)$. Hence we get $\tau_{\dyn}(x) \leq t(x\to y) \leq t_x \leq \tau_{\dyn}(x)$, i.e. $t_x = \tau_{\dyn}(x)$. The Stopping property holds.

In order to check the Hitting property, let us consider a grain $x \in S$ with $\tau_{\dyn}(x)<\infty$. By the Dynamical algorithm, there exists a grain $y \in S\!\setminus\!\{x\}$ such that $\tau_{\dyn}(x)=t(x\to y)<\infty$. We then have to prove that $x$ $\tau_{\dyn}$-hits $y$ at time $\tau_{\dyn}(x)$. Since the $g_x$'s are continuous, the infimum $t(x\to y)$ is reached :
\[
g_x(\tau_{\dyn}(x)) \in g'_y([0,\tau_{\dyn}(x)]) = g_y([0,\tau_{\dyn}(x) \wedge \tau_{\dyn}(y)]) ~.
\]
So there exists $t_y \leq \tau_{\dyn}(x) \wedge \tau_{\dyn}(y)$ satisfying $t_y \leq \tau_{\dyn}(y)$, $g_x(\tau_{\dyn}(x)) = g_y(t_y)$ and $t_y \leq \tau_{\dyn}(x)$. In other words, $x$ $\tau_{\dyn}$-hits $y$ at time $\tau_{\dyn}(x)$.\\

It remains to prove that $\tau_{\dyn}$ is the only $\overline S$-lifetime function. By absurd, assume that there exist two $\overline S$-lifetime function, say $\tau_1$ and $\tau_2$. Thus let us consider the grain $x$ minimizing $\tau_1(x)\wedge \tau_2(x)$ among grains satisfying $\tau_1(x)\not= \tau_2(x)$. This minimal element is well defined since $\tau_1\not= \tau_2$ and $\overline S$ is finite. Without loss of generality, we can assume $\tau_1(x)<\tau_2(x)$ which in particular implies that $\tau_1(x)<\infty$. By the Hitting property (for $\tau_1$), there exists a grain $y\not= x$ such that $x$ $\tau_1$-hits $y$ at time $\tau_1(x)$. Roughly speaking, the grain $x$ is stopped by $y$ for $\tau_1$ and it lives a little longer for $\tau_2$ than for $\tau_1$. Then the Hitting property (for $\tau_2$) forces the grain $y$ to live a little shorter for $\tau_2$ than for $\tau_1$. This contradicts the minimality of $x$. Precisely, $x$ $\tau_1$-hits $y$ at time $\tau_1(x)$ means $g_x(\tau_1(x))=g_y(t_y)$ for some time $t_y$ such that $t_y\leq\tau_1(y)$ and $t_y\leq\tau_1(x)$. The hypothesis $t_y\leq\tau_2(y)$ together with $g_x(\tau_1(x))=g_y(t_y)$, $\tau_1(x)\leq\tau_2(x)$ and $t_y\leq\tau_1(x)$ would mean that $x$ $\tau_2$-hits $y$ at time $\tau_1(x)$ and then $\tau_1(x) = \tau_2(x)$ by the Hitting property (for $\tau_2$). As this conclusion fails, we necessarily have $t_y > \tau_2(y)$. To sum up,
\[
\tau_2(y) < t_y \leq \tau_1(x) \wedge \tau_1(y) < \tau_2(x) ~.
\]
The grain $y$ then satisfies $\tau_1(y)\not= \tau_2(y)$ and $\tau_1(y)\wedge \tau_2(y) < \tau_1(x)\wedge \tau_2(x)$ which contradicts the minimality of $x$.
\end{proof}

\subsection{Model driven by a probability measure $\mu$}

\label{s:model_driven_by_mu}

We need to introduce a few objects to handle some measurability issues about the families of stopping times $\tau(x)$. Consider the set
\[
\configinfo = \{ \widehat S \subset \R^2 \times [0,+\infty] : 
\widehat S \text{ is finite or countable and the restriction of } \widehat\pi \text{ to } \widehat S
\text{ is one-to-one}\}
\]
where $\widehat\pi$ denotes the projection from $\R^2 \times  [0,+\infty]$ onto $\R^2$.
As before, if $\widehat S \in \configinfo$, we can write
\[
\widehat S = \{ (x,\tau(x)), x \in S\}
\]
where $S=\widehat\pi(\widehat S)$. For any $A$ in $\cB(\R^2) \otimes  \cB([0,+\infty])$, define a map $\widehat N_A$ from $\configinfo$ to $\N \cup \{\infty\}$ by 
\[
\widehat N_A(\widehat S) = \#(\widehat S \cap A).
\]
We equip $\configinfo$ with the $\sigma-$field $\tribuconfiginfo$ generated by the family of maps $\widehat N_A, A \in \cB(\R^2)  \otimes \cB([0,+\infty])$.

Let $\overline \cA \subset \config$. We say that $\overline \cA$ is a \textbf{tempered set of configurations} if the following conditions hold:
\begin{enumerate}
\item $\overline\cA$ is measurable, that is $\overline \cA \in \cF^{\overline\cS}$.
\item For all $\overline S \in \overline \cA$, $\overline S$ is a tempered configuration.
\item The set $\overline \cA$ is invariant under the action of spatial translations. 
In other words, for all $\overline S \in \overline \cA$ and all $u \in \R^2$,
\[
\overline S - u \in \overline\cA
\]
where $\overline S-u= \{(x-u, h), (x,h) \in \overline S\}$.
\item The map $\phi : \overline \cA \to \configinfo$ defined by
\[
\phi(\overline S) = \{(x,\tau(x ; \overline S)), x \in S\}
\]
is measurable where $\overline \cA$ is equipped with the $\sigma$-field induced by $\tribuconfig$ and $\configinfo$ is equipped with the $\sigma$-field $\tribuconfiginfo$. Here $\tau(\cdot, \overline S)$ denotes the unique $\overline S$-lifetime.
\end{enumerate}
Note that translating by a common vector the starting points of the grains does not modify the lifetime of the grains. More formally, if $\overline \cA$ is a tempered set of configuration, then for any $\overline S \in \overline \cA$ and any $u \in \R^2$ we have
\[
\phi(\overline S - u) = \phi(\overline S) - u
\]
where $\phi(\overline S) - u = \{(x-u, \tau), (x,\tau) \in \phi(\overline S)\}$.\\

Let $\mu$ be a probability measure on $(\cC,\cF^{\cC})$. Let $\overline\xi$ be a Poisson point process on $\R^2 \times \cC$ with intensity measure $\d x \otimes \mu$ where $\d x$ denotes the Lebesgue measure on $\R^2$. Equivalently, $\overline \xi$ is a marked Poisson point process on $\R^2$ with intensity $\d x$ and independent marks in $\cC$ with distribution $\mu$. We refer to the book by Last and Penrose \cite{Last-Penrose-livre} for background on Poisson point processes. The projection $\overline \pi$ from $\R^2 \times \cC$ onto $\R^2$ is almost surely one-to-one on $\overline \xi$. In the remaining of this article we assume that this condition holds. We can thus write, as before,
\[
\overline\xi = \{(x,h_x), x \in \xi\}
\]
where $\xi = \pi(\overline\xi)$ and $(h_x)_x$ is a family of elements of $\cC$. As before, we use the associated notations $g_x$ and $G_x$. The process $\xi$ is a Poisson point process on $\R^2$ with intensity $\d x$. Conditioned on $\xi$, $(h_x)_x$ is a family of i.i.d.r.v.\ with distribution $\mu$.

We say that \textbf{the measure $\mu$ is tempered} if there exists a tempered set of configurations $\overline \cA$ such that $\overline\xi \in \overline \cA$ with probability $1$. In this case, we always work on a full probability event such that $\overline\xi \in \overline\cA$. Moreover we use as before the notation $\tau(x ; \overline \xi)$ -- or simply $\tau(x)$ -- for $x \in \xi$.

In the article we focus on percolation questions. We however provide a simple sufficient condition (\ref{SuffTempered}) which ensures that a given measure $\mu$ is tempered. In other words, under (\ref{SuffTempered}), the model built from the marked Poisson point process $\overline \xi$ (and driven by the probability measure $\mu$) is well defined.

\begin{prop}
\label{p:existence}
A probability measure $\mu$ on $(\cC,\cF^{\cC})$ is tempered as soon as, for all $t \ge 0$,
\begin{equation}
\label{SuffTempered}
\E \left[ \sup_{s \in [0,t]} \|h(s)\|^2 \right] < \infty
\end{equation}
where $h$ is a random variable with distribution $\mu$.
\end{prop}

The proof of Proposition \ref{p:existence} essentially follows the same strategy as the one of Theorem 3.1 of \cite{CDlS-unbounded} but with some differences about the hypotheses on which these results are based. For the convenience of the reader, we provide a sketched proof in Appendix \ref{s:proof-p:existence}.

\subsection{Absence of percolation}
\label{s:absence-percolation}

Let $\overline\xi$ be a Poisson point process on $\R^2 \times \cC$ with intensity measure $\d x \otimes \mu$ where $\mu$ is a tempered probability measure. We put an edge between $x,y \in \xi$ if $x$ hits $y$ or if $y$ hits $x$. We write in this case $x \sim y$. We thus get an undirected graph structure $\graphe(\overline\xi)$ whose set of vertices is $\xi$. We say that $\graphe(\overline\xi)$ \textbf{percolates} if one of its connected component is infinite.

The \textbf{loop property} expresses the possibility for the model to contain loops as small as we want. Further, this property will be used to locally modify the current configuration by adding a loop at the right scale and at the right place so as to block or stop a connected component. Precisely, we say that the loop property at scale $\alpha>0$ holds if there exists $\cL(\alpha) \in \tribuconfig$ and $r(\alpha) \in (0,\alpha)$ such that:
\begin{enumerate}
\item The event $\cL(\alpha)$ is possible, that is $\P[\overline\xi \in \cL(\alpha)]>0$.
\item The event is local. More precisely $\cL(\alpha)$ only depends on the grains starting in $B(\alpha)$. In other words,
\[
\forall \overline S \in \config, \overline S \in \cL(\alpha) \iff \overline S_{|B(\alpha)} \in \cL(\alpha).
\]
\item On the event $\{\overline \xi \in \cL(\alpha)\}$, the set $\xi \cap B(\alpha)$ can be written as $\{x_1,\ldots,x_n\}$ for some integer $n \geq 2$. Furthermore, in the model associated to $\overline \xi_{|B(\alpha)}$ (whose existence is ensured by the Reconciliation lemma), the following properties hold:
\begin{enumerate}
\item For any $1\leq i\leq n-1$, $x_i$ is stopped by $x_{i+1}$ and $x_n$ is stopped by $x_{1}$.
\item For any $1\leq i\leq n$, $\tau(x_i ; \overline \xi_{|B(\alpha)} ) \le \alpha$.
\item The set
\[
\bigcup_{1\leq i\leq n} G_{x_i}(\tau(x_i ; \overline \xi_{|B(\alpha)} ))
\]
is included in the open annulus $\interior B(\alpha) \setminus B(r(\alpha))$ and it separates $B(r(\alpha))$ from $\R^2 \setminus\interior B(\alpha)$.
\end{enumerate}
\end{enumerate}

\begin{figure}[!ht]
\begin{center}
\includegraphics[width=6.5cm,height=6cm]{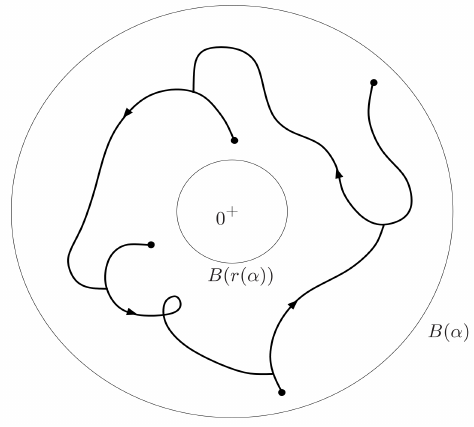}
\caption{\label{fig:loop} On this picture, $\xi \cap B(\alpha)$ is made up with $n=4$ points (the black points). The union of their grains separates $B(r(\alpha))$ from the outside of $B(\alpha)$. The black arrows indicate the sense in which grains grow.}
\end{center}
\end{figure}

Let us remark that in the loop property, the parameter $\alpha$ is used as a spatial parameter and also as a time parameter (Item 3.(b)). Note that the the radius $r(\alpha)$ could be chosen equals to $0$ in this definition without changing anything. Indeed, since the union of grains $\cup_{1\leq i\leq n} G_{x_i}(\tau(x_i ; \overline \xi_{|B(\alpha)} ))$ is a compact set, if Item 3.(c) holds for $r(\alpha)=0$ it holds for $r(\alpha)>0$ as well. By dominated convergence Theorem, if the loop property at scale $\alpha$ holds with $r(\alpha)=0$, it holds also for $r(\alpha)>0$. We keep the presence of $r(\alpha)>0$ in the definition since it plays a crucial role in the following constructions. Note also that the loop property is monotone w.r.t.\ $\alpha$: if it holds at scale $\alpha>0$ then it holds at any scale $\alpha' > \alpha$. 
Indeed, if $\cL(\alpha)$ and $r(\alpha)$ satisfy the above conditions for a given $\alpha>0$, 
then $\cL(\alpha') = \cL(\alpha) \cap \{ \xi \cap (B(\alpha') \setminus B(\alpha)) = \emptyset\}$ and $r(\alpha')=r(\alpha)$ satisfy the above conditions for $\alpha'>\alpha$.    

The main result of the article is the following result.

\begin{theorem}
\label{t}
Let $\overline\xi$ be a Poisson point process on $\R^2 \times \cC$ with intensity measure $\d x \otimes \mu$ where $\mu$ is a tempered probability measure. We denote by $h$ a random variable on $\cC$ with distribution $\mu$. Assume
\begin{itemize}
\item The loop property holds at any scale $\alpha > 0$.
\item There exists $t > 0$ such that
\[
\E \left[\sup_{s \in [0,t]} \|h(s)\|^2 \right] < \infty.
\]
\end{itemize}
Then, almost surely, $\graphe(\overline\xi)$ does not percolate.
\end{theorem}

\subsection{Examples}

\subsubsection{The Brownian model}
\label{sect:BrownianModel}

The Brownian model consists in i.i.d. standard Brownian motions (BM) in $\R^2$ starting from the Poisson points of $\xi$. Equivalently, the Brownian model corresponds to the PPP $\overline \xi$ whose intensity measure $\d x \otimes \mu$ is such that $\mu$ denotes the probability distribution on $(\cC,\cF^{\cC})$ of the standard BM in $\R^2$ starting at the origin. Let $B = \{B_t\}_{t\geq 0}$ be a random variable with distribution $\mu$. We also write $B = \{(B^{(1)}_t,B^{(2)}_t) \}_{t\geq 0}$ where $\{B^{(1)}_t\}_{t\geq 0}$ and $\{B^{(2)}_t\}_{t\geq 0}$ are two i.i.d. one-dimensional (standard) BM starting at $0$. Using the fact that $\sup_{[0,t]} B^{(1)}_s$ is distributed as $|B^{(1)}_t|$, we easily check that
\[
\E \left[\sup_{s \in [0,t]} \|h(s)\|^2 \right]
\]
is finite for any $t$. By Proposition \ref{p:existence}, this suffices to ensure that $\mu$ is tempered, i.e. the Brownian model exists or is well-defined.

The existence of that model was already stated in \cite{CDlS-unbounded} (Corollary 3.2). But the absence of percolation in the Brownian model is a new result.

\begin{prop}
\label{p:NoPercoBM}
The Brownian model previously defined does not percolate (in the sense of Section \ref{s:absence-percolation}) with probability one.
\end{prop}

\begin{proof}
By Theorem \ref{t}, it is enough to prove that the loop property $\mathcal{L}(\alpha)$ holds at any space-time scale $\alpha > 0$ for the Brownian model since the moment condition has been checked just above. We only prove it for $\alpha = 1$ (the same construction works for any $\alpha > 0$, by a standard re-scaling procedure).  Let us first introduce some notations. Given $\varepsilon > 0$ (small), let us set
\begin{itemize}
\item $R_1^\varepsilon := [-\frac{1}{2} - 3\varepsilon , \frac{1}{2} - \varepsilon] \times [-\varepsilon , \varepsilon]$ ;
\item $R_2^\varepsilon := [0 , \frac{1}{2} + 3\varepsilon] \times [-\varepsilon , \varepsilon]$ ;
\item $R_3^\varepsilon := [\frac{1}{2} + \varepsilon , \frac{1}{2} + 3\varepsilon] \times [-\varepsilon , \varepsilon]$.
\end{itemize}
Let $B = \{B_t\}_{t\geq 0}$ be a standard BM in $\R^2$ starting at $(-\frac{1}{2} - 2\varepsilon , 0)$. We claim (without proof) that the event $\mathcal{A}$ depicted in the left hand side of Fig. \ref{fig:Brownian} has positive probability:
\[
\mathcal{A} := \big\{ \forall t \in [0 , 1/2] , \, B_t \in R_1^\varepsilon \big\} \cap \big\{ \forall t \in [1/2 , 1] , \, B_t \in R_2^\varepsilon \big\} \cap \big\{ B_1 \in R_3^\varepsilon \big\} ~.
\]
Thus, taking $\varepsilon$ small enough and four Poisson points respectively located nearby the four points $(\pm\frac{1}{2},\pm\frac{1}{2})$ and using the conjunction of four (translated and rotated) copies of the event $\mathcal{A}$, we can check that the loop property $\mathcal{L}(1)$ holds (with $r(1)$ small enough, say $r(1) = \frac{1}{4}$). See the right hand side of Fig. \ref{fig:Brownian}. The details are left to the reader.
\end{proof}

\begin{figure}[!ht]
\begin{center}
\begin{tabular}{cc}
\includegraphics[width=10cm,height=2.5cm]{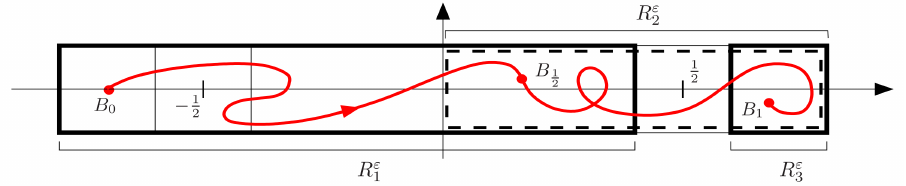} & \includegraphics[width=5cm,height=5cm]{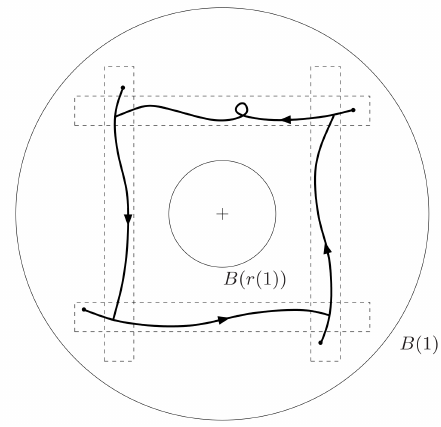}
\end{tabular}
\caption{\label{fig:Brownian} The picture on the left hand side represents the event $\mathcal{A}$. The Brownian trajectory (in red) starting at $B_0 = (-\frac{1}{2} - 2\varepsilon , 0)$ lies in $R_1^\varepsilon \cap R_2^\varepsilon$ at time $t = 1/2$. So, it horizontally crosses the square $(-\frac{1}{2} , 0) + [-\varepsilon , \varepsilon]^2$ during the time interval $[0 , 1/2]$. After that it remains trapped in $R_2^\varepsilon$ and finishes in $R_3^\varepsilon$ at time $t = 1$. In particular, it horizontally crosses the square $(\frac{1}{2} , 0) + [-\varepsilon , \varepsilon]^2$ during the time interval $[1/2 , 1]$. To the right hand side is represented the Brownian model built from four grains whose trajectories satisfy translated and rotated copies of the event $\mathcal{A}$. The resulting grain configuration forms a loop included in the open annulus $\interior B(1) \setminus B(r(1))$ which separates $B(r(1))$ from $\R^2 \setminus\interior B(1)$.}
\end{center}
\end{figure}

\subsubsection{The line-segment model}
\label{sect:Line-Segment}

The line-segment model corresponds to the PPP $\overline \xi$ whose intensity measure $\d x \otimes \mu$ is such that $\mu$ denotes the probability distribution on $(\cC,\cF^{\cC})$ of the trajectory $Y = \{Y(t)\}_{t \geq 0}$ defined as follows: for any $t \geq 0$,
\[
Y(t) := t V \big( \cos \Theta , \sin \Theta \big) \in \R^2
\]
where $V$ and $\Theta$ are two independent random variables such that $V > 0$ a.s. and $\Theta$ is uniformly distributed on $[0,2\pi]$. Hence $Y = \{Y(t)\}_{t \geq 0}$ is a unilateral line-segment growing linearly with velocity $V$ and according to the direction $\Theta$. 

The next result states the existence and the absence of percolation in the line-segment model improving moment conditions given in \cite{CDlS-unbounded} (finite fourth and exponential moments were resp. needed to ensure existence and absence of percolation in Proposition 3.4 and Theorem 3.5 of \cite{CDlS-unbounded}).

\begin{prop}
\label{p:LineSegment}
Whenever $\E[V^2] < \infty$, the line-segment model exists and does not percolate with probability one.
\end{prop}

\begin{proof}
Let us first compute
\[
\E \left[\sup_{s \in [0,t]} \|Y(s)\|^2 \right] = \E \left[\sup_{s \in [0,t]} (s V)^2 \right] = t^2 \E \left[ V^2 \right] ~.
\]
So $\E[V^2] < \infty$ implies the moment conditions in Proposition \ref{p:existence}-- so the existence of the line-segment model --and in Theorem \ref{t}. It then remains to prove the loop property $\mathcal{L}(\alpha)$ at any space-time scale $\alpha > 0$. A basic construction involving four grains in the same spirit as the right hand side of Fig. \ref{fig:Brownian} will work. The details are left to the reader.
\end{proof}

\section{Proof of Theorem \ref{t}}

\label{s:proof}

\subsection{Preliminaries on Palm measure and mass transport principle}

In this section we summarize some standard results around the stationarity of point processes. In particular we recall the special form of the Palm measure for the Poisson point process with marks and the mass transport principle as well. We refer to \cite{Last-Penrose-livre} for a modern and pedagogical presentation.

Let us recall that $\overline \xi$ is a Poisson point process on $\overline S$ with intensity measure $\d x \otimes \mu$. The "heuristic" definition of the Palm measure of $\overline \xi$ at a point $x\in\R^d$ is the conditional distribution of $\overline \xi$ given the event "there is a marked point at the location $x$". Obviously this definition is senseless since this event has null probability. Nevertheless, using the independence properties of the Poisson point process (a spatial version of absence of memory) it is possible to provide a sens to this heuristic definition and show that the Palm measure of $\overline \xi$ at the location $x$ is simply the process $\overline \xi$ itself with an extra point at $x$ with an random mark $h$ following the distribution $\mu$ and being independent to $\overline \xi$. Precisely we define this point process by

$$   \overline \xi^x= \overline \xi \cup \{(x,h)\}$$

and the classical "Slivnyak-Mecke formula" provides a rigorous setting of the heuristic described above: For any  measurable and non-negative function $M : \overline S \times \overline\cS \to [0,+\infty)$ 

\begin{equation}\label{Slivnyak}
\E\left[ \sum_{y \in \overline\xi} M(y,\overline\xi)\right] =  \E\left[ \int_{\R^d} M((x,h),\overline\xi^x)dx\right].
\end{equation}

By stationarity of the Poisson point process $\overline \xi$, it is easy to see that $\overline \xi^x$ has the same distribution than $\overline \xi^0+x$ and therefore $\overline \xi^0$ is the main object of interest. In particular we have the mass transport principle.

\begin{lemma}\label{l:mass-transport-palm} Let $M : \R^d \times \R^d \times \overline\cS \to [0,+\infty)$ be measurable and non-negative.
We assume that $M$ is equivariant under the action of translations of $\R^d$:
\[
\forall x,y,t \in \R^d, \quad \forall \overline S \in \overline\cS, \quad M(x,y,\overline S) = M(x-t,y-t,\overline S-t).
\]
Then,
\[
\E\left[ \sum_{y \in \xi^0} M(0,y,\overline\xi^0)\right] =  \E\left[ \sum_{x \in \xi^0} M(x,0,\overline\xi^0)\right].
\]
\end{lemma}

In words, under the Palm distribution, the mass coming out of the origin is equal to the mass coming into the origin. The lemma admits a proof relying only on stationarity. 
For completeness, we choose to provide a very short and intuitive proof relying on the fact that $\overline\xi$ is a Poisson point process.

\begin{proof}
Let $h$ and $h'$ be two independent random variables with distribution $\mu$.
\begin{align*}
\E\left[ \sum_{y \in \xi^0} M(0,y,\overline\xi^0)\right] 
 & = \E \Big[M\big(0,0,\{(0,h)\} \cup \overline\xi\big) \Big] + \int_{\R^d} \E \Big[M\big(0,y,\{(0,h),(y,h')\} \cup \overline\xi\big) \Big] \d y\\
 & = \E \Big[M\big(0,0,\{(0,h)\} \cup \overline\xi\big) \Big] + \int_{\R^d} \E \Big[M\big(-y,0,\{(-y,h),(0,h')\} \cup (\overline\xi-y)\big) \Big] \d y \\
  & = \E \Big[M\big(0,0,\{(0,h)\} \cup \overline\xi\big) \Big] + \int_{\R^d} \E \Big[M\big(-y,0,\{(-y,h),(0,h')\} \cup \overline\xi\big) \Big] \d y \\
 & = \E \Big[M\big(0,0,\{(0,h)\} \cup \overline\xi\big) \Big] + \int_{\R^d} \E \Big[M\big(x,0,\{(x,h),(0,h')\} \cup \overline\xi\big) \Big] \d x \\
 & = \E\left[ \sum_{x \in \xi^0} M(x,0,\overline\xi^0)\right].
\end{align*}
We used the Slivnyak-Mecke formula in Steps 1 and 4, the equivariance of $M$ in step 2 and the stationarity of $\overline\xi$ in step 3.
\end{proof}

%
%
%\paragraph{All connections are regular:} 
%\begin{equation}\label{e:model_connexion_regular}
%\forall x \neq y \in \xi, \text{ if } (x \text{ hits } y \text{ or } y \text{ hits } x), 
%\text{ then } (x \text{ hits } y \text{ in a regular way or }  y \text{ hits } x \text{ in a regular way}).
%\end{equation}
%Clearly, \eqref{e:model_hit_regular} implies \eqref{e:model_connexion_regular}.
%
%\jb{le contre-exemple avec les grains qui tournent en rond vérifie sans doute cette condition (!?) ; on n'a pas d'exemple où cette condition n'est pas vérifiée (?) ;
%la réduction à des pog reste simple avec des hypothèses plus faibles ; développer ? on pourrait au contraire sucrer \eqref{e:model_connexion_regular} éventuellement mais 
%c'est une solution de facilité ! }

\subsection{Reduction to an out-degree one graph }\label{Soutone}

In the undirected graph $\graphe(\overline\xi)$ (see  Section \ref{s:absence-percolation}), 
two distinct vertices $x,y \in \xi$ are linked by an edge if $x$ hits $y$ or $y$ hits $x$.
We write $x \sim y$ when there is such an edge.

For all $x \in \xi$ we now define a successor $s(x)=s(x ; \overline\xi) \in \xi$ as follows:
\begin{enumerate}
\item If $\tau(x)=+\infty$, we set $s(x)=x$.
\item If $\tau(x)<+\infty$, there are two cases: 
\begin{enumerate}
\item If $x$ is stopped in a regular way by some $y$ (see \eqref{e:x-stopped-by-y-in-a-regular-way}), 
then there exists a unique such $y$ because of the hard-core property expressed in Lemma \ref{lem:hard-core}. In that case we set $s(x)=y$.
\item If $x$ is not stopped in a regular way by any $y$, we set $s(x)=x$.
\end{enumerate}
\end{enumerate}
We now define a directed graph $\grapheoriente(\overline\xi)$ with set of vertices $\xi$ by putting an arrow from $x$ to $y$ when $s(x)=y$.
We write $x \to y$ when there is such an arrow.
The graph $\grapheoriente(\overline\xi)$ satisfies two basic and elementary properties stated in the following lemma.

\begin{lemma} \label{l:c_est_un_pog} $ $
\begin{enumerate}
\item For any $x \in \xi$, the outdegree of $x$ in $\grapheoriente(\overline\xi)$ is $1$.
\item The graph $\grapheoriente(\overline\xi)$ is shift-invariant.
This means that for all $u \in \R^d$,
\[
\grapheoriente(\overline\xi - u) = \grapheoriente(\overline\xi) - u.
\]

\end{enumerate}
\end{lemma}

\begin{proof} Recall that we work on a full probability event on which $\overline\xi \in \overline\cA$ where $\overline\cA$ is a tempered set of configurations.
The first item is a straightforward consequence of the definition of the graph.
Let us check the second idem.
As $\overline\xi$ belongs to $\overline\cA$ and as $\overline\cA$ is invariant under spatial translations, $\overline\xi - u \in \overline\cA$.
Moreover, by uniqueness of the lifetime function, for all $x \in \xi$, $\tau(x-u ; \overline\xi - u) = \tau(x ; \overline\xi)$.
The result easily follows.
\end{proof}

We denote by $\cluster(x; \overline\xi)$ the undirected component of $x$ in $\grapheoriente(\overline\xi)$.
This is the set of vertices we can reach from $x$ by following the arrows forward or backwards.
More formally,
\begin{align}
\cluster(x) = \cluster(x; \overline\xi)=\{ y \in \xi :  & \text{ there exists } n \ge 0 \text{ and a sequence }x=a_0,\dots,a_n=y \nonumber \\
 & \text{  such that, for all }i \in \{1,\dots,n\}, a_{i-1} \to a_i \text{ or } a_i \to a_{i-1}\}. \label{e:cluster_non_oriente}
\end{align}
There is a strong link between percolation in the graphs $\graphe(\overline\xi)$ and $\grapheoriente(\overline\xi)$ as stated in the next lemma.

\begin{lemma}  \label{l:reduction-pog}
On a full probability event, the following holds.
For any $x \in \xi$, if the connected component of $x$ in $\graphe(\overline\xi)$ is infinite, 
then the undirected component $\cluster(x ; \overline\xi)$ of $x$ in $\grapheoriente(\overline\xi)$ is infinite.
\end{lemma}

Thanks to Lemma \ref{l:reduction-pog}, in order to prove Theorem \ref{t}, it is sufficient to prove that all the undirected component of $\grapheoriente(\overline\xi)$
are finite almost surely.

\begin{proof}[Proof of Lemma \ref{l:reduction-pog}]

There is a simple proof in the case where all the hits are regular (see \eqref{e:x-stopped-by-y-in-a-regular-way} for the definition). This means that for all $x \neq y \in \xi$ then $x \sim y$ implies that $x$ hits $y$ or $y$ hits $x$ which is equivalent to $x \to y$ or $y \to x$.
Therefore the connected components of $\graphe(\overline\xi)$ coincide with the undirected components of $\grapheoriente(\overline\xi)$.

We now give the proof in the general case where non-regular hits are possible. The basic idea is that, if in a connected component $C$ of $\graphe(\overline\xi)$ there exists two points $x \sim y$ such that neither $x \to y$ nor $y \to x$, 
then one can associate with the connected component $C$ a unique special point in a translation equivariant invariant way. By a standard application of the mass transport principle, this implies that the connected component $C$ is finite.

We say that $x \in \xi$ stops in a singular way if
\begin{enumerate}
\item $\tau(x)$ is finite.
\item There exists no $y \in \xi$ (we do not exclude $y=x$) such that $g_x(\tau(x))$ belongs to $g_y([0,\tau(y)))$.
\end{enumerate}
In this case we say that $x$ stops in a singular way at $g_x(\tau(x))$.
Let $C$ be a connected component of $\graphe(\overline\xi)$.
We say that $s \in \R^2$ is a singular point of $C$ if there exists $x \in C$ such that $x$ stops in a singular way at $s$.
The proof relies on the following observation whose proof is postponed.

\begin{claim}\label{c:unique-singular-point} The number of singular points of a connected component of $\graphe(\overline\xi)$ is either $0$ or $1$.
\end{claim}

The following claim follows from the previous one and stationarity through standard mass transport arguments. The proof is also postponed.

\begin{claim} \label{c:cluster-speciaux-finis} With probability one, any connected component of $\graphe(\overline\xi)$ which contains a singular point is finite.
\end{claim}

We first show how to conclude using Claim \ref{c:cluster-speciaux-finis}. 
Fix $C$ an infinite connected component of $\graphe(\overline\xi)$.
By Claim \ref{c:cluster-speciaux-finis}, $C$ admits no singular point.
Let $x \neq x' \in C$ be such that $x \sim x'$.
By symmetry we can assume that $x$ hits $x'$.
There are two cases:
\begin{enumerate}
\item If $x$ hits $x'$ in a regular way, then there is an arrow $x \to x'$ in $\grapheoriente(\overline\xi)$.
\item Otherwise, $\tau(x)$ and $\tau(x')$ are finite and
\[
g_x(\tau(x))=g_{x'}(\tau(x')) \text{ and } g_x(\tau(x)) \not\in g_{x'}([0,\tau(x'))).
\]
But $C$ admits no singular point.
Therefore there exists $y \in \chi$ such that $g_x(\tau(x)) \in g_y([0,\tau(y)))$.
Now there are three subcases:
\begin{enumerate}
\item Subcase $y \not\in \{x,x'\}$. 
Then $x$ and $x'$ hits $y$ in a regular way: there are two arrows $x \to y$ and $x' \to y$ in $\grapheoriente(\overline\xi)$.
\item Subcase $y=x'$. Then $x$ hits $x'$ in a regular way. This subcase is ruled out by Case 1.
\item Subcase $y=x$. Then $x'$ hits $x$ in a regular way: there is an arrow $x' \to x$ in  $\grapheoriente(\overline\xi)$.
\end{enumerate}
\end{enumerate}
In all cases, $x$ and $y$ are in the same undirected cluster $\cluster$ of $\grapheoriente(\overline\xi)$.
Therefore $C$ is a subset of such an undirected cluster $\cluster$ .
As $C$ is infinite, $\cluster$ is infinite and the proof of the lemma  is over up to the two postponed proofs.

\begin{proof}[Proof of Claim \ref{c:cluster-speciaux-finis} using Claim \ref{c:unique-singular-point}.]
This is a standard application of stationarity through mass transport principle. Denote by $C(x ; \overline\xi)$ the connected component of $x \in \xi$ in $\graphe(\overline\xi)$. By Slivnyak-Mecke formula \eqref{Slivnyak} and Lemma \ref{l:mass-transport-palm}, for any bounded set $\Lambda\subset \R^2$

\begin{align}\label{proofFC}
   & \E\left[\sum_{x,y \in \xi} \1_{\{C(x; \overline\xi) \text{ admits the singular point } y \text{ and } y \text{ belongs to } \Lambda\}}\right]\nonumber\\
   & = \int_{\R^2}\E\left[ \sum_{x \in \xi^y} \1_{\{C(x; \overline\xi^y) \text{ admits the singular point } y \text{ and } y \text{ belongs to } \Lambda\}}\right]dy\nonumber\\    
   & = |\Lambda|\E\left[ \sum_{x \in \xi^0} \1_{\{C(x; \overline\xi^0) \text{ admits the singular point } 0 \}}\right]\nonumber\\
   & = |\Lambda|\E\left[ \sum_{x \in \xi^0} \1_{\{C(0; \overline\xi^0) \text{ admits the singular point } x \}}\right]\nonumber\\
    & \le |\Lambda|.
\end{align}

It is enough to show that for any bounded $\Lambda$, the variable  $\sum_{x,y \in \xi} \1_{\{C(x; \overline\xi) \text{ admits the singular point } y \text{ and } y \text{ belongs to } \Lambda\}}$ is finite almost surely. Therefore if it exists a connected component with a singular point, necessary this connected component is bounded.   

\end{proof}

\begin{proof}[Proof of Claim \ref{c:unique-singular-point}.]
Let $C$ be a connected component of $\graphe(\overline\xi)$.
Assume, for a contradiction, that $C$ admits two distinct singular points $s$ and $s'$.
Then, there exists $x, x' \in C$ and $n$ such that 
\begin{align*}
& \tau(x) < \infty \text{ and }  g_x(\tau(x))=s \\
\text{ and } & \tau(x') < \infty \text{ and } g_{x'}(\tau(x'))=s' \\
\text{ and } & \text{there exists a path } x=x_0, \dots, x_n=x' \text{ such that for all } i \in \{1,\dots,n\}, x_{i-1} \sim x_i.
\end{align*}
Choose $x, x'$ and $n$ as above and such that $n$ is minimal.
Necessarily, the $x_i$ are distinct (as the length of the path is minimal) and $n \ge 1$ (as $x \neq x'$ because $s \neq s'$).
We now show by induction on $i$ the following property:
\begin{equation}\label{e:tout_pointe_vers_s}
\forall i \in \{1,\dots,n\}, \; \tau(x_i) < \infty \text{ and } g_{x_i}(\tau(x_i)) \in g_{x_{i-1}}([0,\tau(x_{i-1}))).
\end{equation}
Let $i \in \{1,\dots,n\}$. 
As $x_i \sim x_{i-1}$, at least one of the following three conditions holds:
\begin{enumerate}
\item $\tau(x_{i-1})<\infty$ and $g_{x_{i-1}}(\tau(x_{i-1})) \in g_{x_i}([0,\tau(x_i)))$.
\item $\tau(x_i)<\infty$ and $g_{x_i}(\tau(x_i)) \in g_{x_{i-1}}([0,\tau(x_{i-1})))$.
\item $\tau(x_{i-1})<\infty$ and $\tau(x_i)<\infty$ and $g_{x_{i-1}}(\tau(x_{i-1}))=g_{x_i}(\tau(x_i))$.
\end{enumerate}
Consider the case $i=1$.
The first condition above can not hold. Otherwise we would have $s = g_x(\tau(x)) \in g_{x_1}([0,\tau(x_1)))$ and $s$ would not be a singular point.
The third condition above can not hold. Otherwise we could consider the path $x=x_1,\dots,x_n=x'$ and $n$ would not be minimal.
Therefore the second condition holds and the desired property is proven for $i=1$.

Consider now the case $i \in \{2,\dots,n\}$ and assume that the desired property is true for $i-1$.
The first condition can not hold. Otherwise $g_{x_{i-1}}(\tau(x_{i-1}))$ would belong to $g_{x_i}([0,\tau(x_i)))$ and to $g_{x_{i-2}}([0,\tau(x_{i-2})))$
(by the induction hypothesis) and therefore the intersection between $g_{x_i}([0,\tau(x_i)))$ and $g_{x_{i-2}}([0,\tau(x_{i-2})))$ would be non empty.
But as $x_{i-2} \neq x_i$, this is forbidden by the hard-core property, see Lemma \ref{lem:hard-core}.
The third condition can not hold. Otherwise, we would have $g_{x_i}(\tau(x_i)) = g_{x_{i-1}}(\tau(x_{i-1})) \in g_{x_{i-2}}([0,\tau(x_{i-2})))$ and thus
$x_i$ would hit $x_{i-2}$. We could then consider the path $x=x_0,\dots,x_{i-2},x_i,\dots,x_n=x'$ and $n$ would not be mimimal.
Therefore the second condition holds and the desired property is proven for $i$.

By induction this establishes \eqref{e:tout_pointe_vers_s}. 
In particular $s ' = g_{x'}(\tau(x')) = g_{x_n}(\tau(x_n)) \in g_{x_{n-1}}([0,\tau(x_{n-1})))$.
But this is not possible as $s'$ is a singular point.
This is our contradiction and this completes the proof of Claim \ref{c:unique-singular-point}.
\end{proof}
The proof of Lemma \ref{l:reduction-pog} is complete.
\end{proof}

\subsection{Absence of forward percolation is sufficient}\label{Soutonef}

For any $x \in \xi$ we define three components in $\grapheoriente(\overline\xi)$:
\begin{itemize}
\item The forward component of $x$.
This is the set of all vertices we can reach from $x$ by following the arrows.
More formally,
\begin{align*}
\forward(x)=\forward(x;\overline\xi)=\{ y \in \xi :  & \text{ there exists } n \ge 0 \text{ and a sequence }x=a_0,\dots,a_n=y \\
 & \text{  such that, for all }i \in \{1,\dots,n\}, a_{i-1} \to a_i\}.
\end{align*}
\item The backward component of $x$.
This is the set of all vertices we can reach from $x$ by following the arrows in reverse.
More formally,
\[
\back(x)=\back(x;\overline\xi) = \{z \in \xi : x \in \forward(z)\}.
\]
\item The undirected connected component of $x$. This is the set $\cluster(x) = \cluster(x ; \overline \xi)$ defined by \eqref{e:cluster_non_oriente}.
\end{itemize}

The following lemma is a known consequence of Lemmas \ref{l:mass-transport-palm} and \ref{l:c_est_un_pog}. A proof can find in \cite{CDLT} for instance, but for the convenience of the reader we provide a sketch of the proof below. 

\begin{lemma} \label{l:reduction-forward} 
One a full probability event, the following holds:
\[
\forall x \in \xi, \; \forward(x) \text{ is finite } \iff \cluster(x) \text{ is finite}.
\]
\end{lemma}

Combining Lemmas \ref{l:reduction-pog} and Lemmas \ref{l:reduction-forward} we get that, in order to prove that all the connected component of $\graphe(\overline\xi)$ are
almost surely finite, it is sufficient to prove that all forward components $\forward(x)$ of $\grapheoriente(\overline\xi)$ are finite.

\begin{proof} 

A loop in $\grapheoriente(\overline\xi)$ is a finite sequence $x_0,\dots,x_n$, $n \ge 0$, of points of $\xi$ such that $x_0=x_n$ and, for all $i \in \{1,\dots,n\}$, $x_{i-1} \to x_i$.
Using only the first property of Lemma \ref{l:c_est_un_pog} we get for each $x \in \xi$ that exactly one of the following properties is satisfied:
\begin{enumerate}
\item $\forward(x)$ is infinite and $C(x)$ contains no loop. In that case $C(x)$ is infinite.
\item $\forward(x)$ is finite and $C(x)$ contains a unique loop. In that case $C(x)$ can be finite or infinite.
\end{enumerate}
This can be proven as follows.
First one checks that $\cluster(x)$ is the union of $\forward(x)$ and of all $\back(y)$ for $y \in \forward(x)$.
Then one distinguishes two cases: $\forward(x)$ is finite (and thus contains a unique loop) or infinite (and thus contains no loop).
One concludes by noticing that adding the backwards components $\back(y)$ adds no loops.

The proof of Lemma \ref{l:reduction-forward} is then complete once we have proven
\begin{equation} \label{e:boucle_infini_impossible}
\P[\exists x \in \xi : C(x) \text{ is infinite and contains a unique loop}]=0.
\end{equation}

This property can be proved as in the Claim \ref{c:cluster-speciaux-finis}. Indeed a cluster with an unique loop contains a special point which is for instance the most right point in the loop. Following similar computation as in \eqref{proofFC}, we show that almost surely this cluster is finite. 

\end{proof}

\subsection{Absence of forward percolation of bad points is sufficient }\label{Soutonefb} 

In the previous sections, we showed that the proof of Theorem \ref{t} is reduced to the proof of the absence of forward percolation in the graph $\grapheoriente(\overline{\xi})$. In this section, we show that we can reduce the setting further since it is only necessary to prove the absence of forward percolation for a class of "bad" points that we define below. We start with the definition of "good" points.

\subsubsection{Definition of good points}\label{SectionGP}

\paragraph{Loops.} We say that $x \in \xi$ belongs to a loop if there exists $n \ge 1$ such that $s^{(n)}(x)=x$ where $s^{(n)}$ denotes the $n$-th iterate of $s$.
In particular, in the degenerate case where $s(x)=x$, $x$ belongs to a loop.
We say that $x \in \xi$ is followed by a loop if $x$ does not belong to a loop and if its successor $s(x)$ belongs to a loop.

\paragraph{Notion of $(\alpha,u)$-good points: informal definition.} 
Let $\alpha>0$ and $u\in\R^2$ be such that $0 \not\in B(u,\alpha)$.
Slightly informally, we say that $x \in \xi$ is an $(\alpha,u)$-good point if replacing the configuration inside $B(x+u,\alpha)$ by a local configuration belonging to $x+u+\cL(\alpha)$
creates a loop after $x$ and does not decrease the cardinality of the backward of $x$.

\paragraph{Notion of $(\alpha,u)$-good points: formal definition.} 
Let $\overline\cS^0$ denote the set of configurations $\overline S \in \cS$ such that $0$ belongs to $S$.
We can focus on such configurations thanks to the fact that everything is well behaved under the action of translations.

For any $\overline S, \overline S' \in \overline\cS$  we set
\begin{equation}\label{e:resample}
\text{resample}(\overline S, \overline S') = \overline S_{|B(u,\alpha)^c} \cup \overline S'_{|B(u,\alpha)}.
\end{equation}
This define a measurable map from $\overline\cS^2$ to $\overline\cS$.
Let $\overline\xi'$ be an independent copy of $\overline\xi$.
In words, $\text{resample}(\overline\xi,\overline\xi')$ is obtained from $\overline\xi$ by resampling inside $B(u,\alpha)$ using $\overline\xi'$.
Define $\good(\alpha,u)$ as the measurable subset of configurations $\overline S \in \overline\cS^0$ such that:
\begin{enumerate}
\item Almost surely, $\text{resample}(\overline S, \overline \xi')$ belongs to $\overline\cA$.
\item Almost surely, on the event $\{\overline\xi' \in u+\cL(\alpha)\}$, the following conditions hold:
\begin{enumerate}
\item In $\text{resample}(\overline S, \overline \xi')$, $0$ is followed by a loop. 
\item $\#\back(0,\overline S) \le \#\back(0,\text{resample}(\overline S, \overline \xi'))$.
\end{enumerate}
\end{enumerate}
We then say that $x \in S$ is an $(\alpha,u)$-good point in $\overline S$ if $\overline S-x \in \good(\alpha,u)$.

The first condition is needed to ensure that the second condition makes sense.
It is however harmless as stated in Lemma \ref{l:good-definition-fondee} in Appendix \ref{s:des_trucs_de_mesurabilite}.

%Let $\nu$ be the distribution of $\overline\xi$. This is a probability measure on $\cF^{\overline\cS}$.
%Let $\nu^0$ be the Palm distribution of $\overline\xi$. This is a probability measure on $\cF^{\overline\cS^0}$ defined by
%\[
%\nu^0(F) = \E\left[\sum_{x \in \xi \cap [0,1]^d} \1_F(\overline\xi - x)\right].
%\]
%The intuition is that $\nu^0$ is the distribution of the stationary marked point process $\overline\xi$ conditioned by having a point in $\{0\} \times \cC$.
%

\paragraph{Notion of $\alpha$-good point.}
Let $0<\alpha<1$. 
Recall that $r(\alpha)$ is part of the definition of $\cL(\alpha)$.
We fix a finite set $U(\alpha)$ in $\R^2$ such that
\[
B(0,\alpha^{-1})\setminus B(0,\alpha) \subset \bigcup_{u \in U(\alpha)} B(u,r(\alpha)) \subset \R^2 \setminus \{0\}.
\]
We say that $x \in \overline S$ is $\alpha$-good in the configuration $\overline S$ if there exists $u \in U(\alpha)$ such that $x$ is $(\alpha,u)$-good.
In other words, we set
\[
\good(\alpha) = \bigcup_{u \in U(\alpha)} \good(\alpha,u)
\]
and we say that $x\in \overline S$ is $\alpha$-good in the configuration $\overline S$ if $\overline S-x$ belongs to $\good(\alpha)$.
The definition depends on the choice of $U(\alpha)$ but it will not be as issue.
Recall that $U(\alpha)$ is fixed.

\subsubsection{The number of good points in a forward is finite}
\label{sect:NbGoodFinite}

In the following lemma we show that almost surely an infinite forward cluster can not contain an infinite number of good points. The proof is inspired by similar results developped in Section 4.3 in \cite{CDLT}. 

\begin{lemma} \label{l:forward-good-finite}
Let $\alpha>0$. With probability one, for all $x \in \xi$, the set $\{y \in \forward(x) : y \text{ is good for }\overline\xi\}$ is finite.
\end{lemma}

Before proving Lemma \ref{l:forward-good-finite} let us state an immediate consequence.
We say that a point $x \in \overline\xi$ is $\alpha$-bad when it is not $\alpha$-good.

\begin{lemma}
\label{l:csq-forward-good-finite}
Assume that there exists $\alpha>0$ such that, almost surely, there is no $x$ in $\xi$ such that 
\begin{itemize}
\item $\forward(x)$ is infinite.
\item $\forall y \in \forward(x)$, $y$ is $\alpha$-bad for $\overline\xi$.
\end{itemize}
Then, almost surely, for all $x \in \xi$, $\forward(x)$ is finite.
\end{lemma}

Combining Lemmas \ref{l:reduction-pog}, Lemmas \ref{l:reduction-forward} and Lemma \ref{l:csq-forward-good-finite} we get that, in order to prove that all the connected component of 
$\graphe(\overline\xi)$ are almost surely finite, it is sufficient to prove that all infinite forward components contains $\alpha$-good points.
In other words, there is no percolation in the original non oriented graph as soon as there is no forward percolation of $\alpha$-bad points for $\alpha$ small enough.
The problem is thus reduced to the proof of the existence of a subcritical regime in some oriented percolation model.
This is a huge improvement as it enables us to play with $\alpha$.

\begin{proof}[Proof of Lemma \ref{l:csq-forward-good-finite}] Let $\alpha$ be as in the statement. With probability one, by assumption and by Lemma \ref{l:forward-good-finite}, we have
\begin{equation}\label{e:course1}
\forall x \in \xi, \quad \forward(x) \text{ is finite or there exists } y \in \forward(x) \text{ such that } y \text{ is } \alpha\text{-good}
\end{equation}
and
\begin{equation}\label{e:course2}
\forall x \in \xi, \quad \{y \in \forward(x) : y \text{ is good for }\overline\xi\} \text{ is finite}.
\end{equation}
We work on this event. 
Let $x \in \xi$.
Assume that $\forward(x)$ is infinite.
By \eqref{e:course2} there exists $x^* \in \forward(x)$ such that all points of $\forward(x^*)$ are $\alpha$-bad.
This contradicts \eqref{e:course1}.
This establishes the lemma.
\end{proof}

Lemma \ref{l:forward-good-finite} is a consequence of the following result.

\begin{lemma} \label{l:forward-u-good-finite}
Let $\alpha>0$ and $u\in\R^2$ be such that $0 \not\in B(u,\alpha)$.
Then,
\[
\E\Big[\sum_{y \in \forward(0,\overline\xi^0)} \1_{y \text{ is } (\alpha,u)\text{-good for }\overline\xi^0}\Big]
\le
\frac{ 1}{\P\big[\overline\xi \in \cL(\alpha)\big]}.
\]
\end{lemma}

\begin{proof}
Let $\overline\xi'$ be a copy of $\overline\xi$ independent of $\overline\xi$ and of $\overline\xi^0$.
By definition of $(\alpha,u)$-good we have
\begin{align*}
 &\#\back(0,\overline\xi^0) \1_{0 \text{ is } (\alpha,u)\text{-good for }\overline\xi^0 } \1_{\overline\xi' \in u+\cL(\alpha)} \\
& \le
\#\back(0,\text{resample}(\overline\xi^0,\overline\xi')) 
\1_{0 \text{ is followed by a loop in } \text{resample}(\overline\xi^0,\overline\xi')}.
\end{align*}
But $\overline\xi^0$ is independent of $\overline\xi'$,
$\text{resample}(\overline\xi^0,\overline\xi')$ has the same distribution as $\overline\xi^0$ 
(as $0 \not\in B(u,\alpha)$) and $\overline\xi'-u$ has the same distribution as $\overline\xi$.
Therefore, taking expectation, we get
\begin{align} \label{e:goodbackfor1}
\E\left[\#\back(0,\overline\xi^0) \1_{0 \text{ est } (\alpha,u)\text{-good for }\overline\xi^0}\right] \P\big[\overline\xi \in \cL(\alpha)\big] 
\le \E\left[\#\back(0, \overline\xi^0) \1_{0 \text{ is followed by a loop in }\overline\xi^0}\right].
\end{align}
We plan to apply Lemma \ref{l:mass-transport-palm} (mass transport principle).
Recall $\overline\cA$ the set of tempered configuration whose existence is due to the fact that $\mu$ is a tempered distribution.
Recall $\P[\overline\xi^0 \in \overline\cA]=1$ by Lemma \ref{l:good-definition-fondee}.
We consider the map $M$ defined for $x, y \in \R^d$ and $\overline S \in \overline\cS$ by
\[
M(x,y,\overline S)
= \1_{\overline S \in \overline \cA} \1_{x \in \back(y,\overline S)} \1_{y \text{ is } (\alpha,u) \text{-good for } \overline S}
= \1_{\overline S \in \overline \cA} \1_{y \in \forward(x,\overline S)} \1_{y \text{ is } (\alpha,u) \text{-good for } \overline S}.
\]
In particular $M(x,y,\overline S)=0$ if $x$ or $y$ does not belong to $S$.
Now,
\begin{align}\label{e:goodbackfor2}
\E\left[\#\back(0,\overline\xi^0) \1_{0 \text{ is } (\alpha,u)\text{-good for }\overline\xi^0}\right] 
& = \E\left[\sum_{x \in \back(0,\overline\xi^0)} \1_{0 \text{ is } (\alpha,u)\text{-good for }\overline\xi^0}\right] \nonumber \\
& = \E\left[\sum_{x \in \overline\xi^0} M(x,0,\overline\xi^0)\right] \nonumber \\
& = \E\left[\sum_{y \in \overline\xi^0} M(0,y,\overline\xi^0)\right] \nonumber \qquad \text{ (by Lemma \ref{l:mass-transport-palm})}\\
& = \E\left[\sum_{y \in \forward(0,\overline\xi^0)} \1_{y \text{ is } (\alpha,u)\text{-good for }\overline\xi^0}\right].
\end{align}
Similarly,
\begin{equation}\label{e:goodbackfor3}
\E\left[\#\back(0,\overline\xi^0) \1_{0\text{ is followed by a loop in }\overline\xi^0}\right]
= 
\E\left[\sum_{y \in \forward(0,\overline\xi^0)} \1_{y\text{ is followed by a loop in }\overline\xi^0} \right].
\end{equation}
But there is at most one point followed by a loop in any forward.
Combining this fact, \eqref{e:goodbackfor1}, \eqref{e:goodbackfor2} and \eqref{e:goodbackfor3} we get the lemma.
\end{proof}

Lemma \ref{l:forward-u-good-finite} has a nice intuitive interpretation. 
One may think -- this is a rough picture -- that the forward ends by a loop after each $(\alpha,u)$-good point with probability at least $\P[\overline\xi \in \cL(\alpha)]$.
One may then hope that the number of $(\alpha,u)$-good points in a forward is dominated by a geometrical distribution with parameter $\P[\overline\xi \in \cL(\alpha)]$.
The expected value of this number of points would then be at most the inverse of this parameter.
While the actual picture is more complex, the lemma states that the latter prediction actually holds true.

\begin{proof}[Proof of Lemma \ref{l:forward-good-finite}]
For any $\overline S \in \overline \cS$ 
For any $x \in \xi$ write $\forward^\text{Good}(x ; \overline\xi)= \{y \in \forward(x ; \overline\xi) : y \text{ is } \alpha\text{-good for }\xi\}$.
Define similarly $\forward^\text{Good}(0 ; \overline\xi^0)$.
Set $C=[0,1]^d$. Thanks to Slivnyak-Mecke Formula \eqref{Slivnyak}, the Lemma \ref{l:forward-u-good-finite} and the definition of $\#U(\alpha)$, we have 
\[
\E\Big[\sum_{x \in \xi \cap C} \#\forward^\text{Good}(x;\overline\xi)\Big] 
=
\E\Big[\#\forward^\text{Good}(0;\overline\xi^0)\Big] \le \frac{ \#U(\alpha)}{\P\big[\overline\xi \in \cL(\alpha)\big]} < \infty.
\]
We get the result.
\end{proof}

\subsection{Absence of forward percolation of $\sharp$-bad points is sufficient}
\label{sect:DieseBadPoints}

In this section, we introduce a new model, an augmented model in the sense that the collection of new stopped curves contains the curves of the original models. This model has better stochastic properties and its study will be easier. In particular, the proof of Theorem \ref{t} will be reduced to the absence of forward percolation of bad points in this new model.

Before starting its construction, recall that we work under the framework and assumptions of Theorem \ref{t}. In particular, there exists $t>0$ such that 
\[
\E\left[\sup_{s \in [0,t]} \|h(s)\|^2\right]<\infty.
\]
In fact we can assume $t=1$. If $t \ge 1$, this is straightforward. If $t < 1$, it suffices to perform a linear change of time for the probability measure $\mu$: replace $\mu$ by its image under the map $h \mapsto (s \mapsto h(st))$. The assumptions are still satisfied for the new probability measure (for the loop property at scale $\alpha$, one can use the fact that the initial measure satisfies the loop property at scale $\alpha t$ and argue as in the proof of the monotony of the loop property) and the graph obtained does not change.

We therefore assume $t=1$ to simplify notations,
\[
\E\left[\sup_{s \in [0,1]} \|h(s)\|^2\right]<\infty.
\]

\subsubsection{The $\sharp$-model}

\label{s:sharp-model}

Recall $\cL(1)$ and $r(1)$ from the definition of the loop property at scale $1$.
We say that there is an obstacle around $c \in \R^2$ if the following conditions hold:
\begin{enumerate}
\item $\overline\xi - c \in \cL(1)$.
\item For all $x \in \xi \setminus B(c,1)$, $G_1(x) \cap B(c,1) = \emptyset$.
\end{enumerate}
We then define $\cO_c = \cO_c(\overline\xi)$ as follows.
\begin{itemize}
\item If there is an obstacle around $c \in \R^2$ we define $\cO_c$ as the union of the terminal state of all grains which started in $B(c,1)$:
\[
\cO_c = \bigcup_{x \in \xi \cap B(c,1)} G_x(\tau(x)).
\]
\item Otherwise, we set $\cO_c=\emptyset$.
\end{itemize}
We then define the global obstacle by
\[
\cO = \cO(\overline\xi) = \bigcup_{c \in \R^2} \cO(c).
\footnote{Note that if there is an obstacle around $c \in \R^2$, then for any $c' \in \R^2$ close enough to $c$ one has $\cO_c'=\cO_c$.
This yields $\cO = \bigcup_{c \in \Q^2} \cO_c$. This alternative description of $\cO$ enables to solve some potential measurability issues.}
\]

For all $x \in \xi$ we define 
\[
\tau^\sharp(x) = \tau^\sharp(x ; \overline\xi) = \inf \{ t \ge 1 : g_t(x) \in \cO \}
\]
and
\[
G^\sharp(x) = G^\sharp(x ; \overline\xi) = G_x(\tau^\sharp(x)).
\]
As mentioned above, this defines a $\sharp$-model which is an augmented version of the initial model.
More precisely, we have the following lemma.

\begin{lemma} \label{l:le_modele_augmente_domine}
For all $x \in \overline\xi$, $\tau(x) \le \tau^\sharp(x)$.
\end{lemma}

\begin{proof} 
Assume there is an obstacle around some point $c$.
As $G_1(x) \cap B(c,1) = \emptyset$ for all $x \in \xi \setminus B(c,1)$, no grain starting outside $B(c,1)$ reached $B(c,1)$ at time $1$.
As $\overline\xi - c \in \cL(1)$, when considering only the dynamics of the grains inside $B(c,1)$, those grains stops before time $1$ and their union at time $1$ is $\cO(c)$.
Therefore, when considering the whole dynamic, the evolution of the grains starting inside $B(c,1)$ is not modified by the evolution of the grains starting  outside $B(c,1)$.
Thus the set $\cO(c)$ is created before time $1$ and is, indeed, an obstacle for the other grains after time $1$.

Let now $x \in \xi$. 
There are two cases.
\begin{itemize}
\item The grain $x$ is one of the grains involved in the creation of the obstacles.
In this case, we know that $\tau(x)$ is at most $1$.
But by definition $\tau^\sharp(x)$ is at least $1$.
Therefore $\tau(x) \le \tau^\sharp(x)$.
\item The grain $x$ is not one of the grains involved in the creation of the obstacles.
In this case, the grain starting at $x$ is stopped at the latest at the first time after $1$ at which it meets an obstacle.
We thus also get $\tau(x) \le \tau^\sharp(x)$ in this case as well.
\end{itemize}
\end{proof}

Fix $\rho > 1$. For all $c \in \R^2$ and $n \ge 1$, we define the annulus  
\[
A(c,n) =  B(c,\rho^n) \setminus \interior{B}(c,\rho^{n-1})
\]
and $A(n) = A(0,n)$ when the center $c$ equals to the origin. For $x \in \xi$, we consider the fattened set 
\[
G^{\sharp\text{-fat}}(x)=G^\sharp(x)+B(3) = \{g+b : g \in G^\sharp(x), b \in B(3)\}.
\]
For $n \ge 1$, we say $G^{\sharp\text{-fat}}(x)$ crosses the annulus $A(n)$ if 
\[
G^{\sharp\text{-fat}}(x) \cap \interior{B}(\rho^{n-1}) \neq \emptyset \text{ and } G^{\sharp\text{-fat}}(x) \cap B(\rho^n)^c \neq \emptyset.
\]
The following result will enable us to control the dependencies in the $\sharp$-model.

\begin{prop}
\label{p:33}
There exists $\rho_0>1$ such that, for any $\rho \ge \rho_0$, with probability $1$, there exists $n_0$ such that:
\begin{enumerate}
\item For all $n \ge n_0$ and all $x \in \xi$, $G^{\sharp\text{-fat}}(x)$ does not cross $A(n)$.
\item For all $n \ge n_0$ and all $x \in \xi \cap B(\rho^n)$, $G^\sharp(x) = G^\sharp(x ; \overline\xi_{B(\rho^{n+1})})$.
\end{enumerate}
\end{prop}

Let us note that the last object $ G^\sharp(x ; \overline\xi_{B(\rho^{n+1})})$ is well defined since the definition of obstacles can be extended to the case of any finite configuration, in particular $\overline\xi_{B(\rho^{n+1})}$. The definitions of  $\tau^\sharp(x,\overline\xi_{B(\rho^{n+1})})$ and $G^\sharp(x ; \overline\xi_{B(\rho^{n+1})})$ follow. The proof of the proposition is postponed in Section \ref{s_prop17}.

%\[
%\P[\text{There exists } n_0 \text{ such that for all } n\ge n_0 \text{ and all } x \in \xi, G^{\sharp\text{-fattened}}(x) \text{ does not cross } A_n(\rho)]=1
%\]

\subsubsection{Definition of $\sharp$-good points}
\label{s:p:33}

Let us define a new graph structure on the set of vertices $\xi$ by putting an edge between $x,y \in \xi$ when 
\[
G^{\sharp\text{-fat}}(x) \cap G^{\sharp\text{-fat}}(y) \neq \emptyset ~.
\]
In this case, we write $x \sim^\sharp y$. A $\sharp$-path is a (finite or not) path in this graph. In other words, this is a (possibly infinite) sequence $(x_i)$ of points of $\xi$ such that, for any index $i$, $x_{i-1} \sim^\sharp x_i$. Using the natural distance in this graph, we denote by $B^\sharp(x,k)$ the ball with center $x \in \xi$ and radius $k \geq 0$.

The next proposition which considers the case where $0$ belongs to the configuration, introduces the notion of \textbf{$(\alpha,\sharp)$-goodness} through the event $\good^\sharp(\alpha)$.

\begin{prop}
\label{p:DieseGoodImplyGood}
There exists a family of events $\good^\sharp(\alpha) \in \cF^{\overline \cS}$  indexed by $\alpha>0$ such that the following properties hold.
\begin{enumerate}
\item For all $\alpha>0$, $\good^\sharp(\alpha) \subset \{\overline S \in \overline\cS : 0 \in S\}$.
\item A.s. for all $\alpha>0$, any $x \in \xi$ stopped in a regular way (in $\overline\xi$) and such that $\overline\xi-x \in \good^\sharp(\alpha)$ satisfies $\overline\xi-x \in \good(\alpha)$.
\item For all $\alpha>0$, the event $\good^\sharp(\alpha)$ only depends on $\{(y,g_y), y \in B^\sharp(0,4) \}$.
\item  A.s. for any $x \in \xi$, there exists $\alpha_0>0$ such that for all $0<\alpha < \alpha_0$, $\overline\xi-x \in \good^\sharp(\alpha)$.
\end{enumerate}
\end{prop}

An element $x \in \xi$ is an \textbf{$(\alpha,\sharp)$-good point} if $\overline\xi-x$ belongs to $\good^\sharp(\alpha)$. Otherwise, $x$ is an \textbf{$(\alpha,\sharp)$-bad point}. The statements of Proposition \ref{p:DieseGoodImplyGood} can thus be paraphrased as follows:
\begin{itemize}
\item A.s. any $(\alpha,\sharp)$-good point stopped in a regular way is an $\alpha$-good point.
\item Whether $x \in \xi$ is $(\alpha,\sharp)$-good or not only depends on $\{(y,g_y), y \in B^\sharp(x,4)\}$.
\item A.s. any point of $\xi$ is $(\alpha,\sharp)$-good for $\alpha$ small enough.
\end{itemize}

The proof of the proposition is postponed to Section \ref{s-5}. While it may be frustrating for the reader to encounter only an abstract definition of $\good^\sharp(\alpha)$ points at this stage, providing a precise definition now is impractical due to its technical nature, which involves extensive work, such as introducing scenarios, blurred lifetime functions, and more. Therefore, we have chosen to extract and present only the significant properties of $\good^\sharp(\alpha)$ points necessary to complete the proof of Theorem \ref{t}.

\subsection{Proof of Theorem \ref{t} using Propositions \ref{p:33} and \ref{p:DieseGoodImplyGood}}
\label{sect:ProofTh2From}

In the previous section, we successively reduced the proof of Theorem \ref{t} to the proof of the existence of a subcritical regime for forward percolation of $(\alpha,\sharp)$-bad points. In other words, it remains to prove the following lemma.

\begin{lemma}
\label{l:existence-phase-sous-critique-pour-sharp}
There exists $\alpha>0$ such that, with probability $1$, there exists no $x \in \xi$ such that $\forward(x)$ is infinite and contains only $(\alpha,\sharp)$-bad points.
\end{lemma}

\begin{proof}[Proof of Therorem \ref{t} using Lemma \ref{l:existence-phase-sous-critique-pour-sharp}]
Let $\alpha>0$ be as in the statement of the lemma. First remark that any element $y$ belonging to an infinite forward set is stopped in a regular way. We can then use the second item of Proposition \ref{p:DieseGoodImplyGood} to assert that if such $y$ is an $\alpha$-bad point, it is also an $(\alpha,\sharp)$-bad point. Hence, Lemma \ref{l:existence-phase-sous-critique-pour-sharp} implies that a.s. there exists no $x \in \xi$ such that $\forward(x)$ is infinite and contains only $\alpha$-bad points.

We conclude by invoking almost sure results stated previously.  
\begin{itemize}
\item By Lemma \ref{l:csq-forward-good-finite}, all forward clusters $\forward(x)$ are finite.
\item By Lemma \ref{l:reduction-forward}, all undirected clusters $\cluster(x)$ are finite.
\item By Lemma \ref{l:reduction-pog}, all the clusters $C(x)$ of the original undirected graph $\graphe(\overline\xi)$ are finite.
\end{itemize}
Note that in the first two bullet points, forward clusters and undirected clusters of the directed graph $\grapheoriente(\overline\xi)$ are considered while the third one concerns clusters of the original undirected graph $\graphe(\overline\xi)$. Theorem \ref{t} is proven.
\end{proof}

\begin{proof}[Proof of Lemma \ref{l:existence-phase-sous-critique-pour-sharp} using Propositions \ref{p:33} and \ref{p:DieseGoodImplyGood}] Let $\pi$ be a $\sharp$-path (recall the definition in Section \ref{s:p:33}). We say that $\pi$ is a $\sharp$-path of $(\alpha,\sharp)$-bad points if its all its vertices are $(\alpha,\sharp)$-bad points. We say that $\pi$ is self-avoiding if all its vertices are distinct.

Besides, two consecutive elements $x_i$ and $x_{i+1}$ in a forward cluster in $\grapheoriente(\overline\xi)$ satisfies $x_i \to x_{i+1}$. This means that there are different and $x_i \sim^\sharp x_{i+1}$ by Lemma \ref{l:le_modele_augmente_domine}. Hence, any infinite forward set is an infinite self-avoiding $\sharp$-path. So, in order to get Lemma \ref{l:existence-phase-sous-critique-pour-sharp}, it is sufficient to prove that:

\begin{claim}
\label{c:fastoche}
There exists $\alpha>0$ (small) such that a.s. there is no infinite self-avoiding $\sharp$-path of $(\alpha,\sharp)$-bad points.
\end{claim}

Given two subsets $A$ and $B$ of $\R^2$, we say that there exists a $\sharp$-path from $A$ to $B$ if there is a $\sharp$-path $(x_0,\ldots,x_k)$ such that $G^{\sharp\text{-fat}}(x_0) \cap A \neq \emptyset$ and $G^{\sharp\text{-fat}}(x_k) \cap B \neq \emptyset$. Let $\rho_0 > 1$ given by Proposition \ref{p:33} and pick $\rho \ge \max\{\rho_0 , 2\}$. For any $\alpha > 0$, $n \ge 1$ and $c \in \R^2$, let us consider the event
\[
\cross_\alpha(c,n) = \big\{ \text{$\exists$ a self-avoiding $\sharp$-path of $(\alpha,\sharp)$-bad points from $B(c,\rho^{10n})$ to $B(c,\rho^{10(n+1)})^c$} \big\} ~.
\]
By stationarity, it is sufficient-- to get Claim \ref{c:fastoche} --to prove that
\begin{equation}
\label{e:pcrosstendvers0}
\P[\cross_\alpha(0,n)] \to 0 \text{ as } n \to \infty
\end{equation}
for $\alpha > 0$ small enough. Indeed, the previous limit asserts that any self-avoiding $\sharp$-path of $(\alpha,\sharp)$-bad points starting from $c+[0,1]^2$ is a.s. finite, for any $c \in \R^2$, and Claim \ref{c:fastoche} immediatly follows.

The proof of (\ref{e:pcrosstendvers0}) relies on a multiscale analysis: see \cite{G-AOP08} for details. One key point is that the event $\cross_\alpha(c,n)$ can be localized on the following nice event:
\[
\nice(c,n) = \left\lbrace \begin{array}{c}
\text{for all $k \ge 10 n$ and all $x \in \xi$, $G^{\sharp\text{-fat}}(x)$ does not cross $A(c,k)$ and} \\
\text{for all $k \ge 10 n$ and all $x \in \xi \cap B(c,\rho^k)$, $G^\sharp(x) = G^\sharp(x;\overline\xi_{|B(c,\rho^{k+1})})$}
\end{array} \right\rbrace
\]
recalling that $A(c,k) = B(c,\rho^k)\setminus\interior{B}(c,\rho^{k-1})$. This is Lemma \ref{l:cross_local} below. Later, we will use that the event $\nice(c,n)$ has a probability tending to $1$ as $n$ tends to $\infty$ (uniformly on the center $c$ by stationarity) thanks to Proposition \ref{p:33}.

\begin{lemma}
\label{l:cross_local}
For any $\alpha > 0$, $c \in \R^2$ and $n \ge 1$, there exists an event $\crosslocal_\alpha(c,n)$, depending only on $\overline\xi_{|B(c,\rho^{10n+16})}$, such that
\[
\cross_\alpha(c,n) \cap \nice(c,n) = \crosslocal_\alpha(c,n) \cap \nice(c,n) ~.
\]
\end{lemma}

\begin{proof}[Proof of Lemma \ref{l:cross_local}]
Assume that $\nice(c,n)$ holds. In order to know whether $\cross_\alpha(c,n)$ holds or not, we proceed in two steps. First, we focus on all possible self-avoiding $\sharp$-path $(x_0,\dots,x_k)$ from $B(c,\rho^{10n})$ to $B(c,\rho^{10n+10})^c$. W.l.o.g. we can assume that $(x_0,\dots,x_k)$ is minimal in the following sense: $G^{\sharp\text{-fat}}(x_i)$ overlaps $B(c,\rho^{10n+10})^c$ only when $i=k$. Then, the fattened set $G^{\sharp\text{-fat}}(x_k)$ (and in particular the vertex $x_k$) is included in $B(c,\rho^{10n+11})$. Otherwise, it would touch $B(c,\rho^{10n+11})^c$ and $B(c,\rho^{10n+10})$ since by hypothesis $x_k$ is the unique element among the $x_i$'s such that $G^{\sharp\text{-fat}}(x_i) \cap B(c,\rho^{10n+10})^c \not= \emptyset$. This is forbidden when $\nice(c,n)$ holds. Consequently, $x_0,\ldots,x_k$ are all points of $\xi \cap B(c,\rho^{10n+11})$. On $\nice(c,n)$, the $\sharp$-grains $G^{\sharp}(x_i)$'s (and also the fattened sets $G^{\sharp\text{-fat}}(x_i)$'s) only depend on $\overline\xi_{|B(c,\rho^{10n+12})}$.

In a second step, we have to check wether each $x_i$ is an $(\alpha,\sharp)$-bad point or not. Thanks to the second item of Proposition \ref{p:DieseGoodImplyGood}, this only involves elements in $B^\sharp(x_i,4)$. Repeating the same argument as in the first step on the event $\nice(c,n)$, we prove that any vertex belonging to some $B^\sharp(x_i,4)$, for $0\leq i\leq k$, has to be inside $B(c,\rho^{10n+15})$. On $\nice(c,n)$, the $\sharp$-grains of vertices of $\xi \cap B(c,\rho^{10n+15})$ only depend on $\overline\xi_{|B(c,\rho^{10n+16})}$.
%\jb{Techniquement on aboutit à : sur $\nice$, on a suffisamment d'info dans $\overline\xi_{|B(c,\rho^{10n+16})}$ pour savoir si $\cross_\alpha$ est réalisé ; c'est bien ce qu'on veut ? peut-être faudrait-il expliciter $\crosslocal$ et donc $G^\sharp(x ; \xi_{|\cdots})$ ?}
\end{proof}

The next step toward the proof of (\ref{e:pcrosstendvers0}) consists in the following induction relation.

\begin{lemma}
\label{l:multiscale}
There exists a constant $C=C(\rho,d) \ge 1$ such that, for all $n \ge 1$ and all $\alpha > 0$,
\[
\P[\cross_\alpha(0,n+1)] \le C \P[\cross_\alpha(0,n)]^2 + C \P[\nice(0,n)^c] ~.
\]
\end{lemma}

\begin{proof}[Proof of Lemma \ref{l:multiscale}]
Let $\cC_1$ be a finite subset of $\cS(\rho^{10})$ and $\cC_2$ be a finite subset of $\cS(\rho^{20})$ such that 
\[
\cS(\rho^{10}) \subset \bigcup_{c \in \cC_1} B(c,1) \text{ and } \cS(\rho^{20}) \subset \bigcup_{c \in \cC_2} B(c,1).
\]
Denote by $K=K(\rho,d)$ the maximum of the cardinal of $\cC_1$ and $\cC_2$. For any $n \ge 1$, we immediately get
\[
\cS(\rho^{10(n+1)}) \subset \bigcup_{c \in \rho^{10n}\cC_1} B(c,\rho^{10n}) \text{ and } \cS(\rho^{10(n+2)}) \subset \bigcup_{c \in \rho^{10n}\cC_2} B(c,\rho^{10n})
\]
where $\rho^{10n}\cC_1$ (resp. $\rho^{10n}\cC_2$) is a subset of $\cS(\rho^{10(n+1)})$ (resp. $\cS(\rho^{10(n+2)})$) with cardinality at most $K$.

Assume that $\cross_\alpha(0,n+1)$ occurs and let us denote by $\pi = (x_0,\dots,x_k)$ a path realizing this event. W.l.o.g. we can assume that $\pi$ is minimal in the following sense: $G^\sharp(x_i)$ touches $B(0,\rho^{10(n+1)})$ (resp. $B(0,\rho^{10(n+2)})^c$) only when $i=0$ (resp. $i=k$). So, $G^{\sharp\text{-fat}}(x_0)$ touches $B(0,\rho^{10(n+1)})$ and its complement: it also touches $\cS(\rho^{10(n+1)})$. Therefore there exists $c_1 \in \rho^{10n}\cC_1$ such that $G^{\sharp\text{-fat}}(x_0)$ touches $B(c_1,\rho^{10n})$. Similarly there exists $c_2 \in \rho^{10n}\cC_2$ such that $G^{\sharp\text{-fat}}(x_k)$ touches $B(c_2,\rho^{10n})$. As both distances between $c_1$ and $\cS(\rho^{10(n+2)})$ and between $c_2$ and $\cS(\rho^{10(n+1)})$ are at least $\rho^{10(n+2)}-\rho^{10(n+1)} > \rho^{10(n+1)}$ (since $\rho \ge 2$), we get that both events $\cross_\alpha(c_1,n)$ and $\cross_\alpha(c_2,n)$ occur. To sum up,
\[
\cross_\alpha(0,n+1) \subset \bigcup_{c_1,c_2} \cross_\alpha(c_1,n) \cap \cross_\alpha(c_2,n)
\]
where $c_1 \in \rho^{10n}\cC_1$ and $c_2 \in \rho^{10n}\cC_2$.

Thus Lemma \ref{l:cross_local} leads to the following inclusion
\[
\cross_\alpha(0,n+1) \subset \Big( \bigcup_{c_1,c_2} \crosslocal_\alpha(c_1,n) \cap \crosslocal_\alpha(c_2,n) \Big) \cup \Big( \bigcup_{c_3} \nice(c_3,n)^c \Big) 
\]
where $c_1, c_2$ are as before and where $c_3 \in \rho^{10n}\cC_1 \cup \rho^{10n}\cC_2$. By Lemma \ref{l:cross_local} and for $i=1,2$, the event $\crosslocal_\alpha(c_i,n)$ is measurable w.r.t. $\overline\xi_{B(c_i,\rho^{10n+16})}$. So they are independent since $\|c_2-c_1\| \ge \rho^{10(n+2)}-\rho^{10(n+1)} > 2 \rho^{10n+16}$ as $\rho \ge 2$. Therefore, the previous inclusion yields
\[
\P[\cross_\alpha(0,n+1)] \le \sum_{c_1,c_2} \P[\crosslocal_\alpha(c_1,n)]\P[\crosslocal_\alpha(c_2,n)] + \sum_{c_3} \P[\nice(c_3,n)^c] ~.
\]
By stationarity, we can write 
\[
\P[\cross_\alpha(0,n+1)] \le K^2 \P[\crosslocal_\alpha(0,n)]^2 + 2K \P[\nice(0,n)^c] ~.
\]
Using Lemma \ref{l:cross_local} once again, we get $\P[\crosslocal_\alpha(0,n)] \le \P[\cross_\alpha(0,n)] + \P[\nice(0,n)^c]$ and then
\[
\P[\cross_\alpha(0,n+1)] \le K^2 \P[\cross_\alpha(0,n)]^2 + (3K^2+2K) \P[\nice(0,n)^c] ~.
\]
The lemma follows with $C=3K^2+2K$.
\end{proof}

Let us conclude the proof of (\ref{e:pcrosstendvers0}). Let $C \geq 1$ be the constant given in Lemma \ref{l:multiscale}. Since the event $\nice(0,n)^c$ has a vanishing probability as $n \to \infty$ by Proposition \ref{p:33}, we can pick $n_0$ large enough such that, for all $n \ge n_0$, 
\begin{equation}
\label{e:nice-1-4}
C \P[\nice(0,n_0)^c] \le C^2 \P[\nice(0,n)^c] \le \frac 1 4 ~.
\end{equation}

As a consequence of the fourth item of Proposition \ref{p:DieseGoodImplyGood}, any point in a given bounded set is $(\alpha,\sharp)$-good with probrobablity tending to $1$ as $\alpha$ tends to $0$. Hence, pick $\alpha > 0$ small enough such that 
\[
C \P\big[\text{$\exists x \in \xi \cap B(0,\rho^{10n_0+16})$ such that $x$ is $(\alpha,\sharp)$-bad} \big] \le \frac 1 4.
\]

The argument used in the first step of Lemma \ref{l:cross_local} says that the conjunction of both events $\cross_\alpha(0,n_0)$ and $\nice(0,n_0)$ ensures the existence of a $(\alpha,\sharp)$-bad point in $B(0,\rho^{10n_0+11})$. Using what preceedes, we get:
\begin{align}
\label{e:cross-init-1-2}
C \P[\cross_\alpha(0,n_0)]
 & \le C \P\big[ \text{$\exists x \in \xi \cap B(0,\rho^{10n_0+16})$ such that $x$ is $(\alpha,\sharp)$-bad} \big]+ C \P[\nice(0,n_0)^c]  \nonumber \\
 & \le \frac 1 2 ~.
\end{align}
Henceforth, combining \eqref{e:nice-1-4}, \eqref{e:cross-init-1-2} and the induction relation
\begin{equation}
\label{e:multiscale}
C \P[\cross_\alpha(0,n+1)] \le \big( C \P[\cross_\alpha(0,n)] \big)^2 + C^2 \P[\nice(0,n)^c]
\end{equation}
providing by Lemma \ref{l:multiscale} (for all $n \ge n_0$), we easily show by induction that $C \P[\cross_\alpha(0,n)] \le \frac 1 2$ for all $n \ge n_0$. In other words,
\[
0 \le \limsup_n C \P[\cross_\alpha(0,n)] \le \frac 1 2 ~.
\]

Besides, from \eqref{e:nice-1-4} and \eqref{e:multiscale} we also deduce that
\[
\limsup_n C \P[\cross_\alpha(0,n)] \le \big( \limsup_n C \P[\cross_\alpha(0,n)] \big)^2 ~.
\]
Putting together the two previous displays, we finally get \eqref{e:pcrosstendvers0}. Lemma \ref{l:existence-phase-sous-critique-pour-sharp} is proven.
\end{proof}

\section{Proof of Proposition \ref{p:33}: study of the $\sharp$-model}
\label{s_prop17}

\subsection{Local obstacles, pollution and strategy}

\paragraph{Rough strategy.} We mainly have to prove that augmented grains $G^{\sharp}(x)$ do not travel too far.
As augmented grains are stopped when they hit an obstacle, we have to prove a conveniently quantified version of "any long path touches an obstacle".
The difficulty lies in the fact that the definition of the event "there is an obstacle around point $c$" is not local.
Indeed, the second item in the definition requires that there is no $x \in \xi \setminus B(c,1)$ such that $G_1(x)$ touches $B(c,1)$.
This leads to long range dependence.
To deal with that difficulty we introduce two objects:
\begin{itemize}
\item The event "there is a local obstacle around $c$". This is a local event. They are therefore easily dealt with.
\item The pollution $\Sigma$. This is a random subset of $\R^2$ such that, for any point $c$ out of $\Sigma$, "there is a local obstacle around $c$" if and only if "there is
an obstacle around $c$".
\end{itemize}
The difficulty is thus concentrated in the study of the pollution $\Sigma$.
We are led to show that on any long path one can find many points outside of $\Sigma$.
This is related, at least in spirit, to percolation and first passage percolation.
We make a few comment on this point in Section \ref{sect:Comments}. In the remaining of this section, we give precise definitions.

\paragraph{$V_0$-local obstacle.} We define a notion of local obstacle. This notion depends on a positive real number $V_0>0$.
We also write $V_0^+=V_0+1$.
For any $x \in \xi$ we define
\[
V(x) = \sup_{s \in [0,1]} \|g_x(s)-g_x(0)\| = \sup_{s \in [0,1]} \|h_x(s)\| \quad \text{ and } \quad V^+(x)=V(x)+1.
\]
In the line-segment model introduced in Section \ref{sect:Line-Segment}, $V(x)$ denotes the speed of growth of the associated grain. We keep this interpretation as a speed in the general case and speak of low-speed or high-speed grains $x$ depending on whether $V(x) < V_0$ or not. Recall the discussion at the beginning of Section \ref{sect:DieseBadPoints}, the moment assumption of Theorem \ref{t} can be rewritten with $t=1$. Hence, in the whole of Section \ref{s_prop17}, we assume
\[
\E \big[ V^2 \big] = \E \big[ \sup_{s \in [0,1]} \|h(s)\|^2 \big] < \infty
\]
or equivalently $\E[ (V^+)^2] < \infty$.

Recall also $\cL(1)$ and $r(1)$ from the definition of the loop property at scale $1$.
Let us first reformulate the definition of obstacle as stated in Section \ref{s:sharp-model}.
There is an obstacle around $c \in \R^2$ if the following conditions hold:
\begin{enumerate}
\item $\overline\xi - c \in \cL(1)$.
\item For all $x \in \xi \setminus B(c,1)$, if $V(x) < V_0$, then $G_1(x) \cap B(c,1) = \emptyset$.
\item For all $x \in \xi \setminus B(c,1)$, if $V(x) \ge V_0$, then $G_1(x) \cap B(c,1) = \emptyset$.
\end{enumerate}
We say that there is a local obstacle around $c \in \R^2$ if the first two of the above properties are true.
Let us make a few elementary remarks.
\begin{itemize}
\item If there is an obstacle around $c$, then there is a local obstacle around $c$.
\item By definition of the loop property, the probability of having an obstacle around a point is positive. The same therefore holds for the probability of having a local
obstacle around a point.
\item Property 1 only depends on $\overline \xi_{|B(c,1)}$. Property 2 only depends on $\xi_{|B(c,V_0^+)}$.
Indeed any grain $x$ such that $G_1(x) \cap B(c,1) \neq \emptyset$ satisfies $B(x,V(x)) \cap B(c,1) \neq \emptyset$.
If the grain is moreover low-speed, then $B(x,V_0) \cap B(c,1) \neq \emptyset$ and thus $x$ belongs to $B(c, V_0^+)$.
Therefore, 
\begin{equation}\label{e:local_obstacle_local}
\text{The event "there is a local obstacle at $c$" is measurable with respect to $\xi_{|B(c,V_0^+)}$.}
\end{equation}
\item If there is no high-speed grain $x \in \xi$ such that $B(x,V(x))$ touches $B(c,1)$
-- which is equivalent to $c \in B(x,V^+(x))$ --, then there is a local obstacle around $c$ if and only if there is an obstacle around $c$.
\end{itemize}

\paragraph{$V_0$-pollution and strategy.} The latter remark motivates the definition of the following object.
We define a Boolean model by
\[
\Sigma = \Sigma(V_0) = \bigcup_{x \in \xi : V(x) \ge V_0} B(x,V^+(x)).
\]
We can reformulate the remark ending the previous paragraph as follows:
\begin{equation} \label{e:local_paslocal_obstacle}
\text{If $c \in \R^2 \setminus \Sigma$, then there is a local obstacle around $c$ if and only if there is an obstacle around $c$.}
\end{equation}
%
%The aim is basically to show that augmented grains $G^\sharp(x)$ are unlikely to cross large annuli. 
%The plan is roughly the following. 
%\begin{enumerate}
%\item On any path crossing a large annulus one can find, with high probability, many points outside the $V_0$-polluted set $\Sigma(V_0)$ that are at least at distance $2V^+_0$
%apart from each other.
%\item By locality of "local obstacle", one then shows that, with high probability, there are some local obstacles at some of the above points.
%\item As those point are outside $\Sigma(V_0)$, the above local obstacles are actually obstacles.
%If the path considered is the trajectory $t \mapsto g_x(t)$ of some grain $x$, this thus prevent the augmented grain $G^\sharp(x)$ from crossing  the annulus.
%%\end{enumerate}
%
%The difficult point is the first one.

\paragraph{Plan of the section.} 
We fix some constants in Section \ref{s:choice_constant}.
The key result on pollution, Proposition \ref{p:lestnsontsympas}, is stated and proven in Section \ref{s:pollution_control}.
Its main consequence, which is essentially the first item of Proposition \ref{p:33}, is then stated and proven in Section \ref{s:augmented_grains_do_not_cross_large_annuli}.
The conclusion of the proof of Proposition \ref{p:33} is then given in Section \ref{s:p_finale_33}.

\subsection{Choice of constants for Sections \ref{s:pollution_control} and \ref{s:augmented_grains_do_not_cross_large_annuli}}
\label{s:choice_constant}
We fix $\rho$ and $V_0$ large enough for Sections \ref{s:pollution_control} and \ref{s:augmented_grains_do_not_cross_large_annuli}.
In Section \ref{s:p_finale_33} we do not use $V_0$ any more and $\rho$ become a free parameter again.

\paragraph{Scale parameter $\rho$.} We fix $\rho=100$.

\paragraph{Speed parameter $V_0$.} In order to choose $V_0$, we need to introduce a further event.
We say that there is a local low-speed obstacle around $c \in \R^2$ if the following conditions hold:
\begin{enumerate}
\item $\overline\xi^{<V_0} - c \in \cL(1)$ where $\overline\xi^{<V_0}=\{(x,h_x) \in \overline\xi : V(x) < V_0\}$.
\item For all $x \in \xi \setminus B(c,1)$, if $V(x) < V_0$, then $G_1(x) \cap B(c,1) = \emptyset$.
\end{enumerate}
The advantage of this new event is that it is measurable with respect to $\overline\xi^{<V_0}$.
If all the grains of $\xi \cap B(c,1)$ are low-speed, then there is a local low-speed obstacle around $c \in \R^2$ if and only if there is a local obstacle around $c \in \R^2$.
Recall that the probability of having a local obstacle raound $c$ is positive.
Therefore, for $V_0$ large enough, the probability of having a local low-speed obstacle around $c$ is positive.

Let now $C$ be the absolute constant define in \eqref{e:def_C}. Fix $V_0$ large enough to ensure
\begin{equation} \label{e:locallowspeedpositive}
\P[\text{local low-speed obstacle at }0]>0
\end{equation}
and there exists an integer $n_0$ such that
\begin{equation} \label{e:V0etnetqueuedelaloi}
\rho^{n_0-1} < V_0^+ \text{ and } 6V_0^+ < \rho^{n_0-1}(\rho-1) \text{ and } C^2 \int_{(\rho^{n_0-2},+\infty)} (v^+)^2 \d \P_{V^+}(v^+) \le \frac 1 4.
\end{equation}
We can have \eqref{e:locallowspeedpositive} thanks to the discussion of the previous paragraph.
We can also have \eqref{e:V0etnetqueuedelaloi} as $\E[(V^+)^2]$ is finite and as $\rho-1 = 99 > 6$.

\subsection{Pollution control}
\label{s:pollution_control}

Recall that $\rho=100$ and that $V_0$ has been chosen in Section \ref{s:choice_constant}.

\subsubsection{Notations, result and plan}

Let $n \ge 1$.
We consider the annulus 
\[
A_n = B(\rho^n)\setminus \interior B(\rho^{n-1})
\]
and the Boolean model 
\[
\Sigma_n =  \bigcup_{x \in \xi : V_0^+ \le V^+(x) \le \rho^{n-1}} B(x,V^+(x)).
\]
Let us recall the definition of the pollution $\Sigma$:
\[
\Sigma = \bigcup_{x \in \xi : V_0 \le V(x)} B(x, V^+(x)) = \bigcup_{x \in \xi : V_0^+ \le V^+(x)} B(x, V^+(x)).
\]
Thus $\Sigma_n$ is a local version of the pollution $\Sigma$.

To any deterministic continuous path $\pi$ crossing $A_n$ (that is starting from the inner boundary of $A_n$, ending on the outer boundary of $A_n$
and remaining within the annulus $A_n$ in the meantime) we associate the score
\[
S_n(\pi) = \max \{\card(W) : W \subset \pi, W \subset \Sigma^c(V_0), W \text{ is separated}\}
\]
and the score
\[
S'_n(\pi)= \max \{\card(W) : W \subset \pi, W \subset \Sigma^c_n(\rho,V_0), W \text{ is separated}\}
\]
where "$W$  is separated"  means that any two distinct points of $W$ are at least at distance $3V_0^+$ from each other and that each point is at least
at distance $3V_0^+$ from the boundaries of the annulus.

Finally, we introduce
\[
T_n= \inf_{\pi \text{ crosses } A_n} S_n(\pi)
\]
and
\[
T'_n=\inf_{\pi \text{ crosses } A_n} S'_n(\pi).
\]
The latter quantity is local:
\begin{equation}\label{e:Tprimelocal}
T'_n \text{ is measurable with respect to } \xi_{|B(2\rho^n)}.
\end{equation}
Indeed $T'_n$ depends only on $A_n \cap \Sigma_n$, and the radius of balls considered in the definition of $\Sigma_n$ is bounded from above by $\rho^{n-1}$ (and thus by $\rho^n$).

Let us collect a few results in the following lemma.

\begin{lemma}\label{l:basic_W}
\begin{itemize}
\item If $W$  is separated, then the events "$w$ is a local obstacle" indexed by $w \in W$ are independent. 
\item If $W \subset \Sigma^c$, then for any $w \in W$,  "$w$ is a local obstacle" if and only if "$w$ is an obstacle". 
\end{itemize}
\end{lemma}

\begin{proof} The first item is a consequence of \eqref{e:local_obstacle_local}. The second idem is \eqref{e:local_paslocal_obstacle}.  \end{proof}

Our aim is to show that any crossing of an annulus touches an obstacle with high probability.
Thanks to the previous lemma, the task is mainly reduced to showing that $T_n$ is large.
This motivates the following key result about pollution.

\begin{prop}
\label{p:lestnsontsympas} 
With probability $1$, for any $n$ large enough, $T_n \ge (4/3)^n$.
\end{prop}

The lower bound $(4/3)^n$ is negligible compare with the width of the annulus $A_n$ which is of order $100^n$.
The lower bound is thus sub-linear.
This is however sufficient for our purpose.
The proposition is an immediate consequence of the following results and of Borel-Cantelli lemma.

\begin{lemma} \label{l:encoredumultiechelle} There exists $A>0$ such that $\sum_n \P[T'_n < (3/2)^n A] < \infty$.
\end{lemma}

\begin{lemma} \label{l:easy} With probability $1$, for any $n$ large enough, $T'_n= T_n$.
\end{lemma}

\begin{proof}[Proof of Lemma \ref{l:easy}]
Let $n \ge 0$.
If $T'_n \neq T_n$, then there exists $x \in \xi$ such that
\[
V^+(x) \ge \rho^{n-1} \text{ and } B(x,V^+(x)) \cap B(\rho^n) \neq \emptyset
\]
and therefore such that
\[
\|x\| \le (1+\rho)V^+(x).
\]
But the expected value of the number of $x \in \xi$ fulfilling the previous condition is $\pi(1+\rho)^2\E[(V^+)^2]$, using the Campbell's formula, which is finite as $V^+=V+1$ and $\E[V^2]$ is finite. Therefore, almost surely, the number of such $x$ is finite. The lemma follows.
\end{proof}

The proof of Lemma \ref{l:encoredumultiechelle} is given in the next section.

\subsubsection{Proof of Lemma \ref{l:encoredumultiechelle}}

We need the following lemma, which relates the behavior at scales $n$ and $n-1$.

\begin{lemma} \label{l:encoredumultiechelle_unpas}  
There exists an absolute constant $C$ -- given by \eqref{e:def_C} -- such that, for all $n \ge 1$ and all $A>0$,  
\[
\P[T'_n < (3/2)A] 
\le 
C\P[T'_{n-1} < A]^2
+
C\int_{(\rho^{n-2},\rho^{n-1}]} (v^+)^2 d \P_{V^+}(v^+).
\] 
\end{lemma}
\begin{proof} Let $n \ge 1$.
Consider the event
\[
\cD(n) = \{ \text{there exists }x \in \xi
\text{ such that }
V^+(x) \in (\rho^{n-2},\rho^{n-1}]
\text{ and } 
B(x,V^+(x)) \text{ touches } B(\rho^{n})\}.
\]
On the complement of this event, we have
\begin{equation}\label{e:ploufplouf}
\Sigma_n \cap B(\rho^n) = \Sigma_{n-1}\cap B(\rho^n).
\end{equation}
Recall that we set $\rho = 100$.
Fix three finite sets $L_1, L_2$ and $L_3$ such that, for all $i \in \{1,2,3\}$,
\[
L_i \subset \partial B(20i) \text{ and } \partial B(20 i) \subset \bigcup_{w \in L_i} B(w,\rho^{-1}).
\]
We define an absolute constant $C$ by
\begin{equation}\label{e:def_C}
C=\max(\card(L_1)\card(L_2)+\card(L_1)\card(L_3)+\card(L_2)\card(L_3), \pi(1+\rho^2)^2).
\end{equation}
Fix $n \ge 1$. For all $i\in \{1,2,3\}$ we have
\[
\rho^{n-1}L_i \subset \partial B(20i\rho^{n-1}) 
\text{ and } 
\partial B(20 i\rho^{n-1}) \subset \bigcup_{w \in \rho^{n-1}L_i} B(w,\rho^{n-2}).
\]
The key is the inclusion
\begin{equation} \label{e:key-encoredumultichelle}
\{T'_n < (3/2)A\} \subset \cD(n) \cup 
\bigcup_{w_1 \in \rho^{n-1}L_1}
\bigcup_{w_2 \in \rho^{n-1}L_2}
\bigcup_{w_3 \in \rho^{n-1}L_3}
\{T_{n-1}^{'w_1} + T_{n-1}^{'w_2} + T_{n-1}^{'w_3}<(3/2)A\}
\end{equation}
where $T_{n-1}^{'w}$ is the event $T'_{n-1}$ "centered at $w$" (in other words this is a translate of $T_{n-1}$).
Let us prove \eqref{e:key-encoredumultichelle}. We work on $\{T'_n < (3/2)A\} \setminus \cD(n)$. Consider a path $\pi$ which realizes the minimum in the definition of $T'_n$. Hence,
\[
T'_n = S'_n(\pi) \ge \card(W)
\]
for any subset $W$ such that $W \subset \pi$, $W \subset \Sigma^c_n(\rho,V_0)$ and $W$ is separated. Let us construct a 'suitable' set $W$. For each $i$ choose $w_i \in \rho^{n-1} L_i$ such that $B(w_i,\rho^{n-2})$ contains a point of $\pi$. Then consider for each $i$ some restrictions $\pi_i$ of $\pi$ crossing the corresponding annuli. One can for example define $\pi_i$ as follows:
\[
a_i = \sup \{ t : \pi(t_i) \in B(w_i,\rho^{n-2})\}  \text{ and } b_i = \inf \{t \ge a_i : \pi(t_i) \not\in B(w_i,\rho^{n-1})\}
\]
and then set $\pi_i = \pi_{|[a_i,b_i]}$. See Fig. \ref{fig:Anneaux}. Now define for each $i$ a set $W_i$ which realizes the maximum in the definition of $S'_{n-1}(\pi_i)$ (with a definition for an annulus centered at $w_i$). So, for any $i$,
\[
\card(W_i) = S'_{n-1}(\pi_i) \geq T^{'w_i}_{n-1} ~.
\]
Finally, set $W = W_1 \cup W_2 \cup W_3$. This set enjoys the following properties.
\begin{enumerate}
\item $\card(W)=\card(W_1)+\card(W_2)+\card(W_3)$. Indeed each $W_i$ in included in $B(w_i,\rho^{n-1})$ and these balls are disjoint.
\item $W \subset \pi$.
\item $W \subset \Sigma_n^c(\rho,V_0)$. This is due to \eqref{e:ploufplouf}.
\item $W$ is $V_0$-separated. This is a consequence of the fact that each $W_i$ is $V_0$-separated (which implies in particular that each point of $W_i$ is at
distance at least $3V_0^+$ of the complement of $B(w_i,\rho^{n-1})$), the fact that the $W_i$ are included in the disjoint balls $B(w_i,\rho^{n-1})$
and the fact that these balls are included in the annulus at scale $A_n$.
\end{enumerate}
Therefore 
\begin{eqnarray*}
T'_n = S'_n(\pi) \ge \card(W) & = & \card(W_1)+\card(W_2)+\card(W_3) \\
& = & S'_{n-1}(\pi_1)+S'_{n-1}(\pi_2)+S'_{n-1}(\pi_3) \ge T^{'w_1}_{n-1} + T^{'w_2}_{n-1} + T^{'w_3}_{n-1} ~.
\end{eqnarray*}
This establishes \eqref{e:key-encoredumultichelle}.

\begin{figure}[!ht]
\begin{center}
\includegraphics[width=10cm,height=5cm]{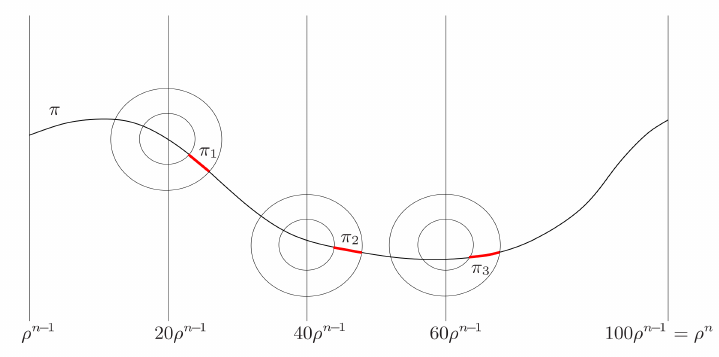}
\caption{\label{fig:Anneaux} On this picture, circular arcs of $\partial B(\rho^{n-1})$, $\partial B(20 \rho^{n-1})$ etc. are represented with vertical lines. The (black) path $\pi$ crosses the annulus $A_n = B(\rho^n)\setminus \interior B(\rho^{n-1})$. When $\pi$ crosses the smaller annuli $w_i+A_{n-1}$ for $i=1,2,3$, it provides the three (red) paths $\pi_1, \pi_2, \pi_3$.}
\end{center}
\end{figure}

If the sum of three numbers is smaller than $(3/2)A$, then the sum of the two smallest of them is smaller than $A$.
Indeed, if $a \le b \le c$ are such that $a+b+c \le (3/2)^n$, then $(a+b)/2 \le c$ and thus $(a+b)+(a+b)/2 \le a+b+c \le (3/2)A$ from which we get $a+b \le A$.
From this simple remark (which is the reason why we need the three layers $L_1, L_2$ and $L_3$ in \eqref{e:key-encoredumultichelle})
and \eqref{e:key-encoredumultichelle} we get, by union bound, by \eqref{e:Tprimelocal} and by stationarity  
\[
\P[T'_n < (3/2)A] \le \P[\cD(n)] + C \P[T'_{n-1} < A]^2
\]
where $C$ is defined by \eqref{e:def_C}.

Finally, using the Campbell's formula,
\begin{align*}
\P[\cD(n)]
& \le \E[\card(x \in \xi : V^+(x) \in (\rho^{n-2},\rho^{n-1}] \text{ and } B(x,V^+(x)) \text{ touches } B(\rho^{n}))] \\
& = \E[\card(x \in \xi : V^+(x) \in (\rho^{n-2},\rho^{n-1}] \text{ and } \|x\| \le V^+(x)+\rho^{n})] \\
& \le \E[\card(x \in \xi : V^+(x) \in (\rho^{n-2},\rho^{n-1}] \text{ and } \|x\| \le (1+\rho^2)V^+(x))] \\
& = \pi(1+\rho^2)^2 \int_{(\rho^{n-2},\rho^{n-1}]} (v^+)^2 \d \P_{V+}(v^+).
\end{align*}
Plugging this inequality in the previous one and using the definition of $C$ by \eqref{e:def_C} ends the proof.
\end{proof}

\begin{proof}[Proof of Lemma \ref{l:encoredumultiechelle}] Let $C$ be the constant given by Lemma \ref{l:encoredumultiechelle_unpas}.
%Recall $\rho=100$.
Fix $n_0$ such that \eqref{e:V0etnetqueuedelaloi} holds.
As $\rho^{n_0-1} < V_0^+$, $\Sigma_{n_0}=\emptyset$.
For any path $\pi$ crossing $A_{n_0}$, as $6V_0^+ < \rho^{n_0-1}(\rho-1)$, we thus get $S'_{n_0}(\pi) \ge 1$ by considering for $W$ the singleton containing 
some point of $\pi$ in the middle of $A_n$ (that is at distance $\rho^{n_0-1} + \rho^{n_0-1}(\rho-1)/2$ of $0$).
Thus $T'_{n_0}  \ge 1$.
Set $A=(3/2)^{-n_0}$.
We get
\begin{equation}\label{e:init_csq1}
\P[T'_{n_0}<(3/2)^{n_0} A] = 0.
\end{equation}
By Lemma \ref{l:encoredumultiechelle_unpas} we have, for all $n \ge 1$,
\begin{equation}\label{e:rec}
C\P[T'_n < (3/2)^nA] 
\le 
C^2\P[T'_{n-1} < (3/2)^{n-1} A]^2
+
C^2\int_{(\rho^{n-2},\rho^{n-1}]} (v^+)^2 d \P_{V^+}(v^+).
\end{equation}
From \eqref{e:V0etnetqueuedelaloi}, \eqref{e:init_csq1} and \eqref{e:rec} we first get, by induction, that for all $n \ge n_0$,
\[
C\P[T'_n < (3/2)^n A] \le \frac 1 2.
\]
From \eqref{e:rec} we then get, for all $n \ge n_0+1$,
\[
C\P[T'_n< (3/2)^nA] 
\le 
\frac 1 2 C\P[T'_{n-1} < (3/2)^{n-1} A]
+
C^2\int_{(\rho^{n-2},\rho^{n-1}]} (v^+)^2 d \P_{V^+}(v^+).
\]
Summing over $n$ we obtain, for all $k \ge n_0$,
\[
\sum_{n =n_0+1 }^{k+1}C \P[T'_n < (3/2)^n A] 
\le 
\frac 1 2 \sum_{n = n_0}^k C\P[T'_n < (3/2)^n A] 
+ C\int_{(\rho^{n_0-1},\rho^k]} (v^+)^2 d \P_{V^+}(v^+)
\]
and then (using \eqref{e:init_csq1})
\[
\sum_{n =n_0+1 }^kC \P[T'_n < (3/2)^n A] 
\le 
\frac 1 2 \sum_{n = n_0+1}^k C\P[T'_n < (3/2)^n A] 
+ C\int_{(\rho^{n_0-1},+\infty)} (v^+)^2 d \P_{V^+}(v^+)
\]
and thus
\[
\sum_{n =n_0+1 }^kC \P[T'_n < (3/2)^n A] 
\le 
2C\int_{(\rho^{n_0-1},+\infty)} (v^+)^2 d \P_{V^+}(v^+) < \infty.
\]
The lemma follows by letting $k$ go to $\infty$. \end{proof}

\subsubsection{Comments on some related results}
\label{sect:Comments}

We review here some results which are related to the first part of the plan (Proposition \ref{p:lestnsontsympas}) but which are not sufficient to prove it, explaining that we have to develop new strategies here.

We can consider $\Sigma(V_0)$ from the percolation point of view. As $\E[V^2]$ is finite, for $V_0$ large enough, all the connected components of $\Sigma(V_0)$ are bounded with probability one. Actually, the following stronger result is true. For all $V_0$ large enough, for all $\rho>1$,
\begin{equation}
\label{e:pollution-perco}
\P[\text{there exists a path in $\Sigma(V_0)$ from $B(0,r)$ to the complement of $B(0,\rho r)$}] \to 0 \text{ as }r \to \infty.
\end{equation}
In words, with high probability, we can find at least one point outside $\Sigma(V_0)$ on any path crossing a large annulus. Those results are implicit in \cite{G-AOP08}. They are not sufficient to establish the first part of the plan which requires many points and qualitative results (on the number of points and on the probability).

We can also consider $\Sigma(V_0)$ from the first passage percolation point of view. In \cite{GT-EJP17}, the following model is considered. Let $T(x,y)$ be the minimum time needed to go from $x \in \R^2$ to $y \in \R^2$ by a traveler who walks at speed $\infty$ inside $\Sigma(V_0)$ and at speed $1$ outside $\Sigma(V_0)$. By standard techniques in first passage percolation, one can show  
\[
\frac{T(0,x)}{\|x\|} \to \mu \text{ a.s. as } \|x\| \to \infty
\]
where $\mu$ is a deterministic constant. As consequence of the main result of \cite{GT-EJP17} and of \eqref{e:pollution-perco}, $\mu$ is positive for $V_0$ large enough under a condition which is slightly stronger than $\E[V^2]<\infty$. This is closer to what we need here but it is not sufficient to prove the first part of the plan as it gives no good probabilistic estimates.

%The difficulty is to get some qualitative estimates. However there is no need (as in \cite{GT-EJP17}) to get some linear estimates \jb{mouais ?} on the number of points etc. De fait, on va se contenter d'une croissance sous-linéaire en $r^\alpha$ pour un $\alpha \in (0,1)$.

\subsection{Augmented grains do not cross large annuli}
\label{s:augmented_grains_do_not_cross_large_annuli}

Recall that $\rho=100$ and that $V_0$ has been chosen in Section \ref{s:choice_constant}.
For $n \ge 1$, consider the following bad event:
\[
\text{Bad}_n = \{\text{there exists $x \in \xi$ such that $G^{\sharp\text{-fat}}(x)$ crosses the big annulus $ \bigcup_{i=0}^5 A_{n+i}$}\}.
\]
Let 
\[
N = \inf \{ n_0 \ge 1 : \forall n \ge n_0, \; \text{Bad}_n \text{ does not occur}\}.
\]

\begin{lemma}
\label{l:N3fini}
With probability one, $N$ is finite.
\end{lemma}

\begin{proof}[Proof of Lemma \ref{l:N3fini}]
Let
\[
N_1 = \inf \{n_0 \ge 1 : \forall n \ge n_0, T_n \ge (4/3)^n\}.
\]
By Proposition \ref{p:lestnsontsympas}, $N_1$ is finite with probability $1$.
This is our control on the influence of high-speed grains.
We say that a ball $B$ crosses an annulus $A_n$ if $B$ touches $\interior B(\rho^{n-1})$ and $B(\rho^n)^c$.
Let
\[
N_2 = \inf \{n_0 \ge 1 : \forall n \ge n_0, \text{ there exists no $x \in \xi$ such that } B(x,V^+(x)) \text{ crosses } A_n\}.
\]
If $n$ is such that $A_n$ is crossed by $B(x,V^+(x))$ for some $x \in \xi$, then $V^+(x) \ge \rho^{n-1}(\rho-1)/2 \ge \rho^{n-1}$ and $B(x,V^+(x)) \cap B(\rho^n)$ is non empty.
Arguing as in the proof of Lemme \ref{l:easy} we deduce that $N_2$ is finite with probability $1$.

If $k$ is an integer, we consider the event
\begin{align*}
\text{SmallBad}_k = \{ & \exists x \in \xi \cap B(\rho^{k-2}) \cup A_{k+2} \text{ such that }  g_x \text{ crosses } A_k \\
& \text{ and there is no obstacle at any point of } \text{Crossing}(g_x, A_k) \}
\end{align*}
where $\text{Crossing}(g_x, A_k)$ denotes the first crossing of the annulus $A_k$ by the path $g_x$,
that is the path $g_x$ restricted to the time interval $(s,t)$ where
\[
t = \inf \{u \ge 0 : \|g_x(u)\| = \rho^k \} \text{ and } s = \sup \{ u \le t : \|g_x(u)\| = \rho^{k-1}\}
\]
for a crossing from the inner to the outer and similarly for a crossing in the opposite direction.

\begin{claim} \label{c:baddecomposition} For all $n \ge 1$, $\{n \ge N_2\}  \cap  \text{Bad}_n    \subset \bigcup_{k \ge n+1} \text{SmallBad}_k$.
\end{claim}
\begin{proof}[Proof of Claim \ref{c:baddecomposition}]
Let $n \ge 1$. Assume that $\text{Bad}_n$ occurs. Let $x \in \xi$ be as in its definition. Assume also $n \ge N_1$ and $n \ge N_2$.
The idea is that, wherever $x$ is:
\begin{itemize}
\item $G^{\sharp-\text{fat}}(x)$ must cross three annuli toward $0$ or toward $\infty$.
\item $g_x([0,1])$ cross no annuli (because $n \ge N_2$).
\item The fattening $+B(3)$ is responsible for no annuli crossing.
\end{itemize}
As a consequence $g_x([1,\tau^\sharp(x)])$ must cross at least one annuli and there can be no obstacle on the crossing.

Let us detail for example the case $x \in  B(\rho^{n-1})$.
As $n \ge N_2$, $g_x([0,1])$ does not cross $A_n$ and therefore remains inside $B(\rho^n)$.
As $3$ is smaller than the width of all annuli (the minimal width is $\rho^{1-1}(\rho-1)=99)$) and as $G^\sharp(x)+B(3)$ touches $A_{n+3}$, $g_x([0,\tau^\sharp(x)])$ touches $A_{n+2}$.
In particular, $g_x$ crosses $A_{n+1}$ during the time interval $[1,\tau^\sharp(x)]$.
As a consequence (by definition of $\tau^\sharp(x)$) there is no point of $\cO$ \jb{ref} on $\text{Crossing}(g_x, A_{n+1})$. 
Therefore there is no obstacle on any point of $\text{Crossing}(g_x, A_{n+1})$ (otherwise this would contradict the previous statement by the separating property
of a loop and by the fact that a loop is contained in a ball of radius much smaller than the width of the annuli).
We thus get that $\text{SmallBad}_{n+1}$ holds.

The cases $x \in A_n$, $x \in A_{n+1}$, $x \in A_{n+2}$ are proven similarly. 
We get (according to the case considered) that $\text{SmallBad}_{n+2}, \text{SmallBad}_{n+3}$ or $\text{SmallBad}_{n+4}$ hold.
The cases $x \in A_{n+k}$ for some $k \ge 3$ are proven in the same way except that we look for a crossing from the outer to the inner.
We then get that $\text{SmallBad}_{n+k-2}$ holds.
\end{proof}
%
%
%From Claim \ref{c:baddecomposition} we get, for all $n_0 \ge 1$,
%\begin{equation}  \label{e:csq-baddecomposition}
%\{n_0 \ge N_2\} \cap \bigcup_{n \ge n_0} \text{Bad}_n
%\subset
%\bigcup_{\ell \ge n_0+1} \text{SmallBad}_\ell.
%\end{equation}

\begin{claim} \label{c:utilisation_bons_points} There exists $\alpha>0$ such that, for all $\ell \ge 2$, 
$
\P[\{\ell \ge N_1\} \cap \text{SmallBad}_\ell] \le \pi \rho^{2\ell+4} \exp(- \alpha (4/3)^\ell).
$
\end{claim}

\begin{proof}[Proof of Claim \ref{c:utilisation_bons_points}] 
Write $D = B(\rho^{\ell-2}) \cup A_{\ell+2}$ to shorten notations.
Then,
\begin{align*}
& \P[  \{\ell \ge N_1\} \cap \text{SmallBad}_\ell] \\
& \le 
 \E\left[  \sum_{x \in \xi \cap D} \1_{\{g_x \text{ crosses } A_\ell \text{ and there is no obstacle at any point of } \text{Crossing}(g_x, A_\ell)\}}\1_{\ell \ge N_1}\right] \\
&\le 
\E\left[  \sum_{x \in \xi \cap D} \1_{\{g_x \text{ crosses } A_\ell \text{ and there is no local obstacle at any point of } W(\text{Crossing}(g_x, A_\ell))\}}\1_{\ell \ge N_1}\right]
\end{align*}
where $W(\text{Crossing}(g_x, A_\ell))$ is a maximiser of $S(\text{Crossing}(g_x, A_\ell))$. The previous inequality is a consequence of $W(\text{Crossing}(g_x, A_\ell)) \subset \text{Crossing}(g_x, A_\ell)$, $W(\text{Crossing}(g_x, A_\ell)) \subset \Sigma^c$ and \eqref{e:local_paslocal_obstacle}.

By Slivnyak-Mecke formula, we then get
\begin{align*}
& \P[  \{\ell \ge N_1\} \cap  \text{SmallBad}_\ell] \\
& \le \int_D  \E\left[ \1_{\{g_x \text{ crosses } A_\ell \text{ and there is no local obstacle at any point of } W(\text{Crossing}(g_x, A_\ell))\}}\1_{\ell \ge N_1} \right] \d x
\end{align*}
where the underlying point process is $\{(x,g_x) \}\cup \overline\xi$, where $g_x=x+h$ and where $h$ is independent of $\overline\xi$
and distributed according to $\mu$. Write $\overline\xi$ as the disjoint union of $\overline\xi^{\ge V_0}$ (the high-speed grains) and $\overline\xi^{<V_0}$ (the low-speed grains) which can be assumed from each other:
\[
\overline\xi^{\ge V_0} = \{(y,g_y) \in \overline\xi : V(y) \ge V_0\} \text{ and } \overline\xi^{<V_0} = \{(y,g_y) \in \overline\xi : V(y) < V_0\}.
\]
Note that $W(\text{Crossing}(g_x, A_\ell))$, the family $(T_n)_n$ and $N_1$ are measurable with respect to $\cF := \sigma(g_x, \overline\xi^{\ge V_0})$.
Moreover, on $\{\ell \ge N_1\}$, we have $T_\ell \ge (4/3)^\ell$ and thus $W(\text{Crossing}(g_x, A_\ell)) \ge (4/3)^\ell$.
For $x \in D$, we have
\begin{align*}
& \E\left[ \1_{\{g_x \text{ crosses } A_\ell \text{ and there is no local obstacle at any point of } W(\text{Crossing}(g_x, A_\ell))\}}\1_{\ell \ge N_1} \right] \\
& = \E\left[ \E\left[ \1_{\{g_x \text{ crosses } A_\ell \text{ and there is no local obstacle at any point of } W(\text{Crossing}(g_x, A_\ell))\}}\1_{\ell \ge N_1} \left.\right | \cF \right] \right] \\ 
& \le \E\left[ \E\left[ \1_{\{g_x \text{ crosses } A_\ell \text{ and there is no local low-speed obstacle at any point of } W(\text{Crossing}(g_x, A_\ell))\}}\1_{\ell \ge N_1} \left.\right | \cF \right] \right] \\
& \le \exp(-\alpha (4/3)^\ell)
\end{align*}
where $\exp(-\alpha)$ is the probability of having no local low-speed obstacle at $0$ which is smaller that $1$ by \eqref{e:locallowspeedpositive}. The claim follows.
\end{proof}

From Claim \ref{c:utilisation_bons_points} and Borel-Cantelli Lemma we get that
\[
N_3 = \inf \{n_0 \ge 1 : \forall \ell \ge n_0, \{\ell \ge N_1\} \cap  \text{SmallBad}_\ell \text{ does not occur}\}
\]
is finite with probability $1$. Let $n \ge \max(N_1,N_2,N_3)$.  Let $\ell \ge n$.
As $\ell \ge n \ge \max(N_1,N_3)$, $\text{SmallBad}_\ell$ does not occur.
By Claim \ref{c:baddecomposition} we deduce that $\{n \ge N_2\} \cap  \text{Bad}_n$ does not occur.
But $n \ge N_2$. Therefore $\text{Bad}_n$ does not occur.
As a consequence, $N \leq \max(N_1,N_2,N_3)$ which is finite with probability $1$.
%
%
%Let $n_0 \ge 1$. We have
%\begin{align*}
% \{n_0 \ge & \max(N_1,N_2,N_3)\} \cap \bigcup_{n \ge n_0} \text{Bad}_n \\
%& \subset \{n_0 \ge \max(N_1,N_2,N_3)\} \cap \bigcup_{\ell \ge n_0+1}  \text{SmallBad}_\ell \qquad & \text{ by \eqref{e:csq-baddecomposition}} \\
%& \subset \{n_0 \ge N_3\} \cap \bigcup_{\ell \ge n_0+1} \{\ell \ge N_1\} \cap \text{SmallBad}_\ell  &
%\end{align*}
%which is empty by definition of $N_3$. Therefore $N \le \max(N_1,N_2,N_3)$ which is finite with probability~$1$.
\end{proof}

\subsection{Proof of Proposition \ref{p:33}}
\label{s:p_finale_33}

In this section $\rho$ becomes a free variable again (we do not fix $\rho=100$ anymore).
By Lemma \ref{l:N3fini}, $N$ is finite with probability $1$.
But for $n \ge N$, there exists no $x \in \xi$ such that $G^{\sharp-\text{fat}}(x)$ crosses the big annulus 
\[
\bigcup_{i=0}^5 A_{n+i}(100) = B(100^{n+5}) \setminus \interior B(100^{n-1}).
\] 
Therefore for any $\rho \ge 100^{13}$, with probability $1$, 
there exists $n_0$ such that for all $n \ge n_0$ and all $x\in\xi$, $G^{\sharp-\text{fat}}(x)$ does not cross $A_n(\rho)$.
%Indeed (considering the logarithm of the distance of a point to the origin) 
%any interval of length $13\ln(100)$ contains at least one interval of the form $[(n-1)\ln(100), (n+5)\ln(100)]$.
The first item of Proposition \ref{p:33} holds for any choice of $\rho_0$ larger than $100^{13}$.

Let $\rho \ge 100^{13}$. Let
\[
N' = \inf \{ n_0 \ge 1 : \forall n \ge n_0, \; G^{\sharp-\text{fat}}(x) \text{ does not cross } A_n(\rho)\}
\]
and
\[
N_2' = \inf \{n_0 \ge 1 : \forall n \ge n_0, \text{ there exists no $x \in \xi$ such that } B(x,V^+(x)) \text{ crosses } A_n(\rho)\}.
\]
We know by the first item that $N'$ is almost surely finite.
By the same argument as for $N_2$ \jb{ref} (we only change the value of $\rho$) we know that $N'_2$ is also almost surely finite.
Let $n \ge \max(N',N'_2)$ and $x \in \xi \cap B(\rho^n)$.
Let us check, for such $n$ and $x$,
\begin{equation}\label{e:lastdisplay}
G^\sharp(x ; \overline\xi) = G^\sharp(x ; \overline\xi_{|B(\rho^{n+3})}).
\end{equation}
This will establish the second idem -- and thus Proposition \ref{p:33} -- with $\rho_0 = (100^{13})^3$.

As $n \ge N'$ we have $G^\sharp(x) \subset B(\rho^{n+1})$.
Therefore, in order to prove \eqref{e:lastdisplay}, we only have to check that the obstacles touching $B(\rho^{n+1})$ are the same
in the configuration $\overline\xi$ and in the configuration $\overline\xi_{|B(\rho^{n+3})}$.
If $c \in \R^2 \setminus B(\rho^{n+1}+1)$ then, even if there is an obstacle at $c$, the obstacle does not touch $B(\rho^{n+1})$.
Therefore it is sufficient to consider $c \in B(\rho^{n+1}+1)$.
Let us inspect to two conditions defining "there is an obstacle at $c$" (we write it for $\overline\xi$) :
\begin{enumerate}
\item $\overline\xi - c \in \cL(1)$ : this only depends on $\overline\xi \cap B(c,1)$ and $B(c,1) \subset B(\rho^{n+3})$.
\item For all $x \in \xi \setminus B(c,1)$, $G_1(x) \cap B(c,1) = \emptyset$ : this only depends on $x \in \xi$ such that $B(x,V^+(x))$ touches $B(c,1)$
and (as $n \ge N'_2$) all such $x$ belongs to $B(\rho^{n+3})$.
\end{enumerate}
This proves \eqref{e:lastdisplay} and ends the proof. \qed

\section{Proof of Proposition \ref{p:DieseGoodImplyGood}: study of $\sharp$-good points}\label{s-5}
\label{sect:ProofDieseGood}

\paragraph{The blurred approach with scenarios.} We want to witness $\alpha$-goodness of $x_0$ by looking only in a limited neighborhood of the graph associated with the $\sharp$-model in Section \ref{s:p:33} (for short we will talk of $\sharp$-neighborhood). With so few information, we have no access to what really happens. For example we do not know $s(x_0)$ or $\tau(x_0)$. We therefore have to accept to lose some information. We develop what we call a blurred approach with scenarios in which we agree to lose information of two kinds:
\begin{itemize}
    \item We accept to lose the exact time at which some grain stops and, instead, keep only the knowledge of the interval of time in which the grain stops. We refer to this as blurring.
    \item We accept to only know that what happens is described (up to the above blurring) by one of a finite number of scenarios.
    Basically, a scenario consists of some meta-information giving the list of particles stopping in each interval of time of the above time discretization.
\end{itemize}
We emphasize the fact that some of these scenarios may have no sense (for example making a particle stop at some empty place). The point is that one of these scenarios corresponds to the true situation. Moreover, this blurred approach with scenarios is local and contains nevertheless enough information to witness $\alpha$-goodness. The basic idea is that $x_0$ is declared $(\alpha,\sharp)$-good if it is good in each of the scenarios. 
This is developed in Section \ref{sect:BlurredApproach}.

\paragraph{The dominant time.} The number of scenarios is finite partly because we only consider a finite number of particles (the ones in a relevant $\sharp$-neighborhood of $x_0$) and partly because we only consider what happens up to some finite time which we call the dominant time. We therefore need to define this dominant time in a local way. The first idea would be to define it as a (local) upper bound on all possible hitting times of the particles which may stop $x_0$ or be stopped by $x_0$. One could think that knowing the local configuration (up to blurring and scenarios) up to that time would be sufficient to place a loop stopping $x_0$ without interfering with its backward component. This is true up to the following subtlety: there may exist non stopped grains $G^{\sharp}(x)=G(x)$ with $\tau^{\sharp}(x)=\tau(x)=+\infty$ which are not compact and could be tangent to $G^{\sharp}(x_0)$. The presence of such pathological grains could make impossible the existence of a suitable small ball $B(v,2\alpha)$ avoiding all the grains other than $x_0$ in the final configuration. Overcome this second obstacle requires considering potential hitting times in a larger $\sharp$-neighborhood to make sure that, even if some grains other than $x_0$ are modified, this will not reduce the backward of $x_0$. See Section \ref{sect:Dominant}. 

%\jb{Version 2 : } The idea is to define it as a (local) upper bound on all possible hitting times of particles in a well chosen $\sharp$-neighborhood of $x_0$. See Section \ref{sect:Dominant}. 

\paragraph{Conclusion.}
Finally, Section \ref{sect:ProofProp15} contains the definition of $\sharp$-good points and the proof of Proposition \ref{p:DieseGoodImplyGood}, i.e. any point is $(\alpha,\sharp)$-good for $\alpha$ small enough, the notion of $\sharp$-goodness is local and any $(\alpha,\sharp)$-good point stopped in a regular way is an $\alpha$-good point.

\subsection{A blurred approach}
\label{sect:BlurredApproach}

In this section, we develop our blurred approach, which generally implies that the stopping character of some particles during an interval is indeterminate. We retain only the information that the particle is stopped at some point within this interval. Some information is lost, forcing us to work with a finite sequence of scenarios, but this enables us to achieve locality.

\paragraph{Setting for the whole Section \ref{sect:BlurredApproach}.} Recall that $x \sim^\sharp y$ means that $G^\sharp(x)+B(3)$ and $G^\sharp(y)+B(3)$ overlap. We denote by $B^\sharp(x,1)$ the ball with center $x$ and radius $1$ w.r.t. the graph distance associated to $\sim^\sharp$. Proposition \ref{p:33} asserts that any ball $B^\sharp(x,1)$ admits finitely many elements with probability $1$. In Section \ref{sect:ProofDieseGood}, the word \textit{locality} has to be understood in terms of $\sharp$-neighborhood.

For the whole Section \ref{sect:BlurredApproach}, we consider an element $x_0 \in \xi$ and we denote by $x_0,x_1,\ldots,x_K$ the elements of $B^\sharp(x_0,1)$. For the moment, no hypothesis is required about the grain $x_0$ (such as stopped in a regular way or with a finite lifetime $\tau(x_0)$).

\paragraph{Scenarios and modified dynamical algorithm.} Let $T_0 > 0$ be a real number. Let us select the integer $N > 0$ as the smallest integer such that 
\begin{equation}
\label{UnifCont}
\forall 0 \leq i \leq K \; \mbox{ and } \; \forall t,t' \in [0,T_0] \; \mbox{ with } \; |t-t'| \leq \frac{T_0}{N} , \; \mbox{ we have } \; | g_{x_i}(t) - g_{x_i}(t') | \leq \frac{1}{K+1} ~.
\end{equation}
Such an integer exists by uniform continuity of the $g_{x_i}$'s on the compact time interval $[0,T_0]$. Remark that $N$ only depends on $T_0$ and the vertices $x_0,x_1,\ldots,x_K$.

Let us set for any $k \in \{1,\dots,N\}$,
\[
I_k(T_0,N) = I_k := \Big[ (k-1) \frac{T_0}{N}, k \frac{T_0}{N} \Big) ~.
\]
Respectively to $x_0,x_1,\ldots,x_K$, $T_0$ and $N$, a \textbf{scenario} $\textbf{s}$ is a couple $(J , (I(j) : j \in J))$ where $J \subset \{0,1,\ldots,K\}$ and for any $j \in J$, $I(j)$ is taken among the intervals $I_1,\dots,I_N$. According to the scenario $\textbf{s}$, the grain $x_j$ is declared \textbf{blurred} if and only if $j \in J$. In this case, $g_{x_j}(I(j))$ represents its blurred part. Note that there are finitely many scenarios.

Given a scenario $\textbf{s} = (J , (I(j) : j\in J))$, we define the \textbf{modified dynamical algorithm} for the growth of the grains $x_0,x_1,\ldots,x_K$. The dynamics is the same as for the dynamical algorithm defined in Section \ref{sect:DynamicView} except for the following points: for any $0\leq i\leq K$,
\begin{itemize}
\item if $i \in J$ and if $x_i$ is still alive at time $t(i) := \inf I(i)$, then it is (artificially) stopped at time $t(i)$;
\item if $i \notin J$ and if $x_i$ is still alive at time $T_0$, then it is (artificially) stopped at time $T_0$;
\end{itemize}
In particular, any blurred part $g_{x_j}(I(j))$ is invisible to the other grains $x_i$, $i\not= j$: it cannot stop them. This modified dynamics could be formalized as we did in Section \ref{sect:DynamicView} by working successively on time intervals of length $T_0/N$ and adding the above extra rules. Note that, in the modified dynamics, a grain $x_i$ with $i\in J$ may be stopped before $t(i) = \inf I(i)$, in this case it does not reach its blurred part. However it can never be stopped after $t(i)$.

This modified dynamics, restricted to the grains $x_0,x_1,\ldots,x_K$ and to the time interval $[0,T_0]$, is realized according to the scenario $\textbf{s}$. It allows to associate to the scenario $\textbf{s}$ a \textbf{blurred configuration}, i.e.\ the union of the grains $x_0,\ldots,x_K$ produced by the modified dynamics with their blurred parts. We denote this union of blurred grains by $C(\textbf{s})$. Note also that the modified dynamics involves only the elements of $B^{\sharp}(x_0,1)$. A blurred configuration is thus locally determined.

\paragraph{Locality through admissible scenarios.} Respectively to $x_0,x_1,\ldots,x_K$ and $T_0$, a scenario $\textbf{s} = (J , \{I(i) : i\in J\})$ (or the corresponding blurred configuration) is said \textbf{admissible} when its set of blurred parts $\cup_{i\in J} g_{x_i}(I(i))$ does not overlap the set
\[
\textrm{Heart} := g_{x_0}([0,\tau^\sharp(x_0)]) + B(1) ~.
\]
As claimed previously, the true configuration of grains around $x_0$ (and till time $T_0$) is known up to a finite number of admissible scenarios.

\begin{lemma}
\label{lem:saucisse}
Using the previous notations, the following holds:
\begin{equation}
\label{saucisse}
\textrm{Heart} \cap \Big( \bigcup_{x \in \xi} g_x([0,\tau(x)\wedge T_0]) \Big) \, \in \, \big\{ \textrm{Heart} \cap C(\mathbf{s}) : \, \mbox{$\mathbf{s}$ is admissible} \big\}
\end{equation}
\end{lemma}

The left hand side of (\ref{saucisse}) is the true configuration of grains in \textrm{Heart} and till time $T_0$: this is a non-local information. However, the right hand side of (\ref{saucisse}) is a finite collection of blurred configurations which are all locally determined. Roughly speaking, to get locality we have weakened our knowledge of what really happens inside $\textrm{Heart}$ and until time $T_0$.\\

The rest of this section is devoted to the proof of Lemma \ref{lem:saucisse}.\\

\noindent
\textbf{Proof.} The proof is splitted into three steps.\\

\noindent
\textbf{Step 1 : Blurring algorithm.} Our first goal is to deteriorate the lifetime function $\tau$ (non-local information) into a \textbf{blurred lifetime function} $\tau^b$ (local information), leading to a scenario $\mathbf{s}_{\tau^b}$.

Let $x_0,x_1,\ldots,x_K, T_0, N$ as before. Given also a configuration $\overline S \in \overline\cS$, a blurred lifetime function $\tilde{\tau}$ is defined as a function 
\[
\tilde{\tau} : S \to [0,T_0] \cup \{I_1, \dots, I_N\} \cup \{\emptyset\} ~.
\]
Each of the possible values of $\tilde{\tau}(x)$, with $x \in S$, will be interpreted as follows:
\begin{itemize}
\item $\tilde{\tau}(x) \in [0,T_0]$: the grain $x$ stops growing at time $\tilde{\tau}(x)$.
\item $\tilde{\tau}(x) = I_k$ for some $k$: the grain $x$ stops growing at some \emph{unknown} time in the interval $I_k$. It is blurred.
\item $\tilde{\tau}(x) = \emptyset$: the grain $x$ is discarded.
\end{itemize}

Let $\overline S \in \overline\cS$ be a tempered configuration and let $\tau$ be its lifetime function. \textbf{The Blurring algorithm} allows to blur the lifetime function $\tau$ around the grain $x_0$ and until time $T_0$, leading to a blurred lifetime function denoted by $\tau^b$.

Start by setting $\tau^b = \tau$. We will then change the values of some of the $\tau^b(x)$ by performing the following steps. We first need some notations. If $t \in [0,T_0)$, $\blur(t)$ is defined as the unique interval $I_k$ containing $t$. We define the blurred region by
\[
\blurred(\overline S, \tau^b,x_0,T_0,N) := \big( g_{x_0}([0,\tau^{\sharp}(x_0)])+B(6) \big)^c \, \cup \Big( \bigcup_{x \in S : \, \tau^b(x) \in \{I_1, \dots, I_N\}} g_x(\tau^b(x)) \Big) ~.
\]

Now we perform the algorithm to deteriorate $\tau^b$. The first two steps below reduces the algorithm to grains of $B^{\sharp}(x_0,1)$ and to the time interval $[0,T_0]$. The third steps is repeated until the algorithm stop.

\begin{enumerate}
\item For all $x \notin B^{\sharp}(x_0,1)$, set $\tau^b(x) = \emptyset$.
\item For all $x \in B^{\sharp}(x_0,1)$ with $\tau(x) \geq T_0$, set $\tau^b(x) = T_0$.
\item Consider the set of all $x \in B^{\sharp}(x_0,1)$ such that $\tau^b(x)$ is a real number in $[0,T_0)$ and $g_x(\tau^b(x)) \in \blurred(\overline S, \tau^b,x_0,T_0,N)$.
\begin{itemize}
\item If this set is not empty:
\begin{itemize}
\item For each $x$ in this set, blur $\tau^b(x)$ by setting $\tau^b(x) = \blur(\tau^b(x))$.
\item Go back to Step 3.
\end{itemize}
\item If this set is empty, the algorithm stops.
\end{itemize}
\end{enumerate}

When Step 3 is realized for the first time, the blurred region is equal to $(g_{x_0}([0,\tau^{\sharp}(x_0)])+B(6))^c$ and at each new Step 3, it increases by adding blurred parts of grains, overlapping the current blurred region. Hence, when the blurred algorithm stops, any connected component of the set
\[
\bigcup_{i \leq K : \, \tau^b(x_i) \in \{I_1, \dots, I_N\}} g_{x_i}(\tau^b(x_i))
\]
overlaps $(g_{x_0}([0,\tau^{\sharp}(x_0)])+B(6))^c$. See Fig. \ref{fig:blur}.

Besides, the blurred lifetime function $\tau^b$ provides a unique scenario $\mathbf{s}_{\tau^b}$ to $x_0,x_1,\ldots,x_K$ and $T_0$, indicating which are the blurred grains (among the $x_i$'s) and what is their blurred part. More formally, $\mathbf{s}_{\tau^b} = (J,(I(j) : j \in J)$ with
\[
J=\{j \in \{0,\dots,K\} : \tau^b(x_j) \text{ is a time interval}\} \text{ and for all } j\in J, I(j)=\tau^b(x_j) ~.
\]

\begin{figure}[!ht]
\begin{center}
\includegraphics[width=11cm,height=5cm]{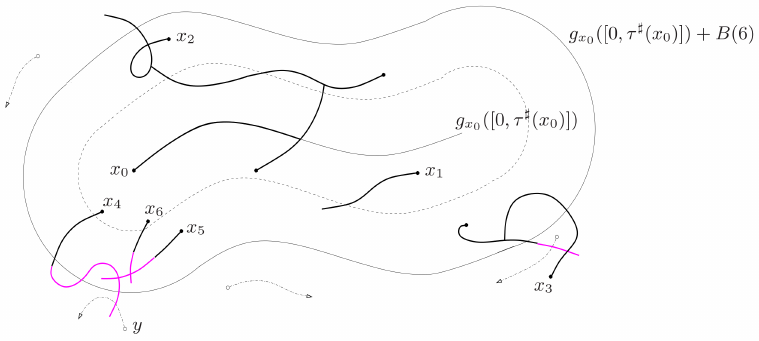}
\caption{\label{fig:blur} Here is the blurred configuration associated to the scenario $\mathbf{s}_{\tau^b}$, itself provided by the blurred lifetime function $\tau^b$. The dotted trajectories (as $y$) are some discarded grains, i.e. with $\tau^b(y) = \emptyset$, by Step 1 of the blurred algorithm. The pink pieces are the blurred parts of grains. Some grains, as $x_1$ and $x_2$, are stopped at time $T_0$ by Step 2. In the whole configuration $\overline\xi$, the grain $x_4$ is stopped before $T_0$ and outside the set $g_{x_0}([0,\tau^\sharp(x_0)])+B(6)$ by $y$. Because $y$ is discarded, this information is lost. The role of the blurred part of the grain $x_4$ is to encode this lack of information. Since, we do not know when $x_4$ is stopped (before or after crossing $x_5$?), this uncertainty is extended to $x_5$, then $x_6$ during Step 3.}
\end{center}
\end{figure}

\noindent
\textbf{Step 2.} Let us prove that the modified dynamics generates from the scenario $\mathbf{s}_{\tau^b}$, a blurred configuration which is consistent with the blurred lifetime function $\tau^b$. In other words, the blurred lifetime function implictly given by the modified dynamics with scenario $\mathbf{s}_{\tau^b}$ is $\tau^b$.

Since the list of discarded grains is the same for the scenario $\mathbf{s}_{\tau^b}$ and the blurred lifetime function $\tau^b$, we only need to focus on the grains $x_0,x_1,\ldots,x_K$. Let us consider the blurred configuration provided by the modified dynamics from the scenario $\mathbf{s}_{\tau^b}$ and the grains $(x_i,g_{x_i})$. Let us call the cleaned configuration the previous blurred configuration without its blurred parts. The key remark is that this cleaned configuration is also provided by the Dynamical algorithm (defined in Section \ref{sect:DynamicView}) from the modified grains $(x_i,f_{x_i})$ where the trajectories $f_{x_i}$ are defined as follows. For any index $i$ such that $\tau^b(x_i) = T_0$, we set $t(i) := T_0$. For any $i$ such that $\tau^b(x_i) \in \{I_1,\ldots,I_N\}$, we set $t(i) := \inf \tau^b(x_i)$. In both cases, we define
\[
f_{x_i}(t) := g_{x_i}(t) \1_{t \leq t(i)} + g_{x_i}(t(i)) \1_{t > t(i)} ~,
\]
Otherwise, $f_{x_i} := g_{x_i}$. Now, the Reconciliation lemma (Lemma \ref{l_rec}) asserts that the Dynamical algorithm provides a unique grain configuration-- which is the cleaned configuration --with a unique lifetime function, say $\tau_0$. Hence, the lifetime function $\tau_0$ satisfies $\tau_0(x_i) = t(i)$ whenever $\tau^b(x_i)$ equals to $T_0$ or one of the intervals $I_1,\ldots,I_N$, and $\tau_0(x_i) = \tau^b(x_i)$ otherwise. Finally, adding the blurred parts to that unique grain configuration and the values $I_1,\ldots,I_N$ and $\emptyset$ to the lifetime function $\tau_0$, we just have proved that to the blurred configuration (associated to $\mathbf{s}_{\tau^b}$), it corresponds a unique blurred lifetime function which is $\tau^b$.\\

\noindent
\textbf{Step 3:} Let us prove that the scenario $\mathbf{s}_{\tau^b}$ is admissible. This means that the set of blurred parts
\begin{equation}
\label{BlurredParts}
\bigcup_{x \in B^\sharp(x_0,1) : \, \tau^b(x) \in \{I_1, \dots, I_N\}} g_x(\tau^b(x))
\end{equation}
does not overlap $\textrm{Heart} = g_{x_0}([0,\tau^\sharp(x_0)]) + B(1)$. Recall that the scenario $\mathbf{s}_{\tau^b}$ is associated to the blurred lifetime function $\tau^b$.

Thanks to (\ref{UnifCont}), each blurred part $g_{x_i}(\tau^b(x_i))$ is included in a ball with diameter $\frac{1}{K+1}$. Moreover each connected component $D$ of the set given in (\ref{BlurredParts}) is made up with at most $K+1$ such blurred parts (since each grain is blurred at most one time). So $D$ is included in a ball with diameter $1$. Combining with the fact that $D$ overlaps $(g_{x_0}([0,\tau^\sharp(x_0)]) + B(6))^c$ by construction of the Blurring algorithm, we get the result.\\

Let us now conclude. Recall that $C(\textbf{s}_{\tau^b})$ denotes the union of (possibly blurred) grains of the blurred configuration provided by the modified dynamics from the scenario $\textbf{s}_{\tau^b}$. Let us write:
\begin{equation}
\label{ConcBlur}
\textrm{Heart} \cap \Big( \bigcup_{x \in \xi} g_x([0,\tau(x)\wedge T_0]) \Big) \, = \, \textrm{Heart} \cap \Big( \bigcup_{x \in B^\sharp(x_0,1)} g_x([0,\tau(x)\wedge T_0]) \Big) \, = \, \textrm{Heart} \cap C(\mathbf{s}_{\tau^b}) ~.
\end{equation}
The second equality of (\ref{ConcBlur}) is due to Steps 2 and 3. The true configuration of grains around $x_0$ and described by the lifetime function $\tau$ is corrupted into a blurred configuration obtained from $\tau^b$ and corresponding to the scenario $\textbf{s}_{\tau^b}$ thanks to Step 2. Hence, the true configuration and the blurred one, i.e. $C(\mathbf{s}_{\tau^b})$, differ only on the set of blurred parts given in (\ref{BlurredParts}). By Step 3, the scenario $\mathbf{s}_{\tau^b}$ is admissible which means that these blurred parts are outside $\textrm{Heart}$. As a consequence, the true configuration and the blurred one coincide on $\textrm{Heart}$. This achieves the proof of Lemma \ref{lem:saucisse}. $\blacksquare$

\subsection{The dominant time $\mathbb{T}(x_0)$}
\label{sect:Dominant}

Up to a finite number of admissible scenarios, we know what happens around the grain $x_0$ until time $T_0$. 
The goal of this section is to ensure that it is possible to choose the time $T_0$ in a local way and large enough so that, up to time $T_0$, we have enough information to place a loop stopping $x_0$ without reducing its backward.
We call this time the \textbf{dominant time} $\mathbb{T}(x_0)$.
The basic idea is that $\mathbb{T}(x_0)$ depends on a suitable $\sharp$-neighbourhood of $x_0$ and provides an upper bound on all possible impact between grains in a relevant $\sharp$-neighbourhood of $x_0$.

% all impacts involving $x_0$ have occurred before $T_0$-- in order to properly locate the small ball $B(v,2\alpha)$ --but finite 
% to avoid the possible presence of non-stopped grains which could be not compact and tangent to $G^{\sharp}(x_0)$. 

Remark that for the whole Section \ref{sect:Dominant}, no hypothesis is required about the grain $x_0$ (such as stopped in a regular way or with a finite lifetime $\tau(x_0)$).

\begin{prop}
\label{prop:DominantTime}
Let $x_0 \in \xi$. Then there a.s. exists $\mathbb{T}(x_0) < \infty$, called the dominant time for $x_0$, such that:
\begin{itemize}
\item[$(i)$] Domination: For any $y \in B^{\sharp}(x_0,2)$ such that $\tau(y)<\infty$ then $\tau(y) \leq \mathbb{T}(x_0)$.
\item[$(ii)$] Locality: $\mathbb{T}(x_0)$ only depends on grains $(z,g_z)$ with $z\in B^{\sharp}(x_0,4)$.
\end{itemize}
\end{prop}

Note that Item $(i)$ above concerns not only the elements of $B^{\sharp}(x_0,1)$, i.e. the $x$'s such that $x \sim^\sharp x_0$, but all those in $B^{\sharp}(x_0,2)$. We will see why in Section \ref{sect:DieseGoodIsGood}. Determining the dominant time $\mathbb{T}(x_0)$ in a local way (Item $(ii)$) will be needed so that the $\sharp$-goodness is a local notion.

Proposition \ref{prop:DominantTime} is an immediate consequence of the next result.

\begin{lemma}
\label{lem:MajorTime}
Let $y \in \xi$. Then there a.s. exists $T(y) \in [0,+\infty]$ only depending on grains $(z,g_z)$ with $z \in B^{\sharp}(y,2)$ and such that:
\[
\tau(y) < +\infty \, \Rightarrow \, \tau(y) \leq T(y) < +\infty ~.
\]
\end{lemma}

\noindent
\textbf{Proof of Proposition \ref{prop:DominantTime}.} Let $x_0 \in \xi$. The set $\{ y \in \xi : y \in B^{\sharp}(x_0,2) \mbox{ and } T(y) < +\infty \}$ is a.s. finite by Proposition \ref{p:33}. If it is empty then we set $\mathbb{T}(x_0) := 0$. Otherwise
\[
\mathbb{T}(x_0) := \max \{ T(y) : \, y \in B^{\sharp}(x_0,2) \mbox{ and } T(y) < +\infty \}
\]
is well defined and finite. By Lemma \ref{lem:MajorTime}, it clearly depends only on grains $(z,g_z)$ with $z \in B^{\sharp}(x_0,4)$. Moreover, for any $y \in B^{\sharp}(x_0,2)$ with $\tau(y)<\infty$, Lemma \ref{lem:MajorTime} asserts that $\tau(y) \leq T(y)$ which is itself smaller than $\mathbb{T}(x_0)$ by definition. $\blacksquare$\\

\noindent
\textbf{Proof of Lemma \ref{lem:MajorTime}.} Let $y \in \xi$ such that $\tau(y) < +\infty$. By the Hitting property, there exists $z \in \xi$, $z \not= y$, such that $y$ hits $z$ (at time $\tau(y)$). But knowing precisely who is the stopping grain $z$ of $y$ is a non-local information. We only know that it belongs to $B^{\sharp}(y,1)$. We are going to prove:\\

\noindent
\textbf{Claim:} For all $z \in B^{\sharp}(y,1)$, there exists $T(y,z) \in [0,+\infty]$ depending only on grains of $B^{\sharp}(y,2)$ such that ``$y$ hits $z$'' implies $\tau(y) \leq T(y,z) < +\infty$.\\

We conclude easily from the Claim with setting
\[
T(y) := \sup \big\{ T(y,z) : z \in B^{\sharp}(y,1) \; \mbox{ and } \; T(y,z) < +\infty \big\} ~.
\]
First, the variable $T(y)$ only depends on grains of $B^{\sharp}(y,2)$. Moreover, if $\tau(y) < +\infty$ then $y$ hits some $z$ in $B^{\sharp}(y,1)$ and, by the Claim, the variable $T(y,z)$ is finite and bounds $\tau(y)$. The set over which the supremum $T(y)$ is taken is then non-empty. It is also finite by Proposition \ref{p:33}. So $\tau(y) \leq T(y,z) \leq T(y) < +\infty$.\\

Hence our goal is to prove the Claim. Let $z \in B^{\sharp}(y,1)$ and assume that $y$ hits $z$. Let us consider
\[
t_{z,y} := \inf \{ t\geq 0 : \, g_z(t) \in g_y(\R_+) \}
\]
which is finite since $y$ hits $z$. First remark that $t_{z,y} \leq \tau(z)$: the grain $z$ is still alive at time $t_{z,y}$. Otherwise $\tau(z) < t_{z,y}$ and $g_z([0,\tau(z)])$ would not overlap $g_y(\R_+)$ which would contradict that $y$ hits $z$. Two cases must be distinghuished depending on whether the infimum $t_{z,y}$ is reached or not. At a first stage, the reader may only focus on Case 1. The second case requires more work.\\

\textbf{Case 1:} The infimum $t_{z,y}$ is reached, i.e. $g_z(t_{z,y}) \in g_y(\R_+)$. Let us introduce the first passage time of the trajectory $g_y$ at $g_z(t_{z,y})$:
\[
T_{z,y}^{(1)} := \inf \{ t \geq 0 : \, g_y(t) = g_z(t_{z,y}) \} ~.
\]
By construction, $T_{z,y}^{(1)}$ is finite and reached. If $T_{z,y}^{(1)} > t_{z,y}$, i.e. the trajectory $g_z$ visits $g_z(t_{z,y})$ before $g_y$, then $\tau(y) \leq T_{z,y}^{(1)}$ by the Stopping property. On the contrary, if $T_{z,y}^{(1)} \leq t_{z,y}$, the trajectory $g_y$ visits $g_z(t_{z,y})$ before $g_z$. We then have $T_{z,y}^{(1)} \leq \tau(y)$. Otherwise $\tau(y) < T_{z,y}^{(1)} \leq t_{z,y} \leq \tau(z)$ would mean that the grain $y$ is stopped before arriving at $g_z(t_{z,y})$: $g_y([0,\tau(y)])$ and $g_z([0,\tau(y)])$ would be disjoint that would contradict the hypothesis ``$y$ hits $z$''. From $T_{z,y}^{(1)} \leq \tau(y)$ combined with $g_z(t_{z,y}) = g_y(T_{z,y}^{(1)})$, $t_{z,y} \leq \tau(z)$ and $T_{z,y}^{(1)} \leq t_{z,y}$, the Stopping property asserts that $t_{z,y} = \tau(z)$, i.e. the grain $z$ is stopped by $y$ at time $t_{z,y}$. The equality $\tau(z)=t_{z,y}$ yields, by definition of $t_{z,y}$, that if $y$ hits $z$ it happens at the point $g_z(t_{z,y})$. Therefore, ``$y$ hits $z$'' forces the grain $y$ to visit $g_z(t_{z,y})$ at some time $t \ge t_{z,y} $ to be stopped. As a conqequence, in this case,
\[
\tau(y) = T_{z,y}^{(2)} := \inf \{ t\geq t_{z,y} : \, g_y(t) = g_z(t_{z,y}) \} \, < \, +\infty ~.
\]
To sum up, in Case 1, the lifetime $\tau(y)$ is a.s. bounded by
\[
T(y,z) := T_{z,y}^{(1)} \1_{T_{z,y}^{(1)}>t_{z,y}} + T_{z,y}^{(2)} \1_{T_{z,y}^{(1)} \leq t_{z,y}}
\]
which is finite and only depends on grains $(z,g_z)$ and $(y,g_y)$.\\

\textbf{Case 2:} The infimum $t_{z,y}$ is not reached. In this case, we have $\tau(z) > t_{z,y}$. Otherwise $\tau(z) \leq t_{z,y}$ would mean that $g_z([0,\tau(z)])$ and $g_y(\R_+)$ are disjoint since $t_{z,y}$ is not reached, and this would contradict the fact that $y$ hits $z$. For the same reason, we know that $g_z([t_{z,y},\tau(z)])$ is not reduced to the singleton $\{g_z(t_{z,y})\}$, i.e. the grain $z$ continues to grow after time $t_{z,y}$. Actually, we need a local version of what preceedes:
\begin{equation}
\label{t+delta}
\begin{array}{c}
\mbox{A.s. $\exists \delta > 0$ only depending on grains $(z',g_{z'})$ with $z' \in B^{\sharp}(z,1)$} \\
\mbox{such that $\tau(z) \geq t_{z,y}+\delta$ and $g_z(t_{z,y}+\delta) \not= g_z(t_{z,y})$.}
\end{array}
\end{equation}

In what follows, we first prove Statement (\ref{t+delta}) by a blurring approach and thus how it allows to conclude. 

It could be possible that the grain $z$ remains stuck at $g_z(t_{z,y})$ before restarting. To take into account this possibility in our proof, we have to introduce the time $t'_{z,y}$ at which the trajectory $g_z$ leaves the point $g_z(t_{z,y})$: $t'_{z,y} := \sup\{t'\geq t_{z,y} : g_z([t_{z,y},t']) = \{g_z(t_{z,y})\} \}$. As $t_{z,y}$, we know that $t'_{z,y}$ is finite and strictly smaller than $\tau(z)$. At first reading, the reader may assume that $t'_{z,y} = t_{z,y}$.

Let us use a blurring approach around the grain $z$ until time $t'_{z,y}$ (instead of $x_0$ and $T_0$ w.r.t. Section \ref{sect:BlurredApproach}). By Proposition \ref{p:33}, the set $B^{\sharp}(z,1)$ is a.s. finite, say $z'_0=z, z'_1,\ldots, z'_K$ be its elements, with $K$ random and finite. Thus, by uniform continuity of the $g_{z'_i}$'s on $[0,t'_{z,y}]$, we choose an integer $N'>0$ such that, for any $0\leq i\leq K$ and for any $t,t'\in [0,t'_{z,y}]$ with $|t-t'|\leq t'_{z,y} / N'$, one has $| g_{x_i}(t) - g_{x_i}(t') | \leq (K+1)^{-1}$. The time interval $[0,t'_{z,y}]$ is then splitted into the disjoint intervals $I_1,\ldots,I_{N'}$. So, any scenario $\textbf{s} = (J , (I(j) : j\in J))$, where $J \subset \{0,1,\ldots,K\}$ and for any $j \in J$, $I(j)$ is taken among the $I_k$'s, generates a blurred configuration of grains $z'_0,z'_1,\ldots,z'_K$ using the modified dynamics. As before, the number of admissible scenarios is a.s. finite and, among them, there is the one corresponding to the blurred lifetime function $\tau^b$ (thanks to the choice of $N'$) and then to the true grain configuration inside $g_z([0,\tau^\sharp(z)])+B(1)$.

Let us consider an admissible scenario $\textbf{s}$ according to which the grain $z$ has been stopped at time $t'_{z,y}$ (artificially by the modified dynamics) and, in the corresponding blurred configuration, any other grain of $B^{\sharp}(z,1)$ is at positive distance from $g_z(t'_{z,y})$. Let $\textbf{S}$ be this set of scenarios. Recall that the true scenario $\textbf{s}_{\tau^b}$ belongs to $\textbf{S}$. Now, let us restart the growth of grains of $B^{\sharp}(z,1)$ which have been stopped at time $t'_{z,y}$ according to $\textbf{s}$. Then, by continuity of trajectories, there exists $\delta(\textbf{s}) > 0$ small enough such that $g_z([0,t'_{z,y} + \delta(\textbf{s})])$ does not intersect any of the sets $g_{z'}([0,t'_{z,y}+\delta(\textbf{s})])$ for $z'$ stopped at time $t'_{z,y}$. The parameter $\delta(\textbf{s}) > 0$ can be chosen small enough so that $g_z([0,t'_{z,y} + \delta(\textbf{s})])$ also does not intersect the grains stopped before $t'_{z,y}$ or blurred according to the scenario $\textbf{s}$ (since $\textbf{s}$ is admissible). Roughly speaking, according to the scenario $\textbf{s}$, the grain $z$ is still alive at time $t'_{z,y} + \delta(\textbf{s})$. By definition of $t'_{z,y}$, decreasing $\delta(\textbf{s}) > 0$ if it is needed, one can also assume that $g_z(t'_{z,y} + \delta(\textbf{s})) \not= g_z(t'_{z,y})$.

Let $\bar\delta := \min \delta(\textbf{s}) > 0$ where the minimum is taken over the finite set $\textbf{S}$. Then,
\[
\tau(z) \geq t'_{z,y} + \delta(\textbf{s}_{\tau^b}) \geq t'_{z,y} + \bar\delta = t_{z,y} + \delta
\]
where $\delta :=\bar\delta + t'_{z,y}-t_{z,y} > 0$  and
\[
g_z(t_{z,y} + \delta) = g_z(t'_{z,y} + \bar\delta) \not= g_z(t'_{z,y}) = g_z(t_{z,y}) ~.
\]
Moreover the construction of $\delta$ (and $\bar\delta$) only depends on time $t_{z,y}, t'_{z,y}$ and on grains $(z',g_{z'})$ with $z'\in B^{\sharp}(z,1)$.\\

It remains to explain how Statement (\ref{t+delta}) implies the result. Because the infimum $t_{z,y}$ is not reached and by definition of $\delta$, there exist $s_0,s_1$ such that $t_{z,y} < s_0 < t_{z,y}+\delta$ and $g_z(s_0) = g_y(s_1)$.  By continuity of trajectories $g_y, g_z$, we can also require that $s_1 > t_{z,y} + \delta$. Henceforth
\[
T_{z,y}^{(3)} := \inf \{ t \geq t_{z,y}+\delta : \, g_y(t) \in g_z((t_{z,y},t_{z,y}+\delta]) \} 
\]
is well defined and finite (it is also reached). Since the grain $z$ is still alive at time $t_{z,y}+\delta$, we necessarily have $\tau(y) \leq T_{z,y}^{(3)}$ by the Stopping property. Let us add that $T_{z,y}^{(3)}$ only depends on grains $(y,g_y)$ and $ (z,g_z)$, and on the parameter $\delta$, i.e. only on grains $z' \in B^{\sharp}(y,2)$. In conclusion, $T(y,z) := T_{z,y}^{(3)}$ satisfies the Claim. $\blacksquare$\\

\subsection{Proof of Proposition \ref{p:DieseGoodImplyGood}}
\label{sect:ProofProp15}

\subsubsection{All grains are $\sharp$-good points}
\label{sect:AllisDiese}

Let $x_0 \in \xi$. Since we only want to use local information, we have to forget the true configuration around $x_0$ and work with admissible scenarios and blurred configurations. So, let us use a blurred approach around the grain $x_0$ and until the dominant time $\mathbb{T}(x_0)$, involving grains of $B^{\sharp}(x_0,1)$: as before they are denoted by $x_0,x_1,\ldots,x_K$. We then pick the integer $N$ as in (\ref{UnifCont}) but this time w.r.t. the vertices $x_0,x_1,\ldots,x_K$ and the dominant time $\mathbb{T}(x_0)$: $N$ is the smallest integer such that 
\begin{equation}
\label{UnifCont2}
\forall 0 \leq i \leq K \; \mbox{ and } \; \forall t,t' \in [0,\mathbb{T}(x_0)] \; \mbox{ with } \; |t-t'| \leq \frac{\mathbb{T}(x_0)}{N} , \; \mbox{ we have } \; | g_{x_i}(t) - g_{x_i}(t') | \leq \frac{1}{K+1} ~.
\end{equation}

In the sequel, we focus our attention on the (finite) set $\textbf{S}$ of admissible scenarios (w.r.t. $x_0,x_1,\ldots,x_K$ and $\mathbb{T}(x_0)$) in which $x_0$ is stopped in a regular way at time $0 < t(\mathbf{s},x_0) \leq \tau^\sharp(x_0)$. Let us point out here that we do not assume that $x_0$ is stopped in a regular way in the true grain configuration or that $x_0$ has a finite lifetime ($\tau(x_0) = \tau^\sharp(x_0) = +\infty$ is allowed). The set $\textbf{S}$ could be empty. Besides, when $\textbf{S} \not= \emptyset$, either $x_0$ is not stopped in a regular way in the true grain configuration and then none of the scenarios in $\textbf{S}$ correspond to reality. Or $x_0$ is stopped in a regular way and then $\textbf{S}$ contains the scenario $\mathbf{s}_{\tau^b}$ provided by the blurred lifetime function $\tau^b$. However, with only local informations, we cannot distinguish these two alternatives, nor the the scenario $\mathbf{s}_{\tau^b}$ in $\textbf{S}$ if there is. This is why we focus our efforts on all the scenarios of the set $\textbf{S}$.

Let $\textbf{s} \in \textbf{S}$ such a scenario. Our strategy consists in determining a small ball $B(v,2\alpha)$ located \textit{just before} the impact of $x_0$ on its stopping grain and overlapping no other grains (for the blurred configuration corresponding to $\mathbf{s}$). To do it, we need extra notations. Let $i(\mathbf{s})$ be the impact point of $x_0$ on its stopping grain. Since $x_0$ is stopped in a regular way (for the scenario $\mathbf{s}$), it passes only one time through $i(\mathbf{s})$. Hence,
\[
t(\mathbf{s},x_0) = \inf\{ t : g_{x_0}(t) = i(\mathbf{s}) \} < +\infty ~.
\]
Since the scenario $\mathbf{s}$ is admissible, there is no blurred parts inside $g_{x_0}([0,t(\mathbf{s},x_0)])+B(1)$ included in $\textrm{Heart}$ since $t(\mathbf{s},x_0) \le \tau^\sharp(x_0)$. We can then read on the blurred configuration the impact points on $g_{x_0}([0,t(\mathbf{s},x_0)])$ created by some of the $x_i$'s. Each impact point on $g_{x_0}([0,t(\mathbf{s},x_0))) = g_{x_0}([0,t(\mathbf{s},x_0)])\!\setminus\!\{i(\mathbf{s})\}$ is created by a $x_i$ which is stopped in a regular way by $x_0$ and then belongs to the Backward set of $x_0$. Nothing prevents the occurrence of an impact point exactly at $i(\mathbf{s})$. But the corresponding stop is not regular by definition and the grain creating this impact point is not in the Backward set of $x_0$. See Fig. \ref{fig:AllDiese}. Let us denote by $i^-(\mathbf{s})$ the last impact point-- when the grain $x_0$ is browsed from time $0$ to $t(\mathbf{s},x_0)$ --on $g_{x_0}([0,t(\mathbf{s},x_0)))$. We set $i^-(\mathbf{s}) := x_0$ if there is no such impact points. Notice that $i^-(\mathbf{s})$ is necessarily different from $i(\mathbf{s})$, otherwise the grain creating this last impact point would not be stopped in a regular way by $x_0$ (or because $t(\mathbf{s},x_0)>0$ in the case where $i^-(\mathbf{s}) = x_0$). Let us finally define $t^-(\mathbf{s},x_0)$ as the last passage time before $t(\mathbf{s},x_0)$ of the grain $x_0$ at $i^-(\mathbf{s})$:
\[
t^{-}(\mathbf{s},x_0) := \sup \{ t < t(\mathbf{s},x_0) : g_{x_0}(t) = i^-(\mathbf{s}) \} ~.
\]
By construction, $i(\mathbf{s})$ does not belong to the compact set $g_{x_0}([0,t^-(\mathbf{s},x_0)])$: $i(\mathbf{s})$ is at positive distance to $g_{x_0}([0,t^-(\mathbf{s},x_0)])$. We can then choose
\[
v(\mathbf{s}) \in g_{x_0}([0,t(\mathbf{s},x_0))) \setminus g_{x_0}([0,t^-(\mathbf{s},x_0)]) \, \subset \, G^\sharp(x_0)
\]
(since $t(\mathbf{s},x_0) \leq \tau^\sharp(x_0)$) and $0 < \alpha < 1/2$ small enough such that $B(v(\mathbf{s}),2\alpha)$ avoids $g_{x_0}([0,t^{-}(\mathbf{s},x_0)])$ and $i(\mathbf{s})$ (Items $(a)$ and $(b)$ of Definition \ref{def:DieseGood} below).

\begin{figure}[!ht]
\begin{center}
\includegraphics[width=10.6cm,height=4.6cm]{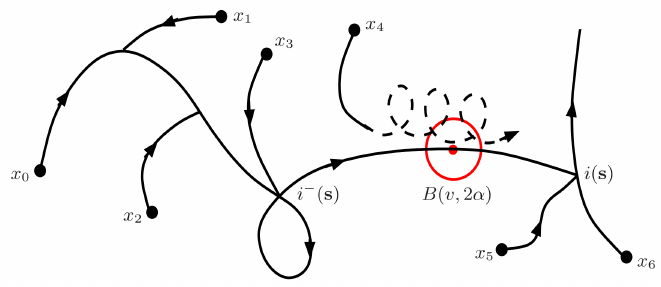}
\caption{\label{fig:AllDiese}On this picture $x_0$ is stopped in a regular way by $x_6$ at $i(\mathbf{s})$. $x_1$, $x_2$ and $x_3$ are all stopped by $x_0$ in a regular way, but not $x_5$ (actually $x_5$ is stopped in a regular way by $x_6$ and is in its Backward set). The small ball $B(v,2\alpha)$ (in red) has to be located between $i^-(\mathbf{s})$ and $i(\mathbf{s})$. Let us point out the presence of the grain $x_4$, with $\tau(x_4) = \infty$, which is non-compact and tangent to the grain $x_0$. By construction of the dominant time $\mathbb{T}(x_0)$, the piece of trajectory $g_{x_{4}}((\mathbb{T}(x_0),+\infty))$ (the dotted curve) is irrelevant and can then be ignored without reducing the backward set of $x_0$. This is why $B(v,2\alpha)$ may overlap $g_{x_{4}}((\mathbb{T}(x_0),+\infty))$.}
\end{center}
\end{figure}

For $1\leq i\leq K$, let $G(\mathbf{s},x_i)$ be the (possibly partially blurred) grain of $x_i$ in the blurred configuration corresponding to the scenario $\mathbf{s}$. By construction, $v(\mathbf{s})$ does not belong to the compact set $\cup_{i\leq K} G(\mathbf{s},x_i)$ (all these grains are compact sets since they are at the latest stopped at time $\mathbb{T}(x_0)$). Hence, $\alpha > 0$ can be chosen small enough so that $B(v(\mathbf{s}),2\alpha)$ avoids $\cup_{i\leq K} G(\mathbf{s},x_i)$. This is Item $(c)$ of Definition \ref{def:DieseGood}.

Since $x_0$ does not belong to $B(v(\mathbf{s}),2\alpha)$, the parameter $\alpha > 0$ can be further reduced so that $\inf\{t : g_{x_0}(t) \in B(v(\mathbf{s}),2\alpha)\}$ is larger than $\alpha$ (Item $(d)$ of Definition \ref{def:DieseGood}).

Clearly, the value $\alpha = \alpha(\textbf{s}) > 0$ previously obtained depends on the scenario $\textbf{s}$. Let us minimize $\alpha(\textbf{s})$ over the (finite) set of scenarios in $\textbf{S}$:
\[
\alpha^\ast := \min_{\textbf{s} \in \textbf{S}} \alpha(\textbf{s}) > 0 ~.
\]

In a last step, we know by Proposition \ref{p:33} that a.s. any $\sharp$-grain is bounded (but not necessarily stopped). Henceforth, the radius
\[
R^{\sharp}_{x_0} := \inf\{ R > 0 : \, G^{\sharp}_{x_0} \subset B(x_0,R) \}
\]
of the $\sharp$-grain $G^{\sharp}_{x_0}$ is a.s. finite. We can then choose $\alpha^\ast > 0$ sufficiently small such that $R^{\sharp}_{x_0} \leq 1/\alpha^\ast$.\\

We are now ready to introduce the notion of $\sharp$-good point.

\begin{definition}
\label{def:DieseGood}
Let $0<\alpha<1/2$ and $x_0 \in \xi$. Let us denote by $x_0,x_1,\ldots,x_K$ the elements of the (almost surely finite) set $B^{\sharp}(x_0,1)$. The point $x_0$ is said $(\alpha,\sharp)$-good if $R^{\sharp}_{x_0} \leq \alpha^{-1}$ and, for any scenario $\mathbf{s}$ (resp. to $x_0,x_1,\ldots,x_K$ and the dominant time $\mathbb{T}(x_0)$) in $\mathbf{S}$ (previously defined in this section), there exists a ball $B(v(\mathbf{s}),2\alpha)$ satisfying the following properties:
\begin{itemize}
\item[$(a)$] $v(\mathbf{s}) \in g_{x_0}([0,t(\mathbf{s},x_0))) \!\setminus\! g_{x_0}([0,t^-(\mathbf{s},x_0)]) \subset G^\sharp(x_0)$.
\item[$(b)$] $B(v(\mathbf{s}),2\alpha)$ avoids the set $g_{x_0}([0,t^{-}(\mathbf{s},x_0)]) \cup \{i(\mathbf{s})\}$.
\item[$(c)$] $B(v(\mathbf{s}),2\alpha)$ avoids all the (possibly blurred) grains $G(\mathbf{s},x_i)$ of the blurred configuration corresponding to the scenario $\mathbf{s}$.
\item[$(d)$] $\inf\{t : g_{x_0}(t) \in B(v(\mathbf{s}),2\alpha)\} \geq \alpha$.
\end{itemize}
Finally, we denote by $\good^\sharp(\alpha) \subset \{\overline S \in \overline\cS : 0 \in S\}$ the following event: $\overline\xi - x_0$ belongs to $\good^\sharp(\alpha)$ iff $x_0$ is an $(\alpha,\sharp)$-good point.
\end{definition}

In the case where the set $\mathbf{S}$ is empty, it becomes easy to be an $(\alpha,\sharp)$-good point since it suffices to check if $R^{\sharp}_{x_0} \leq \alpha^{-1}$. However, the arguments developed before Definition \ref{def:DieseGood} prove that any $x_0 \in \xi$ is $(\alpha^\ast,\sharp)$-good for the (random) value $\alpha^\ast = \alpha^\ast(x_0)$ previously obtained. This is the fourth item of Proposition \ref{p:DieseGoodImplyGood}:

\begin{prop}
\label{prop:AllIsDieseGood}
A.s. any $x_0 \in \xi$ is a $(\alpha,\sharp)$-good point for $\alpha$ small enough.
\end{prop}

Whereas the radius of the ball $B(v(\mathbf{s}),2\alpha)$ is the same for any scenario considered in Definition \ref{def:DieseGood} and is locally determined, it is not the case for its center $v(\mathbf{s})$, which may then depend on the scenario $\mathbf{s}$. This means that our strategy does not allow us to specify locally where the suitable ball $B(v,2\alpha)$ is located.

\subsubsection{The $\sharp$-goodness is a local notion}

Unlike the notion of $\alpha$-good point, the one of $(\alpha,\sharp)$-good point is local in the following sense. This is the third item of Proposition \ref{p:DieseGoodImplyGood}.

\begin{prop}
\label{prop:LocalIsDieseGood}
Let $0<\alpha<1/2$ and $x_0 \in \xi$. The property ``$x_0$ is an $(\alpha,\sharp)$-good point'' only depends on grains $(z,g_z)\in\overline{\xi}$ such that $z\in B^{\sharp}(x_0,4)$.
\end{prop}

\noindent
\textbf{Proof.} Let $x_0 \in \xi$. We develop a blurred approach respectively to the elements of $B^{\sharp}(x_0,1)$ and the dominant time $\mathbb{T}(x_0)$. By Proposition \ref{prop:DominantTime}, $\mathbb{T}(x_0)$ only depends on grains of $B^{\sharp}(x_0,4)$. The same holds for the choice of $N$ in (\ref{UnifCont2}) depending on $g_{x_i}$, with $x_i \in B^{\sharp}(x_0,1)$, and on $\mathbb{T}(x_0)$. This leads to a finite number of scenarios. For each scenario $\mathbf{s}$ in the set $\mathbf{S}$ (defined in the previous section), the locations $i(\mathbf{s}), i^-(\mathbf{s})$, the times $t(\mathbf{s},x_0), t^-(\mathbf{s},x_0)$ and the grains $G(\mathbf{s},x_i)$ (also defined in the previous section) only depend on grains of $B^{\sharp}(x_0,1)$ and $\mathbb{T}(x_0)$. These quantites allow to check if it is possible or not to find a ball $B(v(\mathbf{s}),2\alpha)$ satisfying Items $(a)-(d)$ of Definition \ref{def:DieseGood}. Finally, let us add that the radius $R^{\sharp}_{x_0}$ also depends on $B^{\sharp}(x_0,1)$. $\blacksquare$\\

\subsubsection{$\sharp$-good implies good}
\label{sect:DieseGoodIsGood}

Let us prove the second item of Proposition \ref{p:DieseGoodImplyGood}.

\begin{prop}
\label{prop:DieseGoodisGood}
Let $0 < \alpha < 1/2$. A.s. any $(\alpha,\sharp)$-good point stopped in a regular way is also an $\alpha$-good point.
\end{prop}

\noindent
\textbf{Proof.} Let $x_0 \in \xi$ which is \textbf{stopped in a regular way}. Respectively to the elements $x_0,x_1,\ldots,x_K$ of $B^\sharp(x_0,1)$ and the dominant time $\mathbb{T}(x_0)$, the Blurring algorithm deteriorates the lifetime function $\tau$ into the blurring lifetime function $\tau^b$ (see Section \ref{sect:BlurredApproach}). Associated to $\tau^b$, we consider the scenario $\mathbf{s}_{\tau^b}$ which is admissible thanks to the choice of $N$ in (\ref{UnifCont2}). Since the corresponding blurred configuration coincides with the true configuration inside $\textrm{Heart} = g_{x_0}([0,\tau^\sharp(x_0)]) + B(1)$ thanks to Lemma \ref{lem:saucisse}, then $x_0$ is still stopped in a regular way in the blurred configuration (w.r.t. $\mathbf{s}_{\tau^b}$) and  $t(\mathbf{s}_{\tau^b},x_0) = \tau(x_0) \leq \tau^\sharp(x_0) \wedge \mathbb{T}(x_0)$ (recall that $\tau(x_0) < \infty$ implies $\tau(x_0) \leq \mathbb{T}(x_0)$ by Proposition \ref{prop:DominantTime}). Hence, the scenario $\mathbf{s}_{\tau^b}$ belongs to the set $\mathbf{S}$ defined in Section \ref{sect:AllisDiese}. By assumption, $x_0$ is an $(\alpha,\sharp)$-good point, i.e. $R^{\sharp}_x\leq\alpha^{-1}$ and there exists a ball $B(v(\mathbf{s}_{\tau^b}),2\alpha)$ satisfying Items $(a)-(d)$ of Definition \ref{def:DieseGood}. We can first deduce the existence of a deterministic $u \in U(\alpha)$ such that
\[
v(\mathbf{s}_{\tau^b}) \in B(x_0+u , r(\alpha)) \subset B(x_0+u , \alpha) \subset B(v(\mathbf{s}_{\tau^b}),2\alpha) ~.
\]
We are going to prove that $x_0$ is an $(\alpha,u)$-good point (in $\overline\xi$) which means that $x_0$ is an $\alpha$-good point, then proving Proposition \ref{prop:DieseGoodisGood}.\\

Let us resample the PPP $\overline\xi$ in $B(x_0+u,\alpha)$. Let $\overline\xi'$ be an independent copy of $\overline\xi$ and set
\[
r_{x_0}(\overline \xi, \overline \xi') := \big( \overline \xi \cap B(x_0+u,\alpha)^c \big) \cup \big( \overline\xi' \cap B(x_0+u,\alpha) \big) ~.
\]
Moreover let us assume that $\overline\xi' \in x_0+u+\cL(\alpha)$: for the configuration $\overline\xi'\cap B(x_0+u,\alpha)$, a small loop is created inside $B(x_0+u,\alpha)$ and surrounding $B(x_0+u,r(\alpha))$. We have to investigate which grains are modified by this resampling; we inventory which grains of $\overline\xi$ may visit the ball $B(v(\mathbf{s}_{\tau^b}),2\alpha)$ and at what time such visits occur.

First remark that, by Item $(b)-(c)$, the ball $B(v(\mathbf{s}_{\tau^b}),2\alpha)$ avoids the point process $\xi$: no grains of $\overline\xi$ are deleted by the resampling. Hence, replacing $\overline\xi$ with $r_{x_0}(\overline \xi, \overline \xi')$ adds only some points inside $B(x_0+u,\alpha)$ (i.e. those of $\overline \xi' \cap B(x_0+u,\alpha)$). Let us gather various facts:
\begin{itemize}
\item Since $\overline\xi' \in x_0+u+\cL(\alpha)$, the new grains provided by $\overline\xi' \cap B(x_0+u,\alpha)$ are stopped before time $\alpha$ and remains completely included in $B(x_0+u,\alpha) \subset B(v(\mathbf{s}_{\tau^b}),2\alpha)$.
\item By Item $(d)$, the trajectory $g_{x_0}$ visits for the first time $B(v(\mathbf{s}_{\tau^b}),2\alpha)$ at time 
\[
t_0 := \inf \{ t : g_{x_0}(t) \in B(v(\mathbf{s}_{\tau^b}),2\alpha) \} \geq \alpha ~.
\]
\item By Item $(c)$, no grains of $B^{\sharp}(x_0,1)\!\setminus\!\{x_0\}$ may visit $B(v(\mathbf{s}_{\tau^b}),2\alpha)$ before the dominant time $\mathbb{T}(x_0) \geq \tau(x_0) > t_0$.
\item Since $\alpha < 1/2$, no grains (or $\sharp$-grains) of $B^{\sharp}(x_0,1)^c$ may visit $B(v(\mathbf{s}_{\tau^b}),2\alpha)$.
\end{itemize}
Hence all the grains of $\xi\!\setminus\!\{x_0\}$ evolve in the same way until time $\mathbb{T}(x_0)$ for both configurations $\overline \xi$ and $r_{x_0}(\overline \xi, \overline \xi')$. This has two consequences. First, for $r_{x_0}(\overline \xi, \overline \xi')$, nothing prevents grains of $\overline\xi' \cap B(x_0+u,\alpha)$ to form a loop (before time $\alpha$) separating $v(\mathbf{s}_{\tau^b})$ from $B(v(\mathbf{s}_{\tau^b}),2\alpha)^c$. Thus, nothing prevents the grain $x_0$ to grow until $t_0$, time at which it comes into the ball $B(v(\mathbf{s}_{\tau^b}),2\alpha)$, and to be stopped before $t_1 := \inf\{ t : g_{x_0}(t) = v(\mathbf{s}_{\tau^b})\} \leq \mathbb{T}(x_0)$ by the loop created by the grains of $\overline\xi' \cap B(x_0+u,\alpha)$. In other words, for the configuration $r_{x_0}(\overline \xi, \overline \xi')$, $x_0$ is followed by a loop.\\

Let us now check that $\# \back(x_0,\overline \xi) \leq \# \back(x_0,r_{x_0}(\overline \xi, \overline \xi'))$. Let us first recall that the piece of trajectory $g_{x_0}([0,t^-(\mathbf{s}_{\tau^b},x_0)])$ which gathers all the impacts of the grains of $\overline\xi$ whose $x_0$ is the successor, remains unchanged when $\overline\xi$ is replaced with $r_{x_0}(\overline \xi, \overline \xi')$ ($t^-(\mathbf{s}_{\tau^b},x_0)$ is defined in Section \ref{sect:AllisDiese}). Moreover, by Item $(b)-(c)$ and the construction of the dominant time, only the pieces of grains $g_{x_i}((\mathbb{T}(x_0),\infty))$, with $1\leq i\leq K$ and $\tau(x_i) = +\infty$ on the one hand and $g_{x_0}((t^-(\mathbf{s}_{\tau^b},x_0),\tau(x_0)))$ on the other hand, may overlap $B(v(\mathbf{s}_{\tau^b}),2\alpha)$ and could be possibly removed by the resampling. Let us set $W$ the union of these pieces of grains.

We claim that no grains in $\overline\xi\!\setminus\!\{x_0\}$ are stopped by the set $W$. Indeed, by absurd, assume that such a grain, say $z$, exists. First, no grains are stopped by $g_{x_0}((t^-(\mathbf{s}_{\tau^b},x_0),\tau(x_0)))$ by construction. So the grain $z$ is stopped by some $g_{x_i}((\mathbb{T}(x_0),\infty))$ which means $\tau(z) > \mathbb{T}(x_0)$. Besides, $z$ being stopped by one of the $x_i$'s, it necessarily belongs to $B^{\sharp}(x_0,2)$. Hence, $\tau(z) < +\infty$ implies $\tau(z) \leq \mathbb{T}(x_0)$ by Proposition \ref{prop:DominantTime}. This is our contradiction. Actually, this is the only place where we use that the dominant time $\mathbb{T}(x_0)$ involves all the grains of $B^{\sharp}(x_0,2)$, and not only those of $B^{\sharp}(x_0,1)$.

Hence, the (possible) deletion of any element in $W$ has no effect on the grains of $z \in \overline\xi\!\setminus\!\{x_0\}$ with $\tau(z)(\overline \xi) < +\infty$, which includes in particular the elements of $\back(x_0,\overline \xi)$. In other words the backward set of $x_0$ may only increase when the configuration $\overline \xi$ is replaced with $r_{x_0}(\overline \xi, \overline \xi')$. Let us specify that the only grains which could belong to $\back(x_0,r_{x_0}(\overline \xi, \overline \xi'))\setminus \back(x_0,\overline \xi)$ are the $x_i$'s with $i \leq K$ and $\tau(x_i) = +\infty$.\\

\subsection*{Acknowledgement}

\emph{DC, DD et J-BG are supported in part by the CNRS GdR 3477 GeoSto. DC is supported by the ANR project GrHyDy (ANR-20-CE40-0002). DD is also supported by the Labex CEMPI (ANR-11-LABX-0007-01) and the ANR project RANDOM (ANR-19-CE24-0014).}

\section{Appendix}

\subsection{Proof of Proposition \ref{p:existence}}
\label{s:proof-p:existence}

Let us consider a probability measure $\mu$ on $(\cC,\cF^{\cC})$ satisfying the moment condition (\ref{SuffTempered}): for all $t \ge 0$,
\[
\E \left[ \sup_{s \in [0,t]} \| h(s) \|^2 \right] < \infty
\]
where $h$ is a random variable with distribution $\mu$. Our goal is to define, for a.e. realization of $\overline \xi$, a function $\tau_{\dyn}^{\infty} : \xi \to (0,+\infty]$ specifying the lifetimes of all the grains of $\xi$. In a second time we will prove that the constructed function $\tau_{\dyn}^{\infty}$ is a.s. a $\overline \xi$-lifetime function and this is the only one, proving that $\mu$ is tempered.\\

The first step is the following technical result that we will proved later.

\begin{lemma}
\label{lem:subdivision}
Under the moment condition (\ref{SuffTempered}), for any $\alpha > 0$, there exists an increasing sequence $(t_n)_{n\geq 0}$ in $\R_+$ (depending on $\alpha$) starting at $t_0 = 0$ and tending to $\infty$ such that, for all $n$,
\[
\E \left[ \sup_{s \in [t_n,t_{n+1}]} \| h(s) - h(t_n) \|^2 \right] \leq \alpha ~.
\]
\end{lemma}

The key idea of the proof of Theorem 3.1 of \cite{CDlS-unbounded} is to subdivide $\R_+$ into small time intervals $[t_n,t_{n+1})$ during which the growth of grains can be compared to a subcritical Boolean model. Reducing the model with infinitely many grains to (infinitely many) finite clusters, i.e. each of them involving only a finite number of grains, will allow us to define step by step a lifetime function $\tau_{\dyn}^{\infty}$ (see further).\\

Let us first emphasize the comparison with subcritical Boolean models. Let us consider a Boolean model in $\R^2$ whose centers are given by an homogeneous Poisson point process $\zeta$ with intensity $1$ and i.i.d. radii $\{ R_x : x \in \zeta\}$:
\[
\Sigma := \bigcup_{x \in \zeta} B(x , R_x) ~.
\]
The Boolean model $\Sigma$ is known to be subcritical whenever $\E R^2$ is small enough (where $R$ denotes the common distribution of the $R_x$'s). Let $(t_n)_{n\geq 0}$ be the sequence given by Lemma \ref{lem:subdivision} and let us set
\[
R_{x,n} := \sup_{s \in [t_n,t_{n+1}]} \| g_x(s) - g_x(t_n) \| = \sup_{s \in [t_n,t_{n+1}]} \| h_x(s) - h_x(t_n) \| ~.
\]
Then, for any $n$, the random set
\[
\Sigma_n := \bigcup_{x \in \xi} B(g_x(t_n) , R_{x,n})
\]
is a Boolean model-- notice that $\{g_x(t_n) : x\in \xi\}$ is still a Poisson point process with intensity $1$ --which can be made subcritical by Lemma \ref{lem:subdivision} if $\alpha$ is chosen small enough.

Let us now build step by step a function $\tau_{\dyn}^{\infty} : \xi \to (0,+\infty]$. We proceed by induction with assuming that $\tau_{\dyn}^{\infty}$ has been already built until time $t_n$. I.e. for any grain $x \in \xi$, $\tau_{\dyn}^{\infty}(x) < t_n$ means that the grain $x$ has been stopped before time $t_n$ while $\tau_{\dyn}^{\infty}(x) = t_n$ indicates that $x$ is still alive at time $t_n$. Let us set $A_n := \{x \in \xi : \tau_{\dyn}^{\infty}(x) = t_n \}$. The random set 
\[
\Sigma_n' := \bigcup_{x\in A_n} B(g_x(t_n) , R_{x,n})
\]
is included in the Boolean model $\Sigma_n$ and then is subcritical; all its clusters are a.s. finite. Let $\mathcal{C}$ be one of them. So we can apply the dynamical algorithm introduced in Section \ref{sect:DynamicView} to the finite set of (alive) grains involved in $\mathcal{C}$ and during the time interval $[t_n,t_{n+1}]$. Let's be careful since for this step we also have to take into account the set of grains already realized till time $t_n$ and possibly important for the cluster $\mathcal{C}$, i.e. the set
\[
\{ x \in \xi : \, g_x([0,\tau_{\dyn}^{\infty}(x)]) \cap \mathcal{C} \not= \emptyset \}
\]
which could include grains already stopped at time $t_n$ or still alive but not in $\mathcal{C}$. Thanks to Lemma 4.2 of \cite{CDlS-unbounded} which only uses hypothesis (\ref{SuffTempered}), this set is a.s. finite and then the dynamical algorithm applies without extra difficulties. Henceforth, we treat independently each (finite) cluster of $\Sigma_n'$ which allows us to extend the construction of the function $\tau_{\dyn}^{\infty}$ until time $t_{n+1}$.\\

The last step consists in proving that the function $\tau_{\dyn}^{\infty}$ previously built is a.s. a $\overline \xi$-lifetime function and this is the only one. First consider $x\not= y\in \xi$ and $t_x\geq 0$ such that $x$ $\tau_{\dyn}^{\infty}$-hits $y$ at time $t_x$. Let $t_n$ be such that $t_x\in [t_n,t_{n+1})$ and consider the cluster $\mathcal{C}$ of the random set $\Sigma_n'$ (defined above) containing $x$-- with a slight abuse of notations, we use the same symbol $\mathcal{C}$ for the subset of $\Sigma_n'$ and the set of grains involved in $\mathcal{C}$). Remark that $y$ may not belong to $\mathcal{C}$ but necessarily $g_y([0,t_n \wedge \tau_{\dyn}^{\infty}(y)])$ overlaps $\mathcal{C}$. The dynamical algorithm applies to the (alive) grains of $\mathcal{C}$. A very similar proof to that of the Reconciliation Lemma implies that $t_x = \tau_{\dyn}^{\infty}(x)$, i.e. the Stopping property holds. The same arguments work for the Hitting property.

In order to prove that there is uniqueness of the $\overline \xi$-lifetime function, we proceed as in the proof of the Reconciliation Lemma by absurd. Assume that there exist two $\overline \xi$-lifetime function, say $\tau_1$ and $\tau_2$. Let $[t_n,t_{n+1})$ be the first interval on which a discrepancy between $\tau_1$ and $\tau_2$ occurs (the growths of all the grains until $t_n$ coincide for $\tau_1$ and $\tau_2$). Hence, consider a cluster $\mathcal{C}$ of $\Sigma_n'$ containing such discrepancy. Since $\mathcal{C}$ is a.s. finite, we can select a grain $x$ minimizing $\tau_1(x)\wedge \tau_2(x)$ among the grains of
\[
\{ x' \in \mathcal{C} : \, \tau_1(x') \not= \tau_2(x') \; \mbox{ and } \; t_n \leq \tau_1(x')\wedge \tau_2(x') < t_{n+1} \} ~.
\]
By symmetry, let us assume that $\tau_1(x) < \tau_2(x)$ which implies that $\tau_1(x) < \infty$. By the Hitting property (for $\tau_1$), there exists a grain $y\not= x$ such that $x$ $\tau_1$-hits $y$ at time $\tau_1(x)$. As in the proof of the Reconciliation Lemma, we can prove that $\tau_2(y) < t_y \leq \tau_1(y)$ and then deduce that $y$ satisfies $y \in \mathcal{C}$, $\tau_1(y) \not= \tau_2(y)$ and
\[
t_n \leq \tau_1(y) \wedge \tau_2(y) < \tau_1(x) \wedge \tau_2(x) < t_{n+1} ~. 
\] 
contradicting the minimality of $x$. This achieves the proof of Proposition \ref{p:existence}.

\begin{proof}[Proof of Lemma \ref{lem:subdivision}]
Let $t,t' \geq 0$. Let us first prove that
\begin{equation}
\label{LimitMom2}
\lim_{t' \to t} \E \left[ \sup_{s \in [t^-,t^+]} \| h(t) - h(s) \|^2 \right] = 0
\end{equation}
where $t^- := t \wedge t'$ and $t^+ := t \vee t'$. By continuity of the random variable $h$, $\sup_{s \in [t^-,t^+]} \| h(t) - h(s) \|^2$ a.s. tends to $0$ as $t' \to t$. It then remains to check the domination hypothesis to apply the Lebesgue's dominated convergence theorem and conclude that (\ref{LimitMom2}) holds. Restricting our attention to $|t'-t| \leq 1$,
\[
f(t) := \sup_{s \in [t-1,t+1]} \| h(t) - h(s) \|^2
\]
will play the role of the dominating function. It is not difficult to see that $f$ is integrable writting
\[
f(t) \leq \| h(t) \|^2 + \sup_{s \in [t-1,t+1]} \| h(s) \|^2 + 2 \| h(t) \| \sup_{s \in [t-1,t+1]} \| h(s) \| ~,
\]
and using the Cauchy–Schwarz inequality and the moment condition (\ref{SuffTempered}).

Now let us define by induction a non-decreasing sequence $(t_n)_{n\geq 0}$ as follows: $t_0 = 0$ and for any integer $n$, given $t_n$, we set
\[
t_{n+1} := \sup \left\{ t \geq t_n : \, \E \left[ \sup_{s \in [t_n,t]} \| h(s) - h(t_n) \|^2 \right] \leq \alpha \right\} \in [t_n , \infty]
\]
with convention that $t_n = \infty$ implies $t_{n+1} = \infty$. Let us assume for the moment that
\begin{equation}
\label{LimitFinite-tn}
\sup_{n\geq 0} t_n = \infty ~. 
\end{equation}
Either all the $t_n$'s are finite and by the monotone convergence theorem, we have
\[
\E \left[ \sup_{s \in [t_n,t_{n+1}]} \| h(s) - h(t_n) \|^2 \right] \leq \alpha
\]
leading to Lemma \ref{lem:subdivision}. Or the sequence $(t_n)_n$ is infinite from some index $k \geq 1$ and, in this case, the sequence $(s_n)_{n\geq 0}$ defined by $s_n = t_n$ for any $n \leq k-1$ and $s_n = s_{n-1}+1$ for any $n \geq k$, will work.

It then remains to prove (\ref{LimitFinite-tn}). Let us proceed by absurd with assuming that $t^\ast :=  \sup_n t_n$ is finite. Each time $t_n$ is then finite and smaller than $t^\ast$. Thus (\ref{LimitMom2}) gives the existence of $\varepsilon > 0$ such that
\begin{equation}
\label{Limit-t^ast}
\E \left[ \sup_{s \in [t^\ast-\varepsilon,t^\ast+\varepsilon]} \| h(s) - h(t^\ast) \|^2 \right] \leq \frac{\alpha}{4} ~.
\end{equation}
Let us pick $t_n \in [t^\ast-\varepsilon,t^\ast]$ and set $t := t^\ast+\varepsilon$. Proceeding as before with the Cauchy–Schwarz inequality and (\ref{Limit-t^ast}), we get
\[
\E \left[ \sup_{s \in [t_n,t]} \| h(s) - h(t_n) \|^2 \right] \leq \alpha
\]
proving that, by definition of $t_{n+1}$, $t^\ast \geq t_{n+1} \geq t = t^\ast + \varepsilon$. This is absurd and proves (\ref{LimitFinite-tn}).
\end{proof}

Let us point out that w.r.t. Theorem 3.1 of \cite{CDlS-unbounded}, two hypotheses have been deleted : 1. the Lebesgue measure of $(h-h')(\R_+)$ equals $0$ a.s. and 2. the Lebesgue measure of $h(\R_+)$ equals $0$ a.s., where $h,h'$ are two independent r.v.'s with distribution $\mu$. The first one was required to prevent collisions of grains at the same time and at the same location. Actually such collisions are allowed by the dynamical algorithm and thence they are absolutely not an obstacle to the existence of the model. The second hypothesis concerned the multi-branching case (i.e. various grain emanating from the same germ) which is not considered here.

\subsection{Stability of tempered configurationsby resampling}

\label{s:des_trucs_de_mesurabilite}

In this section, we provide a technical justification of an item in the definition of "good point" given in Section \ref{SectionGP}. Precisely we have the following lemma which claims that after resampling the point process in a ball  then almost surely the configuration remains tempered.     

\begin{lemma} \label{l:good-definition-fondee}
\begin{enumerate}
\item With probability one, $\overline\xi^0 \in \overline\cA$.
\item With probability one, $\overline\xi^0$ belongs to $\{\overline S \in \overline\cS^0 : \P[\text{resample}(\overline S, \overline \xi') \in \overline\cA]=1\}$.
\item With probability one, for all $x \in \xi$, $\overline\xi-x$ belongs to the set 
$\{\overline S \in \overline\cS^0 : \P[\text{resample}(\overline S, \overline \xi') \in \overline\cA]=1\}$.
\end{enumerate}
\end{lemma}

\begin{proof} Let us prove the first item.
Fix $C = [0,1]^d$. Then
\[
\P[\overline\xi^0 \in \overline\cA] 
  = \E\left[\sum_{x \in \xi \cap C} \1_{\overline\xi-x \in \overline\cA}\right] 
  = \E\left[\sum_{x \in \xi \cap C} \1_{\overline\xi \in \overline\cA}\right] 
  = \E\left[\sum_{x \in \xi \cap C} \1\right] 
  = 1.
\]
The second equality is due to the invariance of $\overline\cA$ under translations. 
The third equality holds because $\overline\xi$ belongs almost surely to $\overline\cA$.

Let us consider the second item.
The point processes $\overline\xi^0$ and $\text{resample}(\overline\xi^0,\overline\xi')$ have the same distribution.
By the first item, we then know that $\text{resample}(\overline\xi^0,\overline\xi')$ belongs almost surely to $\overline\cA$.
The second item follows by Fubini theorem.

Set $\overline\cR^0 = \{\overline S \in \overline\cS^0 : \P[\text{resample}(\overline S, \overline \xi') \in \overline\cA]=1\}$.
We have just shown $\P[\overline\xi^0 \in \overline\cR^0]=1$.
As a consequence, 
\[
\E\left[ \sum_{x \in \xi \cap C} \1_{\overline\xi-x \in \overline\cR^0}\right] = \P[\overline\xi^0 \in \overline\cR^0]=1.
\]
Therefore with probability one, for any $x \in \xi \cap C$, we have $\overline\xi-x \in \overline\cR^0$.
The third item follow by stationarity.
\end{proof}

\bibliographystyle{plain}
\bibliography{biblio.bib}

\end{document}